\documentclass[12pt,cleveref,capitalize,nameinlink,final]{colt2026} 
\usepackage{amsmath}
\usepackage{amssymb}
\usepackage{forloop}
\usepackage{nicefrac}
\usepackage{mathtools}
\usepackage[table,dvipsnames]{xcolor}
\usepackage{url}
\usepackage{bm}
\usepackage{float}
\usepackage{natbib}

\usepackage[normalem]{ulem}
\usepackage{amsfonts,mathtools,mathrsfs,color}
\usepackage{caption}
\usepackage{subcaption}
\usepackage{bm}
\usepackage{enumerate}
\usepackage{soul}
\usepackage{tcolorbox}
\usepackage{framed}
\usepackage{fontawesome5}

\usepackage{enumitem}
\usepackage{cancel}
\usepackage{framed}
\usepackage{booktabs}
\usepackage{makecell}
\usepackage{hyperref}
\usepackage{multirow}
\usepackage{wrapfig}

\newenvironment{prfsketch}{%
  \proof}{\endproof}
\newenvironment{proofof}[1]{\begin{proof}\textbf{of {#1}}}{\end{proof}}

\newcommand{\defcal} [1]{\expandafter\newcommand\csname cal#1\endcsname{{\mathcal #1}}}
\newcommand{\defsf} [1]{\expandafter\newcommand\csname sf#1\endcsname{{\mathsf #1}}}
\newcommand{\defbf} [1]{\expandafter\newcommand\csname bf#1\endcsname{{\mathbf #1}}}
\newcommand{\defbb} [1]{\expandafter\newcommand\csname bb#1\endcsname{{\mathbb{#1}}}}
\newcommand{\deffrak} [1]{\expandafter\newcommand\csname frak#1\endcsname{{\mathfrak{#1}}}}

\newcounter{ct}
\forLoop{1}{26}{ct}{
    \edef\letter{\Alph{ct}}
    \expandafter\defcal\letter
    \expandafter\defsf\letter
    \expandafter\defbf\letter
    \expandafter\defbb\letter
    \expandafter\deffrak\letter
}
\forLoop{1}{26}{ct}{
    \edef\letter{\alph{ct}}
    \expandafter\defcal\letter
    \expandafter\defsf\letter
    \expandafter\defbf\letter
    \expandafter\defbb\letter
    \expandafter\deffrak\letter
}

\newcommand{\rmd}{\mathrm{d}}
\newcommand{\rmI}{\mathrm{I}}

\newcommand{\numberthis}{\addtocounter{equation}{1}\tag{\theequation}}
\newcommand{\xstar}{x^{\star}}

\newcommand{\subscript}[2]{$#1 _ #2$}
\newlist{assumplist}{enumerate}{1}
\setlist[assumplist]{label=\subscript{\textbf{\textsf{A}}}{\textsf{{\arabic*}}},leftmargin=*, itemsep=0pt}

\newtheorem*{lemma*}{Lemma}
\newtheorem*{corollary*}{Corollary}
\newtheorem*{assumption*}{Assumption}

\makeatletter
\newcommand*{\algotitle}[2]{%
  \stepcounter{algocf}%
  \hypertarget{algocf.title.\theHalgocf}{}%
  \NR@gettitle{#1}%
  \label{#2}%
  \addtocounter{algocf}{-1}%
}
\makeatother

\usepackage{algorithm}
\usepackage{algpseudocode}

\algrenewcommand\algorithmicrequire{\textbf{Input:}}
\algrenewcommand\algorithmicensure{\textbf{Output:}}

\crefname{algocf}{alg.}{algs.}
\Crefname{algocf}{Algorithm}{Algorithms}

\newcommand{\R}{\mathbb{R}}

\newcommand{\lp}{\left(}
\newcommand{\rp}{\right)}

\usepackage{booktabs}
\usepackage{array}
\newcolumntype{Y}{>{\raggedright\arraybackslash}X}

\newcommand{\HF}{\mathsf{HF}}

\newcommand{\bz}{\mathbf{0}}

\usepackage{minitoc}

\makeatletter
\@ifundefined{lemma}{}{ }
\@ifundefined{corollary}{}{ }
\@ifundefined{theorem}{}{ }
\@ifundefined{remark}{}{ }
\@ifundefined{definition}{}{ }
\makeatother

\newtheorem{theorem}{Theorem} 
\newtheorem{lemma}{Lemma}
\newtheorem{corollary}{Corollary}
\newtheorem{definition}{Definition}
\newtheorem{remark}{Remark}

\title[Accelerated Convex Optimization via Hamiltonian Dynamics]{Accelerated Convex Optimization via\\
Hamiltonian Dynamics with Deterministic Integration Time
}
\usepackage{times}
\allowdisplaybreaks

\coltauthor{%
 \Name{Xiuyuan Wang$^*$} \Email{xiuyuan.wang@yale.edu}\\
 \addr Department of Computer Science, Yale University 
 \AND
 \Name{Vishwak Srinivasan$^*$} \Email{vishwaks@mit.edu} \\
 \addr Department of EECS, Massachusetts Institute of Technology 
 \AND
 \Name{Qiang Fu} \Email{qiang.fu@yale.edu} \\
 \addr Department of Computer Science, Yale University 
 \AND
 \Name{Siddharth Mitra} \Email{siddharth.mitra@yale.edu} \\
 \addr Department of Computer Science, Yale University 
 \AND
 \Name{Ashia Wilson} \Email{ashia@mit.edu} \\
 \addr Department of EECS, Massachusetts Institute of Technology 
 \AND
 \Name{Andre Wibisono} \Email{andre.wibisono@yale.edu} \\
 \addr Department of Computer Science, Yale University
}

\newcommand{\CONTALG}{\textsf{HFA}} 
\newcommand{\DISCALG}{\textsf{dHFA-eg}} 
\newcommand{\PDISCALG}{\textsf{dHFA-im}} 

\newcommand{\HMC}{\textsf{HMC}}

\newcommand{\convexassump}{\hyperref[assump:cvx]{ \textsf{A}\textsubscript{\textsf{cvx}}}}
\newcommand{\smoothassump}{\hyperref[assump:smth]{ \textsf{A}\textsubscript{\(L\)-\textsf{sm}}}}
\newcommand{\quadgrowassump}{\hyperref[assump:qg]{ \textsf{A}\textsubscript{\(\alpha\)-\textsf{qg}}}}
\newcommand{\strongconvexassump}{\hyperref[assump:sc]{ \textsf{A}\textsubscript{\(\alpha\)-\textsf{sc}}}}
\begin{document}

\maketitle

\vspace*{-1.75mm}
\begin{abstract}
\noindent We develop Hamiltonian dynamics-based algorithms for smooth convex optimization that achieve accelerated rates of convergence. By exploiting contraction of averaged Hamiltonian flow trajectories rather than requiring contraction at trajectory endpoints, we show that Hamiltonian dynamics-based optimization methods admit deterministic and accelerated convergence guarantees, extending prior work that is limited to quadratic objectives or holds only in expectation. We analyze an idealized continuous-time algorithm and derive practical discrete-time implementations with optimal first-order complexity, thereby establishing Hamiltonian dynamics as a useful algorithmic primitive for deterministic accelerated convex optimization.
\end{abstract}

\def\thefootnote{$^*$}\footnotetext{Both authors contributed equally.}

\doparttoc
\faketableofcontents

\part{}

\vspace*{-10mm}
\section{Introduction}
\label{sec:intro}

Algorithms to optimize a function are ubiquitous in science and engineering, and they have attracted increased attention in recent years especially in machine learning, computer science, and statistics owing to the need to solve large-scale high-dimensional computational problems.
Recent algorithmic designs and analyses have drawn insights from continuous-time dynamical systems, both to obtain a better understanding of existing optimization algorithms and their convergence analyses, as well as to design new optimization algorithms with rigorous theoretical guarantees.

In this work, we study the optimization problem
\begin{equation*}
    \min_{x \in \R^d} ~~ f(x)
\end{equation*}
where \(f \colon \bbR^{d} \to \bbR\) is a differentiable convex function, and we assume access to the gradient oracle $\nabla f \colon \R^d \to \R^d$.
This is classically referred to as \textit{first-order convex optimization}.
We propose a novel accelerated optimization algorithm based on \emph{Hamiltonian dynamics}, which has been highly successful in the closely related algorithmic task of sampling.

The simplest gradient-based method for minimizing \(f\) is the gradient descent algorithm.
Gradient descent can be viewed as the forward-Euler discretization of the \textit{gradient flow} dynamics, which is a simple continuous-time greedy dynamics that converges to a minimizer of \(f\) when \(f\) is convex.
\begin{equation*}
    \underbrace{\dot{X}_{t} = -\nabla f(X_{t})}_{\text{Gradient flow of }f} \enskip \xrightarrow{\text{forward Euler}} \enskip \underbrace{x_{k + 1} = x_{k} - \eta_{k} \, \nabla f(x_{k})}_{\text{Gradient descent of }f}~.
\end{equation*}
The choice of step size \(\eta_{k}\) determines the rate of convergence of gradient descent to a minimizer of \(f\), which is often chosen based on properties of \(f\), or via a backtracking line search scheme.
When \(f\) is known to be \(L\)-smooth (see \smoothassump{}), a recommended choice for \(\eta_{k}\) is \(\nicefrac{1}{L}\).
In the table below, we present the scalings of time \(T\) for the continuous-time dynamics, and the number of the iterations \(K\) for the discrete-time algorithm, to return a point \(\hat{x}\) such that \(f(\hat{x}) - f^{\star} \leq \varepsilon\) in the settings when $f$ is either convex or strongly convex, in addition to being $L$-smooth.

\begin{center}
\renewcommand{\arraystretch}{1.25}
\centering
\begin{tabular}{c|c|c}
    & Gradient Flow & Gradient Descent \\
    [1mm]
    \hline
    \(f\) is convex & \(1/\varepsilon\) & \(L/\varepsilon\) \\
    \(f\) is \(\alpha\)-strongly convex  & \((1/\alpha)\log\left(1/\varepsilon\right)\) & \( (L/\alpha) \log\left(1/\varepsilon\right)\)
\end{tabular}
\end{center}

A natural question is if the complexity achieved by gradient flow and gradient descent can be improved.
The seminal work of \citet{nesterov1983method} develops a discrete-time accelerated gradient descent algorithm whose complexity are quantitatively better than gradient descent for smooth convex optimization.
Notably, this method is \emph{optimal} in the sense that it achieves the best possible complexity for this problem in discrete time \citep[\S 2.1]{nesterov2018lectures}.
However, both the method and its analysis was rather delicate, and there have been many alternative constructions and analyses of accelerated algorithms \citep{drori2014performance,bubeck2015geometric,allen2017linear,drusvyatskiy2018optimal, wilson2019accelerating}.

A dynamical perspective of \citeauthor{nesterov1983method}'s method was presented in \citet{su2016differential}, where they demonstrate accelerated gradient descent is a discretization of a second-order continuous-time dynamics:
\begin{equation*}
    \underbrace{\ddot{X}_{t} + \beta_{t} \dot{X}_{t} + \nabla f(X_{t}) = 0}_{\text{Accelerated Gradient Flow of }f} \enskip \xrightarrow{\text{discretize}}\enskip \text{Accelerated Gradient Descent of }f~.
\end{equation*}
\citet{su2016differential} shows the damping factor is \(\beta_{t} = \nicefrac{3}{t}\) for convex $f$, and shows accelerated gradient flow has a matching convergence rate as accelerated gradient descent.
For \(\alpha\)-strongly convex \(f\), \citet{wilson2021lyapunov} show that accelerated gradient flow with \(\beta_{t} = 2\sqrt{\alpha}\) converges faster than the gradient flow, and that a discretization of these dynamics is equivalent to an accelerated method for strongly convex function \citep[\S 2.2.2]{nesterov2018lectures}.
See table below for a summary of the rates.
\begin{center}
\renewcommand{\arraystretch}{1.25}
\centering
\begin{tabular}{c|c|c}
    & Accelerated Gradient Flow & Accelerated Gradient Descent \\
    [1mm]
    \hline
    \(f\) is convex & \(1/\sqrt{\varepsilon}\) & \(\sqrt{L / \varepsilon}\) \\
    \(f\) is \(\alpha\)-strongly convex  & \((1/\sqrt{\alpha}) \, \log(1 / \varepsilon)\) & \(\sqrt{L/\alpha} \, \log(1 / \varepsilon)\)
\end{tabular}
\refstepcounter{table}\label{tab:rates-agf-agd}
\end{center}

While the accelerated gradient method continues to be studied past these developments, in recent years, there have been new approaches to improve the complexity of gradient descent via a more principled choice of stepsizes \(\{\eta_{k}\}_{k \geq 0}\); see \citet{grimmer2024provably} for an overview of recent work.
Among these approaches, the best known rates for minimizing $L$-smooth convex and \(\alpha\)-strongly convex \(f\) scale as \(L/k^{\theta}\) and \((1 - ({\alpha}/{L})^{\nicefrac{1}{\theta}})^{k}\), respectively, where $\theta=\log_2(1+\sqrt{2})\approx 1.2716$ \citep{altschuler2025accelerationa,altschuler2025accelerationb}.
These approaches are noteworthy in how they deviate from classical gradient descent analyses which advocated for small stepsizes to enforce a descent method.

In this vein, our work deviates from classical dynamics and algorithm design for methods that achieve accelerated rates of convergence for smooth (strongly) convex optimization by adopting ideas from sampling while retaining a dynamical perspective.
We specifically take inspiration from the \textit{Hamiltonian Monte Carlo} (\HMC{}) \citep{duane1987hybrid,neal2011mcmc}, which is an algorithmic framework for sampling based on the Hamiltonian flow (\ref{eq:HamFlow}):
\begin{equation}
\dot{X}_{t} = Y_t~,
\qquad
\dot{Y}_{t} = -\nabla f(X_t)~.\tag{\textsf{HF}}
\end{equation}
\HMC{}-based samplers are widely used in practice; a popular example is the No-U-Turn Sampler (NUTS)~\citep{hoffman2014no}, which underlies modern probabilistic programming languages such as Stan~\citep{carpenter2017stan}, and is regarded as a leading approach for high-dimensional Bayesian inference.
In recent years, the perspective of sampling as performing optimization in the space of measures~\citep{wibisono2018sampling} have led to better analyses of sampling algorithms (e.g., \citet{dalalyan2017further,dalalyan2017theoretical,durmus2019analysis}) and the development of sampling algorithms that translate ideas from optimization (e.g., \citet{salim2019stochastic,zhang2020wasserstein,ma2021there,srinivasan2025highaccuracy,chen2025accelerating}).
In contrast, our work takes the opposite perspective, where we try to translate ideas from sampling to develop better optimization algorithms.

The Hamiltonian dynamics-based algorithm we study for optimization augments the position space $\R^d$ we want to optimize over with a velocity space $\R^d$, to get  the \emph{phase space} $\R^{2d}$ over which the Hamiltonian dynamics is defined (see \cref{sec:ham-dynamics-def} for a more detailed overview).
The most basic Hamiltonian dynamics-based optimization algorithms -- which we call \ref{eq:HF-opt-update} -- performs the following iteratively: from current iterate \(x_{\text{in}}\), do
\begin{equation*}
\hspace*{-0.03cm}
\left.\begin{aligned}
    & \text{Step (i)} \quad ~~\text{Lift to phase space as } (X_0, Y_0) = (x_{\text{in}}, \bm{0}) \in \bbR^{2d}~. \\
    & \text{Step (ii)} \quad ~\text{Simulate (\ref{eq:HamFlow}) } \text{ for \emph{integration time} } T > 0 \text{ to reach } (X_T,Y_T)~. ~~~ \\
    & \text{Step (iii)} \quad \text{Output } x_{\text{out}} := X_{T}~\text{ as the next iterate.}
\end{aligned}
\!\right\}
\label{eq:HF-opt-update}\tag{\textsf{HFopt}}
\end{equation*}
This scheme results in a non-increasing sequence of function values, which is due to the conservation property of the Hamiltonian dynamics.
To the best of our knowledge, this was first proposed in \citet{teel2019first}, and they show a non-accelerated rate of convergence.
\citet{wang2025frictionlesshamiltoniandescentcoordinate} studies the specific setting of minimizing strongly convex quadratics, and show that a careful choice of integration times can result in an accelerated rate of convergence.
Subsequent work by \citet{fu2025hamiltoniandescentalgorithmsoptimization} advance this frontier by showing that \ref{eq:HF-opt-update} with \emph{randomized} integration times achieve accelerated convergence rates for general smooth convex and strongly convex objective functions, albeit the guarantees being in expectation over the randomness of the integration time.
See \Cref{app:AdditionalRelatedWorks} for additional discussion on related works.

\subsection{Our contributions}
A key insight we develop in this work is in how time averages of the solutions of \ref{eq:HamFlow} behave when \(f\) is convex; this is formalized as~\Cref{lem:hf-convex-main} and~\Cref{cor:hf-convex-main}.
Essentially, we discover that an average of the positions \((X_{t})_{t \in [0, T]}\) along~\ref{eq:HamFlow} over a sufficiently long integration time $T$ is more conducive to the goal of optimizing \(f\) rather than simply the endpoint \(X_{T}\) used in \ref{eq:HF-opt-update}; we provide an example to illustrate this in \Cref{app:sec:example-hf-quad}.
Specifically, when \(f\) satisfies the \(\alpha\)-quadratic growth condition (which is implied by \(\alpha\)-strong convexity), we find a good choice of the integration time scales as \(\nicefrac{1}{\sqrt{\alpha}}\).
Notably, this scaling coincides with the expected integration time considered by \citet{fu2025hamiltoniandescentalgorithmsoptimization} in their randomized version of \ref{eq:HF-opt-update} in the strongly convex setting.
This can also be contrasted to short integration times for which non-accelerated guarantees for \ref{eq:HF-opt-update} have been shown, see \citep[Theorem~1]{fu2025hamiltoniandescentalgorithmsoptimization}.
This insight forms the basis of our proposed idealized algorithm ``Hamiltonian Flow for optimization with Averaging '' (\nameref{alg:phfopt}), shown as Algorithm \ref{alg:phfopt}.
(Here ``idealized'' means we assume we can simulate the continuous-time Hamiltonian flow \ref{eq:HamFlow} exactly.)
For minimizing convex and strongly convex \(f\), we show in \Cref{sec:hf_convexity} that \CONTALG{} achieves accelerated convergence rates, matching those of the accelerated gradient flow;
see \Cref{thm:hf-cts-cvx} for the convex case, and \Cref{thm:hf-cts-strong-cvx} for the strongly convex case.

\begin{algorithm}[t]
\DontPrintSemicolon

\caption{Hamiltonian Flow for optimization with Averaging (\CONTALG{})}
\algotitle{\CONTALG{}}{alg:phfopt}
\SetKwInOut{Input}{Input}\SetKwInOut{Output}{Output}
\SetAlgoLined
\Input{Initialization: $X_0\in \R^d$; \,  number of iterations: $K \in \mathbb{N}$; \, parameter: $\lambda \geq 0$; \, integration time sequence: $\{T_1, \dots, T_K\}$.}
\For{\(k = 1\) \KwTo \(K\)}{
    Solve \ref{eq:HamFlow} from \((X_{0}^{(k)}, Y_{0}^{(k)}) = (X_{k - 1}, \bm{0})\) for time \(T_{k}\) to obtain trajectory \(((X_{t}^{(k)}, Y_{t}^{(k)}))_{t \in [0, T_{k}]}\).
    
    Compute \(\displaystyle X^{\mathrm{avg}}(X_{k - 1}; T_{k}) = \frac{2}{T_{k}^2} \int_0^{T_{k}} (T_{k}-t) X_{t}^{(k)}\rmd t\)\,.

     Set $\displaystyle X_{k} = \frac{1}{\lambda+1}X^{\mathrm{avg}}(X_{k - 1}; T_{k}) + \frac{\lambda}{\lambda+1}X_{T_{k}}^{(k)}$\,.
    }
\Return $X_K$
\end{algorithm}

While the Hamiltonian flow \ref{eq:HamFlow} provide valuable algorithmic guidance, these dynamics admit exact solutions only in special cases.
Therefore, to implement the \nameref{alg:phfopt} as a concrete discrete-time algorithm, we discretize \ref{eq:HamFlow} via a numerical integrator, and replace the continuous-time average with a discrete average of the iterates.
In particular, we discretize the Hamiltonian flow using the \emph{extragradient integrator} (which approximates the implicit integrator, see~\Cref{sec:ham-dynamics-def} for a review).
Our discrete-time version of \nameref{alg:phfopt} is presented in \Cref{sec:discrete-algos}. 
To the best of our knowledge, our algorithm differs from previous strategies for designing accelerated gradient methods based on averaging \citep{d2021acceleration}.
We show that the iteration complexity of this algorithm to find an \(\varepsilon\)-approximate solution has the same scaling as the iteration complexity of Nesterov's accelerated gradient method for smooth convex and strongly convex optimization presented in \cref{tab:rates-agf-agd}, see \Cref{thm:hf-extg-cvx} for the convex case and \Cref{thm:hf-extg-strong-cvx} for the strongly convex case.
Furthermore, we emphasize the rates of our Hamiltonian-based algorithms are \emph{deterministic}, in contrast those in the previous work of~\citet{fu2025hamiltoniandescentalgorithmsoptimization} which only hold in expectation.
\section{Preliminaries}

\paragraph{Notation.} Throughout the paper, we consider minimizing a differentiable function $f \colon \R^d \to \R$.
The minimum of \(f\) is assumed to exist uniquely, which we denote with $\xstar$.
For vectors \(v, w\), the Euclidean inner product is \(\langle v, w\rangle\), and \(\|v\| = \sqrt{\langle v, v\rangle}\) denotes the Euclidean norm of \(v\).
For a time-dependent quantity \(g_{t}\), we use the shorthand \(\dot{g}_{t} = \frac{\rmd}{\rmd t} g_{t}\) when the derivative exists.
For a real number \(a\), the ceiling \(\lceil a \rceil\) denotes the smallest integer that is greater than or equal to \(a\).
We say that \(\hat{x}\) is an \textit{\(\varepsilon\)-accurate minimizer} of \(f\) if \(f(\hat{x}) - f(x^{\star}) \leq \varepsilon\)~.

\paragraph{Assumptions.}
We begin by introducing structural properties of \(f\).

\newcounter{desccount}
\newcommand{\descitem}[2]{%
\item[\textsf{A}\textsubscript{\textsf{#1}}]\refstepcounter{desccount}\label{assump:#1}
}

\begin{description}
    \item [\textsf{A}\textsubscript{\textsf{cvx}}]\refstepcounter{desccount}\label{assump:cvx} \hskip6.7pt (\textbf{Convexity: }) for any \(x, x' \in \bbR^{d}\), \(f(x') \geq f(x) + \langle \nabla f(x), x' - x\rangle\).
    \item [\textsf{A}\textsubscript{\(L\)-\textsf{sm}}]\refstepcounter{desccount}\label{assump:smth} (\textbf{\(L\)-Smoothness: }) for any \(x, x' \in \bbR^{d}\), \(f(x') \leq f(x) + \langle \nabla f(x), x' - x\rangle + \frac{L}{2}\|x' - x\|^{2}\).
    \item [\textsf{A}\textsubscript{\(\alpha\)-\textsf{qg}}]\refstepcounter{desccount} \label{assump:qg}  \hskip2.1pt (\textbf{\(\alpha\)-Quadratic Growth: }) for any \(x \in \bbR^{d}\), \(f(x) \geq f(x^{\star}) + \frac{\alpha}{2}\|x - x^{\star}\|^{2}\).
    \item [\textsf{A}\textsubscript{\(\alpha\)-\textsf{sc}}]\refstepcounter{desccount}\label{assump:sc} \hskip3pt (\textbf{\(\alpha\)-Strong Convexity: }) for any \(x, x' \in \bbR^{d}\), \(f(x') \geq f(x) + \langle \nabla f(x), x' - x\rangle+ \frac{\alpha}{2}\|x - x'\|^{2}\).
\end{description}

For the key property of the Hamiltonian dynamics that we derive in \cref{sec:hf_convexity}, it is sufficient for \convexassump{} and \smoothassump{} to hold.
Note that \strongconvexassump{} implies \quadgrowassump{} and \convexassump{}, and our algorithmic results only require \convexassump{} and \quadgrowassump{}.
The assumptions
\convexassump{} and \quadgrowassump{} imply for any \(x \in \bbR^{d}\), \(\langle \nabla f(x), x - x^{\star}\rangle \geq \frac{\alpha}{2}\|x - x^{\star}\|^{2}\);
this is the \emph{restricted strong convexity} condition \citep[Definition 3]{zhang2013gradient}, which is weaker than strong convexity, and implies \(f\) satisfies the Polyak-\L{}ojasiewicz inequality, see \citet[Appendix~A]{karimi2016linear}.
When \(f\) satisfies \smoothassump{} and \quadgrowassump{}, we use \(\kappa := \nicefrac{L}{\alpha}\) to denote the \emph{condition number} of \(f\).

\subsection{Hamiltonian Dynamics}
\label{sec:ham-dynamics-def}

We review the Hamiltonian dynamics, which is the key ingredient of the algorithms in this paper.
Recall we seek to minimize an objective function $f \colon \R^d \to \R$.
We define the \emph{energy function} or the \textit{Hamiltonian} \(H \colon \R^d \times \R^d \to \R\) by 
$$H(x, y) = f(x) +\frac{1}{2}\|y\|^{2} \,.$$
In classical mechanics, $x$ denotes position and $y$ denotes momentum (or velocity), and the Hamiltonian can be interpreted as the sum of the potential energy $f(x)$ and the kinetic energy $\frac{1}{2}\|y\|^2$ of a unit-mass particle.
The Hamiltonian dynamics or Hamiltonian flow is a system of differential equations that describes the evolution of the position and velocity $(X_t, Y_t) \in \R^{2d}$ over time $t \ge 0$ via
\begin{equation}\label{eq:HamFlow}
\dot{X}_{t} = \nabla_y H(X_t,Y_t) = Y_t~,
\qquad
\dot{Y}_{t} = -\nabla_x H(X_t,Y_t) = -\nabla f(X_t)~.\tag{\textsf{HF}}
\end{equation}

A key property of the Hamiltonian dynamics is that the energy function of the system is conserved for all time \(t \geq 0\); this can be shown by direct computation (see \cref{app:HF_Background} for a review).

\begin{lemma}
\label{lem:hamflow-conserve}
For any \(t \geq 0\) and initial conditions \((X_{0}, Y_{0})\), \(f(X_{t}) + \frac{1}{2}\|Y_{t}\|^{2} = f(X_{0}) + \frac{1}{2}\|Y_{0}\|^{2}\)~.
\end{lemma}

The Hamiltonian flow map \(\mathsf{HF} : \bbR^{2d} \times [0, \infty) \to \bbR^{2d}\) maps a tuple \((x, y)\) and time \(t\) to the solution of \ref{eq:HamFlow} at time \(t\) from initial condition \((X_{0}, Y_{0}) = (x, y)\) i.e., 
$$\mathsf{HF}((x, y); t) = (X_{t}, Y_{t})\,.$$
As discussed in \Cref{sec:intro}, the Hamiltonian flow map cannot be computed explicitly for a general potential \(f\).
A setting in which is this map is available is when \(f\) is quadratic, which is leveraged by \citet{wang2025frictionlesshamiltoniandescentcoordinate} to develop accelerated algorithms for minimizing strongly convex quadratic functions.

\paragraph{Numerical integrators.}
We discuss two numerical integrators to approximate \(\mathsf{HF}((x, y); t)\): the implicit integrator and the extragradient integrator, both of which depend on a step size \(\eta > 0\).
Below, $n \in \mathbb{N}_0$ denotes the discrete-time index.

\begin{definition}\label{defn:imp}
The \emph{implicit integrator} of Hamiltonian flow \emph{(\ref{eq:HamFlow})} from initial values \((x_{0}, y_{0})\) with step size \(\eta > 0\) generates the sequence \(\{(x_{n}, y_{n})\}_{n \geq 0}\) where \((x_{n}, y_{n})\) satisfies the recursion
\begin{equation}
\tag{\textsf{HF\textsuperscript{imp}}}
\label{eq:implicit-integ}
\begin{aligned}
    x_{n+1} &= x_n + \eta y_{n+1}~; \\
    y_{n+1}  &= y_n-\eta \nabla f(x_{n+1})~.
\end{aligned}
\end{equation}
\end{definition}
The implicit integrator is named as such since given the current iterate \((x_{n}, y_{n})\), the next iterate \((x_{n + 1}, y_{n + 1})\) is obtained by solving an implicit update. We choose the implicit integrator as a starting point due to: (a) its dissipative nature for convex Hamiltonian, which is sufficient for a descent scheme; and (b) its numerical stability, which does not require smoothness of the Hamiltonian. The implicit integrator can be implemented, for example, using a proximal oracle for \(f\) in addition to a gradient oracle as stated below (see \cref{app:sec:integrator-discuss}).
\begin{align*}
    x_{n + 1} &= \mathrm{Prox}_{\eta^{2}f}(x_{n} + \eta y_{n})~;\\
    \quad y_{n + 1} &= y_{n} - \eta \nabla f(x_{n + 1})~.
\end{align*}
However, a proximal oracle is stronger than having access to gradients of \(f\).
To circumvent this, \citet{fu2025hamiltoniandescentalgorithmsoptimization} propose the \emph{extragradient integrator}, which approximates a step of the implicit integrator.
This specifically operates by replacing the call to the proximal oracle for \(f\) by a gradient descent step, defined below.
\begin{definition}\label{defn:extg}
The \emph{extragradient integrator} of Hamiltonian flow \emph{(\ref{eq:HamFlow})} from initial values \((x_{0}, y_{0})\) with step size \(\eta > 0\) generates the sequence \(\{(x_{n}, y_{n})\}_{n \geq 0}\) where \((x_{n}, y_{n})\) satisfies the recursion
\begin{equation}
\tag{\textsf{HF\textsuperscript{extg}}}
\label{eqn:extg-update}
\begin{aligned}
    x_{n + \nicefrac{1}{2}} &= x_{n} + \eta \, y_{n}~; \\
    x_{n + 1} &= x_{n + \nicefrac{1}{2}} - \eta^{2} \, \nabla f(x_{n + \nicefrac{1}{2}})~;\\
    y_{n + 1} &= y_{n} - \eta \, \nabla f(x_{n + 1})~.
\end{aligned}
\end{equation}
\end{definition}
Since we are interested in the first-order methods for minimizing \(f\), we focus on the extragradient in the main text, and discuss further details of the integrators in \cref{app:sec:integrator-discuss}.
For either integrator, the output after \(n\) discrete iterations with step size \(\eta\) can be viewed as approximating the solution of \ref{eq:HamFlow} at continuous time \(t = n \times \eta\).
\section{Exact Hamiltonian Flow for minimizing convex \texorpdfstring{\(f\)}{f}}
\label{sec:hf_convexity}

\subsection{Properties of the average position along the Hamiltonian flow trajectory}
The principle of energy conservation (see \Cref{lem:hamflow-conserve}) is a key property of the Hamiltonian flow, and does not place any assumptions on the potential \(f\) beyond differentiability. As a consequence of energy conservation, when \(Y_{0} = \bm{0}\), we have $f(X_t)+\frac12\|Y_t\|^2=f(X_0)$, and hence \(f(X_t)\leq f(X_0)\) for every \(t\geq 0\). Therefore, outputting the endpoint of Hamiltonian flow is always a descent method. In particular, for short integration times of order \(1/\sqrt L\) when $f$ is $L$-smooth, endpoint guarantees can be obtained; see \cite[Lemma~6]{fu2025hamiltoniandescentalgorithmsoptimization}. However, such short-time endpoint guarantees do not yield acceleration: in \Cref{lem:LowerBound}, we show that with a fixed short integration time \(T=c/\sqrt L\), the total integration time required to find an $\varepsilon$-close point with \eqref{eq:HF-opt-update} scales as \(\Omega((\sqrt{L}/\alpha)\log(1/\varepsilon))\) in the worst case. 
Our key observation is that when \(f\) is convex, an \textit{averaged} iterate over a suitably long trajectory can yield a stronger descent guarantee than that implied by the endpoint of Hamiltonian flow run from \(Y_{0}= \bm{0}\).
We formalize this key property of the Hamiltonian flow trajectory in the following lemma, which to the best of our knowledge has not been derived previously.
\begin{lemma}
\label{lem:hf-convex-main}
Assume \(f\) satisfies \emph{\convexassump{}}.
For any \(t > 0\), initial position \(X_{0} \in \bbR^{d}\), and any reference position \(z \in \bbR^{d}\), the Hamiltonian flow trajectory \((\HF((X_{0}, \bm{0}); \tau))_{\tau \in [0, t]}\) satisfies
\begin{equation}
\label{eq:ref-ineq-hf-convex}
f(X_{t}) - f(z) \leq \frac{2}{3}(f(X_{0}) - f(z)) - \frac{1}{6}\frac{\rmd^{2}}{\rmd t^{2}}\|X_{t} - z\|^{2}~.
\end{equation}
\end{lemma}
\begin{proof}
Consider \(h(t) = \langle X_{t} - z, Y_{t}\rangle = \frac{1}{2} \frac{\rmd}{\rmd t}\|X_{t} - z\|^{2}\).
We have
\begin{align*}
    \frac{1}{2} \frac{\rmd^{2}}{\rmd t^{2}} \|X_{t} - z\|^{2} = \dot{h}(t) &\overset{(a)}= \|Y_{t}\|^{2} - \langle X_{t} - z, \nabla f(X_{t})\rangle \\
    &\overset{(b)}= 2(f(X_{0}) - f(X_{t})) - \langle X_{t} - z, \nabla f(X_{t})\rangle \\
    &\overset{(c)}\leq 2(f(X_{0}) - f(X_{t})) - (f(X_{t}) - f(z)) \\
    &= -3(f(X_{t}) - f(z)) + 2(f(X_{0}) - f(z))~.
\end{align*}
Step \((a)\) uses the Hamiltonian flow \(\dot{X}_{t} = Y_{t}, \dot{Y}_{t} = -\nabla f(X_{t})\) and the chain rule; step \((b)\) uses the fact that the Hamiltonian is conserved for any \(t \geq 0\); and step \((c)\) uses the convexity of \(f\)~.
\end{proof}
For any \(T > 0\), given the Hamiltonian flow trajectory $(X_{t}, Y_{t}) = \HF((X_{0}, \bm{0}); t)$ for $0 \le t \le T$, we define the \textit{weighted average}:
\begin{equation}\label{eq:WeightedAvg}
    X^{\mathrm{avg}}(X_{0}; T) = \frac{2}{T^{2}}\int_{0}^{T}(T - t)X_{t}\,\rmd t\,.
\end{equation}
An important corollary of \Cref{lem:hf-convex-main} is stated below, which we prove in \cref{prf:cor:hf-convex-main}.
\begin{corollary}
\label{cor:hf-convex-main}
Consider the setting of \Cref{lem:hf-convex-main}.
The weighted average \(X^{\mathrm{avg}}(X_{0}; T)\) satisfies
\begin{equation}
\label{eq:hf-convex-weighted-avg}
    f(X^{\mathrm{avg}}(X_{0}; T)) - f(z) + \frac{1}{3T^{2}} \|X_{T} - z\|^{2} \leq \frac{2}{3} (f(X_{0}) - f(z)) + \frac{1}{3T^{2}} \|X_{0} - z\|^{2}\,.
\end{equation}
\end{corollary}

\Cref{cor:hf-convex-main} implies that optimality gap of the time-average of the solution to the Hamiltonian flow can be controlled with respect to the optimality gap at the initialization, with a \textit{universal} contraction factor $\frac{2}{3}$.
This key property allows us to leverage the Hamiltonian flow as a subroutine to design \emph{deterministic} algorithms for convex optimization with accelerated rates.
Importantly, when \(z \leftarrow x^{\star}\) and \(T \geq \sqrt{\frac{2\|X_{0} - \xstar\|^{2}}{f(X_{0}) - f(\xstar)}}\)\,, from the bound above we obtain
\begin{equation*}
    f(X^{\mathrm{avg}}(X_{0}; T)) - f(\xstar) \leq \frac{5}{6} (f(X_{0}) - f(\xstar))~.
\end{equation*}

\begin{remark}
\label{rmk:stronger}
A corollary of \Cref{lem:hf-convex-main} that we show en route proving \Cref{cor:hf-convex-main} is
\begin{equation}
\label{eq:lem:hf-convex-main-stronger}
    \frac{2}{T^{2}}\int_{0}^{T} (T- t) (f(X_{t}) - f(z)) \rmd t + \frac{1}{3T^{2}} \|X_{T} - z\|^{2} \leq \frac{2}{3} (f(X_{0}) - f(z)) + \frac{1}{3T^{2}} \|X_{0} - z\|^{2}~.
\end{equation}
Suppose \(\sft\) is distributed according to density \(p_{\sft}(t) = \frac{2}{T^{2}}(T - t)\) over \([0, T]\).
Then, \cref{eq:lem:hf-convex-main-stronger} implies
\begin{equation*}
    \bbE_{\sft \sim p_{\sft}}\left[f(X_{\sft}) - f(z)\right] + \frac{1}{3T^{2}} \|X_{T} - z\|^{2} \leq \frac{2}{3} (f(X_{0}) - f(z)) + \frac{1}{3T^{2}} \|X_{0} - z\|^{2}~.
\end{equation*}
When \(T \geq \sqrt{\frac{2\|X_{0} - \xstar\|^{2}}{f(X_{0}) - f(\xstar)}}\), the above inequality implies that \(f(X_{\sft}) - f(\xstar) \leq \frac{5}{6}(f(X_{0}) - f(\xstar))\) 
in expectation, taken over \(\sft \sim p_{\sft}\).
This novel randomized scheme is fundamentally different from the randomized Hamiltonian flow scheme studied in \citet{fu2025hamiltoniandescentalgorithmsoptimization}.
\end{remark}

Additionally, \Cref{cor:hf-convex-main} can be generalized to different weighting schemes like the simple average; we elaborate on this in \Cref{app:hf-convex-main-other-weights}.
\subsection{Provable minimization of \texorpdfstring{\(\alpha\)-strongly convex \(f\)}{α-strongly convex f}}

We begin by stating an implication of \Cref{cor:hf-convex-main} for functions satisfying \(\alpha\)-quadratic growth.
\begin{lemma}
\label{lem:hf-strong-convex-main}
Assume \(f\) satisfies \emph{\convexassump{}} and \emph{\quadgrowassump{}}.
For any \(T > 0\), the weighted average \(X^{\mathrm{avg}}(X_{0}; T)\) defined in \cref{eq:WeightedAvg} satisfies
\begin{equation*}
    f(X^{\mathrm{avg}}(X_{0}; T)) - f(\xstar) \leq \frac{2}{3}\lp1 + \frac{1}{\alpha T^{2}}\rp (f(X_{0}) - f(\xstar))~.
\end{equation*}
\end{lemma}
This is a direct consequence of the quadratic growth condition of \(f\); the complete proof is given in \Cref{app:prf:hf-strong-convex-main}.
Algorithmically, this lemma conveys that when given access to the Hamiltonian flow trajectory, aggregating the position according to the average \(X^{\mathrm{avg}}(X_{0}; T)\) for sufficiently long time results in an exponential contraction of the optimality gap.
This suggests the algorithm \nameref{alg:phfopt} presented in~\Cref{alg:phfopt} will converge exponentially fast; 
we present the formal guarantee in the following theorem.
We note for a strongly convex objective function, we do not need any additional averaging (i.e., we set $\lambda = 0$ in~\Cref{alg:phfopt}).

\begin{theorem}
\label{thm:hf-cts-strong-cvx}
Assume that \(f\) satisfies \emph{\convexassump{}} and \emph{\quadgrowassump{}}.
Let \(X_{K}\) be the output of \emph{\nameref{alg:phfopt}} (\Cref{alg:phfopt}) for \(K\) iterations with parameter \(\lambda = 0\) and the integration time sequence \(T_{k} = T \geq \frac{C}{\sqrt{\alpha}}\) for all \(k \in [K]\), for some constant $C > 0$.
Then, 
\begin{equation*}
    f(X_{K}) - f(\xstar) \leq \left(\frac{2}{3} + \frac{2}{3C^{2}}\right)^{K} (f(X_{0}) - f(\xstar))~.
\end{equation*}
\end{theorem}

When \(C > \sqrt{2}\), we have \(\frac{2}{3} + \frac{2}{3C^{2}} < 1\).
Moreover, when \(C = 2\), the lower bound for \(T\) in \Cref{thm:hf-cts-strong-cvx} implies, by the \(\alpha\)-quadratic growth condition of $f$, the choice \(T \geq \sqrt{\frac{4}{\alpha}} \geq \sqrt{\frac{2\|X_{0} - x^{\star}\|^{2}}{f(X_{0}) - f(x^{\star})}}\) discussed previously.
From this theorem, we obtain a continuous-time convergence guarantee: an $\varepsilon$-accurate point can be obtained with total integration time scaling as $\alpha^{-1/2}\log(1/\varepsilon)$.
\begin{corollary}
\label{cor:hf-cts-strong-cvx}
Assume that \(f\) satisfies \emph{\convexassump{}} and \emph{\quadgrowassump{}}.
To obtain \(X_{K}\) that is an \(\varepsilon\)-accurate minimizer, it suffices to run \emph{\CONTALG{}} with parameter \(\lambda = 0\) for \(K\) iterations and constant integration times \(\{T_{k}\}_{k=1}^{K}\) where
\begin{equation*}
    K = \left\lceil4 \log \frac{f(X_{0}) - f(x^{\star})}{\varepsilon}\right\rceil~; \quad T_{k} = \frac{5}{2\sqrt{\alpha}} ~~\forall ~k \in [K]~.
\end{equation*}
The total integration time is $T_{\mathrm{tot}} = \frac{5}{2\sqrt{\alpha}} \left\lceil 4 \log \frac{f(X_{0}) - f(x^{\star})}{\varepsilon}\right\rceil~.$
\end{corollary}

The proofs of the above theorem and corollary are stated in \Cref{app:prf:hf-cts-strong-cvx}.

\subsection{Provable minimization of convex \texorpdfstring{\(f\)}{f}}
We revisit \cref{eq:hf-convex-weighted-avg} and the choice  \(\sqrt{\frac{2\|X_{0} - \xstar\|^{2}}{f(X_{0}) - f(\xstar)}}\) which is sufficient for a contraction in the optimality gap.
When \(f\) is convex, this ratio can be arbitrarily large, and therefore optimizing \(f\) using the averaged iterate can appear futile even when using exact Hamiltonian flow.
We note that this lower bound for \(T\) stems from disregarding the \(\frac{1}{3T^{2}}\|X_{T} - \xstar\|^{2}\) term in \cref{eq:hf-convex-weighted-avg}.
On the other hand, by including this term, it is unclear if there exists some \(\widehat{X}\) that implies a decrease in optimality gap; this is due to the mismatch between \(X^{\mathrm{avg}}(X_{0}; T)\) passed to \(f\) and the \(X_{T}\) in the squared norm term.
Interestingly, by using an additional averaging with the last iterate, we derive a version of \cref{eq:hf-convex-weighted-avg} that explicitly defines a \(\widehat{X}\) that leads to actionable algorithms in the convex setting.
\begin{lemma}
\label{lem:hf-convex-mix-main}
Let \(f\) satisfy \emph{\convexassump{}}.
For any initial~position~\(X_{0} \in \bbR^{d}\),~time~\(T > 0\),~and~parameter~\(\lambda > 0\),~define
\begin{equation*}
    X^{\mathrm{mix}}(X_{0}; T) = \frac{\lambda}{\lambda + 1} X_{T} + \frac{1}{1 + \lambda} X^{\mathrm{avg}}(X_{0}; T)~.
\end{equation*}
Then, defining \(C_{\lambda} = \frac{(3\lambda + 1)^{2}}{9\lambda(1 + \lambda)}\)\,, we have
\begin{equation*}
    f(X^{\mathrm{mix}}(X_{0}; T)) - f(\xstar) + \frac{1}{3\lambda T^{2}} \|X^{\mathrm{mix}}(X_{0}; T) - \xstar\|^{2} \leq C_{\lambda} (f(X_{0}) - f(\xstar)) + \frac{1}{3\lambda T^{2}} \|X_{0} - \xstar\|^{2}\,.
\end{equation*}
\end{lemma}
\begin{prfsketch}
By the properties of the Hamiltonian flow,
\begin{subequations}
    \begin{align}
        f(X^{\mathrm{avg}}(X_{0}; T)) - f(\xstar) + \frac{1}{3T^{2}}\|X_{T} - \xstar\|^{2} &\leq \frac{2}{3}(f(X_{0}) - f(\xstar)) + \frac{1}{3T^{2}}\|X_{0} - \xstar\|^{2} \label{eq:cor:hf-convex-main} \\
        \frac{1}{3T^{2}}\|X^{\mathrm{avg}}(X_{0}; T) - \xstar\|^{2} &\leq \frac{1}{9}(f(X_{0}) - f(\xstar)) + \frac{1}{3T^{2}}\|X_{0} - \xstar\|^{2} \label{eq:cor:hf-convex-main-2} \\
        f(X_{T}) - f(\xstar) &\leq f(X_{0}) - f(\xstar)~.\label{eq:hf-conserve}
    \end{align}
\end{subequations}
We simplify the weighted sum \(\lambda \cdot \text{\cref{eq:cor:hf-convex-main}} + \text{\cref{eq:cor:hf-convex-main-2}} + \lambda^{2} \cdot \text{\cref{eq:hf-conserve}}\), and use Jensen's inequality with the convexity of \(f\) and \(\|\cdot\|^{2}\) to complete the proof.
\end{prfsketch}

We give the complete proof in Appendix \ref{prf:lem:hf-convex-mix-main}.
We highlight that \(C_{\lambda} < 1\) for all \(\lambda > \frac{1}{3}\); $C_\lambda$ is minimized at \(\lambda = 1\), where it takes the value \(C_1 = \frac{8}{9}\); and \(X^{\mathrm{mix}}\) for \(\lambda = 1\) is a uniform average between \(X_{T}\) and \(X^{\mathrm{avg}}(X_{0}; T)\).
The above lemma implies a convergence guarantee for minimizing convex functions using \nameref{alg:phfopt}, provided the integration times increase suitably at each iteration, and we output $X^{\mathrm{mix}}$ as the next iterate (which is the reason we introduce the parameter $\lambda$ in~\nameref{alg:phfopt}).

\begin{theorem}
\label{thm:hf-cts-cvx}
Assume \(f\) satisfies \emph{\convexassump{}}.
Suppose \(X_{K}\) is the output of \emph{\nameref{alg:phfopt}} after \(K\) iterations with parameter \(\lambda = 1\) and integration time sequence satisfying \(T_{0} = 1\) and \(T_{k}^{2} \geq \frac{9}{8}T_{k - 1}^{2}\) for all \(k \in [K]\).
Then,
\begin{equation*}
    f(X_{K}) - f(x^{\star}) \leq \left(\frac{8}{9}\right)^{K} \left(f(X_{0}) - f(x^{\star}) + \frac{1}{3}\|X_{0} - x^{\star}\|^{2}\right)~.
\end{equation*}
\end{theorem}

While the above theorem indicates that the optimality gap decays geometrically by the factor of $\frac{8}{9}$, the integration time increases geometrically at each iteration.
However, since the \emph{squared} integration time \(T^2_{k}\) increases precisely by $\frac{9}{8}$, the inverse contraction factor, the total integration time to achieve a \(\varepsilon\)-accurate solution scales as \(\frac{1}{\sqrt{\varepsilon}}\); this is presented in the following corollary, which is a counterpart to \Cref{cor:hf-cts-strong-cvx} for convex \(f\).
\begin{corollary}
\label{cor:hf-cts-cvx}
Let \(f\) satisfy \emph{\convexassump{}}.
To obtain \(X_{K}\) that is a \(\varepsilon\)-accurate minimizer, it suffices to run \emph{\CONTALG{}} with parameter \(\lambda = 1\) for \(K\) iterations and integration time sequence \(\{T_{k}\}_{k=1}^{K}\) such that
\begin{equation*}
    K = \left\lceil \frac{1}{\log(\frac{9}{8})} \cdot \log \frac{f(X_{0}) - f(x^{\star}) + \frac{1}{3}\|X_{0}  - x^{\star}\|^{2}}{\varepsilon}\right\rceil~; \quad T_{k} = \left(\frac{9}{8}\right)^{\nicefrac{k}{2}}~.
\end{equation*}
The total integration time is $T_{\mathrm{tot}} \leq 18 \sqrt{\frac{f(X_{0}) - f(x^{\star}) + \frac{1}{3}\|X_{0} - x^{\star}\|^{2}}{\varepsilon}}~.$
\end{corollary}

The proofs of the above theorem and corollary are stated in \Cref{app:prf:hf-cts-cvx}.
\section{From exact Hamiltonian Flow to a discrete time method for minimizing convex \texorpdfstring{\(f\)}{f}}
\label{sec:discrete-algos}
The results of the prior section showcased the potential of an idealized iterative method based on the Hamiltonian flow for minimizing \(f\) at an accelerated rate while only having access to a gradient oracle of \(f\) and an exact solver for the Hamiltonian dynamics.
In this section, we explore how to transfer the guarantees therein to an \textit{implementable} algorithm via a practical discretization of the Hamiltonian flow, specifically using the \emph{extragradient} integrator \citep[Sec. 4]{fu2025hamiltoniandescentalgorithmsoptimization}.
\subsection{Properties of the extragradient integrator}
A crucial property of the extragradient integrator that makes it a viable scheme for minimizing convex \(f\) is that, when the step size is sufficiently small, the Hamiltonian value \emph{decreases} with every iteration; this is stated in the following lemma (see \Cref{app:prf:hf-extg-nonincrease-ham} for the proof).
\begin{lemma}
    \label{lem:hf-extg-nonincrease-ham}
    Let \(f\) satisfy \emph{\convexassump{}} and \emph{\smoothassump{}}.
    Choose \(\eta \leq \nicefrac{1}{\sqrt{L}}\), and suppose \((x_{n + 1}, y_{n + 1})\) is produced by one step of the extragradient integrator \emph{(\ref{eqn:extg-update})} from \((x_{n}, y_{n})\).
    Then
    \begin{equation}
    \label{eq:hf-extg-nonincrease-ham}
        f(x_{n+1}) + \frac{1}{2}\|y_{n+1}\|^2 \leq f(x_n) + \frac12\|y_n\|^2~.
    \end{equation}
\end{lemma}
We contrast this with the property of the exact Hamiltonian flow along which the Hamiltonian value is \emph{preserved}.
Nonetheless, a consequence of this lemma is that running the extragradient integrator from \(y_{0} = \bm{0}\) and any position \(x_{0} \in \bbR^{d}\) results in a sequence of \(\{x_{n}\}_{n \geq 0}\) such that \(f(x_{n}) \leq f(x_{0})\), mirroring what happens along the Hamiltonian flow.
However, as in the continuous-time setting, this bound relative to the initial function value alone does not provide an explicit quantitative convergence rate.
It is therefore natural to ask if the extragradient integrator satisfies a version of \Cref{cor:hf-convex-main}, and we answer this affirmatively.
In particular, provided that the step size is sufficiently small, it is possible to control the optimality gap of an average of the iterates produced by the extragradient integrator.
For \(N > 0\), we define the weighted average
\begin{equation*}
    x^{\mathrm{avg}}(x_{0}; N) = \frac{2}{N(N + 1)}\sum_{n=1}^{N}(N - n + 1) x_{n}~.
\end{equation*}
\begin{lemma}\label{lem:hf-extg-convex-main}
Assume $f$ satisfies \emph{\convexassump{}} and \emph{\smoothassump{}}, and let $\{(x_n, y_n)\}_{n \geq 1}$ be the iterates of extragradient integrator \emph{(\ref{eqn:extg-update})} with step size $\eta \leq \frac{1}{\sqrt{L}}$, starting from $(x_0, \bm{0})$.
Then, for any \(N \geq 1\),
\begin{equation}
    \label{eqn:oracle_step_c_extg}
f(x^{\mathrm{avg}}(x_{0}; N )) - f(x^{\star}) + \frac{\|x_{N} - x^{\star}\|^{2}}{3\eta^{2}N(N+1)} 
    \leq \frac{2}{3}(f(x_{0}) - f(x^{\star})) + \frac{\|x_{0} - x^{\star}\|^{2}}{3\eta^{2}N(N + 1)}~.
\end{equation}
\end{lemma}
We prove this lemma in Appendix \ref{sec:prf:hf-extg-convex-main}.
The result above is motivated by the properties that we establish for the average iterate along the Hamiltonian flow trajectory in the exact Hamiltonian flow setting discussed in the prior section, which leads to the similarities between \cref{eqn:oracle_step_c_extg} and \cref{eq:hf-convex-weighted-avg}.

\begin{algorithm}[t]
\caption{Discretized Hamiltonian Flow for optimization with Averaging (\DISCALG{})}
\algotitle{\DISCALG{}}{alg:disc_phfopt}
\SetKwInOut{Input}{Input}\SetKwInOut{Output}{Output}
\Input{Initialization: $x_0\in \R^d$, Number of iterations: $K \in \mathbb{N}$, parameter: $\lambda  \in \R_{\geq 0}$, step size: $\eta > 0$, discretization steps sequence: $\{N_1, \dots, N_K\} \in \mathbb{N}^K$}
\For{\(k = 1\) \KwTo \(K\)}{
    Obtain sequence \(\{(x^{(k)}_{n}, y^{(k)}_{n})\}_{n = 1}^{N_{k}}\) using \ref{eqn:extg-update} from \((x_{0}^{(k)}, y_{0}^{(k)}) = (x_{k- 1}, \bm{0})\).

    Compute \(x^{\mathrm{avg}}(x_{k - 1}; N_{k}) = \frac{2}{N_{k}(N_{k} + 1)}\sum_{n=1}^{N_{k}} (N_{k} - n + 1) x^{(k)}_{n}\).

    Set \(x_{k} = \frac{1}{\lambda + 1} x^{\mathrm{avg}}(x_{k - 1}; N_{k}) + \frac{\lambda}{\lambda + 1} x^{(k)}_{N_{k}}\).
    }
\Return $x_K$
\end{algorithm}

\Cref{lem:hf-extg-convex-main} provides the theoretical basis for the design of \nameref{alg:disc_phfopt} (\Cref{alg:disc_phfopt}) for minimizing convex and \(L\)-smooth functions.
Contrary to \nameref{alg:phfopt} (which assumes we can simulate the exact Hamiltonian flow), \nameref{alg:disc_phfopt} is implementable with access to a gradient oracle to \(f\).
From an implementation standpoint, obtaining the sequence \(\{x_{n}^{(k)}, y_{n}^{(k)}\}_{n=1}^{N_{k}}\) for every \(k\) would be performed in an inner loop.
The computation of \(x^{\mathrm{avg}}(x_{k - 1}; N_{k})\) can be incorporated as part of this loop wherein it is computed as a running weighted sum (analogous to a streaming sum), and therefore does not require storing the entire sequence in memory for its computation; the presentation in~\Cref{alg:disc_phfopt} is intended for readability.
It is also possible to replace \ref{eqn:extg-update} with \ref{eq:implicit-integ}, which requires access to a proximal oracle for \(f\) (see details in \cref{app:sec:dhfa-im}); essentially, this substitution results in quantitatively similar iteration complexities as using \textsf{HF\textsuperscript{extg}} but without assuming smoothness of \(f\).

\subsection{Provable minimization of \texorpdfstring{$\alpha$-strongly convex $f$}{α-strongly convex f}}
\begin{theorem}
\label{thm:hf-extg-strong-cvx}
Assume that \(f\) satisfies \emph{\convexassump{}}, \emph{\smoothassump{}} and \emph{\quadgrowassump{}}.
Suppose \(x_{K}\) is the output of running \emph{\nameref{alg:disc_phfopt}} for \(K\) iterations from initialization \(x_{0} \in \bbR^{d}\), with parameter \(\lambda = 0\), step size \(0 < \eta \leq \frac{1}{\sqrt{L}}\), and discretization steps sequence satisfying \(N_{k} = N \geq \lceil \frac{c}{\eta\sqrt{\alpha}}\rceil\) for all \(k \in [K]\), for any $c > 0$.
Then,
\begin{equation*}
    f(x_{K}) - f(x^{\star}) \leq \left(\frac{2}{3} + \frac{2}{3c^{2}}\right)^{K}(f(x_{0}) - f(x^{\star}))~.
\end{equation*}
\end{theorem}

From this, we obtain the following corollary on the iteration complexity of~\nameref{alg:disc_phfopt} for minimizing convex functions that satisfy both \(\alpha\)-quadratic growth and \(L\)-smoothness.
The proofs of the previous theorem and the following corollary are given in \Cref{app:prf:hf-extg-strong-cvx}.
\begin{corollary}
\label{cor:hf-extg-strong-cvx}
Assume that \(f\) satisfies \emph{\convexassump{}}, \emph{\smoothassump{}} and \emph{\quadgrowassump{}}.
To obtain \(x_{K}\) that is an \(\varepsilon\)-accurate minimizer, it suffices to run \emph{\nameref{alg:disc_phfopt}} with parameter \(\lambda = 0\) and step size \(\eta = \frac{1}{\sqrt{L}}\) for \(K\) iterations and discretization step sequence \(\{N_{k}\}_{k=1}^{K}\) such that
\begin{equation*}
    K = \left\lceil 4 \log \frac{f(x_{0}) - f(x^{\star})}{\varepsilon} \right\rceil~; \quad N_{k} = \left\lceil \frac{5}{2}\sqrt{\kappa} \right\rceil~~~\forall~k \in [K]~.
\end{equation*}
The total number of gradient evaluations is $N_{\mathrm{tot}} = 2 \cdot \left\lceil \frac{5}{2}\sqrt{\kappa}\right\rceil \cdot \left\lceil 4 \log \frac{f(x_{0}) - f(x^{\star})}{\varepsilon} \right\rceil~.$
\end{corollary}

\subsection{Provable minimization of convex \texorpdfstring{$f$}{f}}

Owing to the similarity between \cref{eq:hf-convex-weighted-avg} and \cref{eqn:oracle_step_c_extg}, one can naturally expect that a combination of \(x^{\mathrm{avg}}(x_{0}; N)\) and \(x_{N}\) leads to a viable strategy to minimizing \(f\).
In the following lemma, we show that this is indeed the case, and this explains the necessity of the parameter \(\lambda\) in \nameref{alg:disc_phfopt}.
We provide the proof of this lemma in~\Cref{prf:lem:hf-extg-convex-mix-main}.
\begin{lemma}
\label{lem:hf-extg-convex-mix-main}
Assume that \(f\) satisfies \emph{\convexassump{}} and \emph{\smoothassump{}}.
Let \(\{(x_{n}, y_{n})\}_{n\geq 1}\) be generated by the extragradient integrator with step size \(\eta \leq \frac{1}{\sqrt{L}}\).
For any \(N \geq 4\) and initial position \(x_{0}\), define
\begin{equation*}
    x^{\mathrm{mix}}(x_{0}; N) = \frac{\lambda}{\lambda + 1} x_{N} + \frac{1}{1 + \lambda} x^{\mathrm{avg}}(x_{0}; N)~.
\end{equation*}
Then, defining \(c_{\lambda} = \frac{(6\lambda^{2} + 4\lambda + 1)}{6\lambda(\lambda + 1)}\), we have
\[
f(x^{\mathrm{mix}}(x_{0}; N)) - f(x^{\star}) + \frac{\|x^{\mathrm{mix}}(x_{0}; N) - x^{\star}\|^{2}}{3\lambda \eta^{2}N(N + 1)} \leq c_{\lambda}(f(x_{0}) - f(x^{\star})) + \frac{\|x_{0} - x^{\star}\|^{2}}{3\lambda\eta^{2}N(N + 1)}~.
\]
\end{lemma}

The factor \(c_{\lambda}\) is strictly less than \(1\) for \(\lambda > \frac{1}{2}\), is minimized at \(\lambda^{\star} = \frac{\sqrt{3} + 1}{2}\) (where it equals \(c_{\lambda^{\star}} = \frac{\sqrt{3} + 1}{3}\)), and this combination biases \(x^{\mathrm{mix}}(x_{0}; N)\) toward \(x_{N}\).
The discrepancy between \(C_{\lambda}\) in \Cref{lem:hf-convex-mix-main} (for continuous time) and \(c_{\lambda}\) in the above lemma is due to a discrete sum appearing in the analysis instead of an integral.

\begin{theorem}
\label{thm:hf-extg-cvx}
Assume that \(f\)  satisfies \emph{\convexassump{}} and \emph{\smoothassump{}}.
Suppose \(x_{K}\) is the output of running \emph{\nameref{alg:disc_phfopt}} for \(K\) iterations with parameter \(\lambda = \frac{\sqrt{3}+1}{2}\), step size \(0 < \eta \leq \frac{1}{\sqrt{L}}\), and discretization steps sequence \(\{N_{k}\}_{k=1}^{K}\) satisfying \(N_{0} = 4\) and \(N_{k}(N_{k} + 1) \geq \frac{3}{\sqrt{3}+1}N_{k - 1}(N_{k - 1} + 1)\) for all \(k \in [K]\).
Then,
\begin{equation*}
    f(x_{K}) - f(x^{\star}) \leq \lp \frac{\sqrt{3}+1}{3}\rp^{K}\left(f(x_{0}) - f(x^{\star}) + \frac{\sqrt{3}-1}{60\eta^{2}}\|x_{0} - x^{\star}\|^{2}\right)~.
\end{equation*}
\end{theorem}

Drawing a comparison to \Cref{thm:hf-cts-cvx}, we see a geometric rate of decrease over iterations \(K\), while the number of discretization steps per iteration increases exponentially.
However, due to the fact that the \emph{squared} discretization steps increases by the reciprocal of the factor that the optimality gap decreases by at each iteration, we obtain a \(\frac{1}{\sqrt{\varepsilon}}\) scaling in the total gradient complexity to obtain an \(\varepsilon\)-accurate solution.
This is stated in the following corollary. The proofs of the previous theorem and the following corollary are given in \Cref{app:prf:hf-extg-cvx}.

\begin{corollary}
\label{cor:hf-extg-cvx}
Let \(f\) satisfy \emph{\convexassump{}} and \emph{\smoothassump{}}.
To obtain \(x_{K}\) that is an \(\varepsilon\)-accurate minimizer, it suffices to run \emph{\nameref{alg:disc_phfopt}} with parameter \(\lambda = \frac{\sqrt{3}+1}{2}\) and step size \(\eta = \frac{1}{\sqrt{L}}\) for \(K\) iterations and discretization step sequence \(\{N_{k}\}_{k=1}^{K}\) such that
\begin{equation*}
    K = \left\lceil 2\left(\log \frac{3}{1+\sqrt{3}}\right)^{-1} \cdot \log \lp\sqrt{\frac{L}{\varepsilon}}\,\|x_0-\xstar\|\rp\right\rceil;~~ N_k=\left\lceil \sqrt{\frac{3}{1+\sqrt{3}}}N_{k-1}+\tfrac{1}{2} \right\rceil;~~N_0=4.
\end{equation*}
The total number of gradient evaluations satisfies $N_{\mathrm{tot}} \leq 1620 \sqrt{\frac{L}{\varepsilon}}\,\|x_0-\xstar\|~.$
\end{corollary}
\section{Conclusion}
In this work, we provide a new perspective on how the Hamiltonian flow can be used for convex optimization.
Through a novel differential inequality along Hamiltonian flow for convex objective (\Cref{lem:hf-convex-main}), we show that a suitably chosen aggregate of the Hamiltonian flow trajectory is well-suited for optimization purposes, and leads to the idealized algorithm \nameref{alg:phfopt}.
To make this practical in settings where \ref{eq:HamFlow} cannot be simulated exactly, we propose a discretized version based on the extragradient integrator in \nameref{alg:disc_phfopt}, which only requires access to the gradients of \(f\).
We prove~\nameref{alg:disc_phfopt} achieves accelerated convergence rates for smooth convex and strongly convex optimization.

Our results spur many interesting future directions.
We note that our analysis predicts the same qualitative rate for different aggregation schemes such as the uniform average or the minimizer along the trajectory.
We conjecture that choosing the next iterate to be the minimizer along the trajectory converges at a strictly better rate than other aggregation schemes.
While we give guarantees for minimizing convex functions, it would be interesting to study whether \nameref{alg:phfopt} or \nameref{alg:disc_phfopt} can provably optimize other classes of functions, such as those satisfying a Polyak-\L{}ojaciewicz inequality, or even non-convex functions. 
In this work, we consider two integrators that dissipate the Hamiltonian energy, namely implicit integrator and extragradient integrator.
However, our experiments suggest that the leapfrog integrator exhibits a convergence behavior similar to the implicit and extragradient integrators; it would be interesting to study this theoretically.
Finally, while the results of this paper concern optimization, it would be interesting to see if the broader techniques developed in this work can be translated to developed \textsf{HMC}-based algorithms that perform \emph{accelerated sampling}.

\paragraph{Acknowledgments.}
Qiang Fu, Siddharth Mitra, Xiuyuan Wang, and Andre Wibisono were supported by NSF awards CCF-2403391 and CAREER CCF-2443097.
Vishwak Srinivasan and Ashia Wilson were supported in part by Assicurazioni Generali S.p.A. through MIT Award 036189-00006.

\newpage

\bibliography{references.bib}

\newpage
\appendix
\part{Appendix}
\setcounter{parttocdepth}{3}
\renewcommand \beforeparttoc{}
\renewcommand \afterparttoc{}
\parttoc
\newpage

\crefalias{section}{appendix}
\crefalias{subsection}{appendix}
\crefalias{subsubsection}{appendix}
\section{Additional Related Works}\label{app:AdditionalRelatedWorks}

\paragraph{Accelerated gradient flow and accelerated gradient descent.}

The accelerated gradient method of \citet{nesterov1983method} is the following update:
\begin{equation*}
    \begin{cases}
    x_{k+1} = y_k - \eta \nabla f(y_k), \\
    y_{k+1} = x_{k+1} + \gamma_k (x_{k+1} - x_k).
    \end{cases}\tag{\textsf{AGD}}\label{eq:AGD}
\end{equation*}

The dynamical systems perspective of \ref{eq:AGD} presented in \citet{su2016differential} spawned many follow-up works pertaining to algorithm design and generalizations of the techniques.
\citet{wibisono2016variational} develop a variational perspective of accelerated methods in optimization via the principle of least action using the Bregman Lagrangian, which recover known accelerated methods and derive new and faster accelerated methods beyond \ref{eq:AGD}, as well as the discrete-time algorithms with matching convergence rates.
Additional papers that adopt the dynamical perspective of acceleration include \citet{hu2017dissipativity,maddison2018hamiltonian,muehlebach2019dynamical,even2021continuized,attouch2022first,shi2022understanding}.
This continues to be an active area of research.

\paragraph{Hamiltonian dynamics for optimization.}
Taking this dynamical perspective further, the \emph{Hamiltonian flow} (\ref{eq:HamFlow}) corresponds to an undamped second-order system that conserves total energy and therefore does not converge to a minimizer in continuous time.
While there have been many works that study \textit{damped} Hamiltonian systems (such as those derived from accelerated gradient methods), to the best of our knowledge, only a small number of prior works have investigated optimization algorithms based on \textit{conservative} Hamiltonian dynamics, as we study in this paper.

As mentioned in \cref{sec:intro}, \citet{teel2019first} propose \ref{eq:HF-opt-update}, and in their scheme, the velocity is reset to zero either when the trajectory approaches the boundary of the region 
$\{(x,y)\in \mathbb{R}^d\times\mathbb{R}^d: \langle\nabla f(x), y\rangle \le 0,\ \|y\|^2 \ge \|\nabla f(x)\|^2/L\}$
where $L$ is the smoothness constant, or when a fixed timer expires; see the pseudocode below. 
They establish uniform global stability and (non-accelerated) convergence guarantees for minimizing smooth and strongly convex functions, albeit only in the idealized setting assuming we can simulate the continuous-time Hamiltonian flow.

\begin{algorithm}[H]
\DontPrintSemicolon
\caption{Hybrid Hamiltonian Algorithm with Resets (\citet{teel2019first})}
\SetKwInOut{Input}{Input}\SetKwInOut{Output}{Output}
\SetAlgoLined
\Input{Initialization: $X_0 \in \R^d$, number of iterations: $K \in \mathbb{N}$, maximum integration time: $T \in (0,\infty)$, threshold parameter: $\bar L \ge 0$}
\For{$k = 1$ \KwTo $K$}{
    Solve \ref{eq:HamFlow} from \((X_{k - 1}, \bm{0})\) for time \(T\) to obtain trajectory \(((X_{t}^{(k)}, Y_{t}^{(k)}))_{t \in [0, T]}\).

    Compute $T_k=
    \inf\left\{
        t \in [0,T]:
\langle\nabla f(X_t^{(k)}),Y_t^{(k)} \rangle = 0
        \ \text{and}\
        \|Y_t^{(k)}\|^2
        \ge \frac{\|\nabla f(X_t^{(k)})\|^2}{\bar L}
    \right\}
    $, with the convention that $T_k = T$ if this set is empty.

    Set $X_k = X_{T_k}^{(k)}$.
}
\Return $X_K$
\end{algorithm}

\citet{Scagliotti_2021} propose a different restart rule for the same conservative dynamics \eqref{eq:HF-opt-update}. Starting from $(X_{k-1},\bm 0)$, they evolve the Hamiltonian flow until the \emph{mean dissipation}
$t \mapsto \frac{\|Y_t^{(k)}\|^2}{2t}$ attains a local maximum. Along the Hamiltonian flow, this stopping rule can be written as
$$
T_k
=
\inf\left\{
t>0:
-\,t\langle \nabla f(X_t^{(k)}), Y_t^{(k)}\rangle
-\frac12\|Y_t^{(k)}\|^2 <0
\right\},
$$
see the pseudocode below. They proved that this restart time is finite for coercive objectives, and uniformly bounded for $\alpha$-strongly convex and $L$-smooth $f$. As a consequence, the resulting piecewise conservative trajectory converges linearly:
$$
f(\widetilde X(t))-f(x^\star)
\le
\left(1+\frac{\alpha}{L}\right)^{-\lfloor t/T_R\rfloor}
\bigl(f(X_0)-f(x^\star)\bigr),
\qquad
T_R=\frac{32L}{\alpha\sqrt{\alpha}},
$$
and in particular
$$
f(X_k)-f(x^\star)
\le
\left(1+\frac{\alpha}{L}\right)^{-k}
\bigl(f(X_0)-f(x^\star)\bigr).
$$

\begin{algorithm}[H]
\DontPrintSemicolon
\caption{Hamiltonian Method of \citet{Scagliotti_2021}}
\SetKwInOut{Input}{Input}\SetKwInOut{Output}{Output}
\SetAlgoLined
\Input{Initialization: $X_0 \in \R^d$, number of iterations: $K \in \mathbb{N}$}
\For{$k = 1$ \KwTo $K$}{
    Solve \eqref{eq:HamFlow} from $(X_{k-1},\bm 0)$ until the first time
    $T_k=\inf\left\{t>0:\frac{\rmd}{\rmd t}\left(\frac{\|Y_t^{(k)}\|^2}{2t}\right)<0\right\}.$

    Set $X_k = X_{T_k}^{(k)}.$
}
\Return $X_K$
\end{algorithm}

\citet{diakonikolas2021generalized} analyze a class of \emph{generalized} Hamiltonian dynamics driven by a time-dependent Hamiltonian
$H(x,y,\tau) = h(\tau)\, f(x/\tau) + \psi^*(y)$,
where $h(\tau)>0$ is a function of the scaled time $\tau$, and $\psi^*$ is a strongly convex and differentiable function. 
They show that along the Hamiltonian flow, the averaged gradient
$\bigl\|\frac{1}{t}\int_0^t \nabla f(x_\tau)\, d\tau\bigr\|$
decays at an $t^{-1}$ rate. Moreover, they demonstrate that a broad family of momentum methods in both Euclidean and non-Euclidean geometries can be derived from these generalized dynamics, encompassing classical algorithms such as \ref{eq:AGD} and other accelerated methods~\citep{wibisono2016variational,wilson2021lyapunov}. 
These generalized constructions are conceptually distinct from the energy-conserving Hamiltonian dynamics studied in this work, where we use the standard Hamiltonian function which is time-independent with a quadratic kinetic energy.

\citet{de2023improving} propose a Hamiltonian-based method with Hamiltonian function
$H(x,y)=\lambda\log (F(x)-F_0)+\log(\|y\|^2)$,
where $\lambda, F_{0}$ are user-specified parameters.
They performs an empirical evaluation of this algorithm, but do not provide a convergence rate guarantee.

\paragraph{Accelerated optimization via Hamiltonian dynamics.}
A recent work by \citet{wang2025frictionlesshamiltoniandescentcoordinate} studies \ref{eq:HF-opt-update} for minimizing strongly convex quadratics
$f(x)=\frac{1}{2}x^{\top}Ax-b^{\top}x$ with $\alpha I\preceq A \preceq LI$.
In this quadratic setting, the Hamiltonian dynamics~\ref{eq:HamFlow} admit a closed-form solution, which enables a sharp analysis. 
By selecting a sequence of integration times $\{T_{k}\}$ according to the roots of Chebyshev polynomials,~\citet{wang2025frictionlesshamiltoniandescentcoordinate} proves that \ref{eq:HF-opt-update} converges to the minimizer of \(f\) at a rate of $\exp(-k/\sqrt{\kappa}))$ after $k$ iterations, with a total integration time to achieve an \(\varepsilon\)-accurate solution that scales as $1/\sqrt{\alpha}$, matching accelerated time complexity in \cref{tab:rates-agf-agd}.
A more systematic investigation of \ref{eq:HamFlow} as a foundation for first-order optimization was conducted by \citet{fu2025hamiltoniandescentalgorithmsoptimization}, wherein they study different integration schedules and discretization schemes, both for the general convex and strongly convex objectives. In particular, they show that \ref{eq:HamFlow} with a short integration time recovers the convergence rate of gradient descent.
Moreover, by randomizing the integration time, the resulting dynamics exhibit acceleration convergence rates analogous to the accelerated gradient flow.
The corresponding discretized algorithm, termed randomized Hamiltonian gradient descent (\textsf{RHGD}), achieves accelerated convergence rates that match the rates of \ref{eq:AGD}. These results demonstrate that both non-accelerated and accelerated first-order methods can be viewed as arising from Hamiltonian dynamics under different integration and discretization schemes.
Whereas the work of~\citet{fu2025hamiltoniandescentalgorithmsoptimization} provides accelerated rates for randomized integration times, so the guarantees only hold in expectation or in high probability, in this work we provide deterministic accelerated rates via long integration times.

\section{Background on Hamiltonian Dynamics}\label{app:HF_Background}

We present an extended discussion about  Hamiltonian dynamics and numerical integrators.
In \Cref{sec:ham-dynamics-def}, we restricted our attention to the energy function \(H(x, y) = f(x) + \frac{1}{2}\|y\|^{2}\) as it was more pertinent to optimization.
However, the Hamiltonian can be defined more generally over the phase space \(\bbR^{2d}\).
A key definition in the Hamiltonian dynamics is the \emph{symplectic matrix} \(\Omega \in \bbR^{2d \times 2d}\):
\begin{equation*}
    \Omega := \begin{bmatrix} \bm{0} & \rmI_{d} \\ -\rmI_{d} & \bm{0} \end{bmatrix}
\end{equation*}
where $\rmI_d$ is the identity matrix.
Note that $\Omega$ is a skew-symmetric matrix, i.e., \(\Omega^{\top} = -\Omega\).

\begin{definition}\label{def:hf-apdx}
Let $H \colon \R^{2d} \to \R$ be a continuously differentiable \emph{Hamiltonian} function.
The Hamiltonian flow with respect to \(H\) is the system of differential equations on the phase space given by
\begin{equation}\label{eq:hf-apdx}
\dot Z_t = \Omega \nabla H(Z_t)~.\tag{\textsf{HF\textsubscript{gen}}}
\end{equation}
\end{definition}

When the Hamiltonian function is separable, \(H(x, y) = f(x) + g(y)\), the Hamiltonian dynamics in \ref{eq:hf-apdx} reduces to
\begin{equation*}
    \dot X_t = \nabla g(Y_t)~,\qquad \dot Y_t = - \nabla f(X_t)~.
\end{equation*}
Note that when \(g(y) = \frac{1}{2}\|y\|^{2}\), this recovers \ref{eq:HamFlow} in~\Cref{sec:ham-dynamics-def}.
Throughout this paper, we assume that the Hamiltonian flow admits a unique solution for any initial condition \(Z_{0} = (X_{0}, Y_{0}) = (x, y)\), and the solution at time \(t\) is given by \(Z_{t} = \HF((x, y); t)\) as introduced in \Cref{sec:ham-dynamics-def}.
See e.g., \citet[\S3.1, \S3.2]{Leimkuhler_Reich_2005} for additional details and examples of Hamiltonian functions and dynamics in physical contexts.

\subsection{Properties of the Hamiltonian flow}

The Hamiltonian flow satisfies desirable properties in addition to the conservation property stated as \Cref{lem:hamflow-conserve} (stated for a specific Hamiltonian, but holds more generally).
These properties also play a crucial role in the development of Hamiltonian Monte Carlo algorithm for sampling.
Throughout, we assume that the Hamiltonian is a continuously differentiable function over the phase space.

\paragraph{Conservation of Hamiltonian}
\begin{lemma}
For any time \(t > 0\) and \(x, y \in \bbR^{d}\), \(H(\HF((x, y); t)) = H(x, y)\).
\end{lemma}
\begin{proof}
Let \(Z_t = \HF((x, y); t)\).
Since $H$ is continuously differentiable, by the chain rule we have
$$
\frac{\rmd}{\rmd t} H(Z_t)
= \nabla H(Z_t)^\top \dot Z_t
= \nabla H(Z_t)^\top \Omega \nabla H(Z_t)=0
$$
where we use the fact that \(v^{\top}\Omega v = 0\) for any vector $v$, since $\Omega$ is skew-symmetric.
\end{proof}

\paragraph{Symplecticity and preservation of volume}

\begin{definition}\label{def:symplectic-map}
A differentiable map $\Phi \colon \bbR^{2d}\to\bbR^{2d}$ is called \emph{symplectic} if for every $z\in\bbR^{2d}$, its Jacobian $\nabla \Phi(z)$ satisfies
\[
\nabla\Phi(z)^\top \Omega \nabla\Phi(z)=\Omega~.
\]
\end{definition}

\begin{definition}\label{def:volpreserve-map}
A differentiable map \(\Phi \colon \bbR^{2d} \to \bbR^{2d}\) is called \emph{volume-preserving} if for every \(z \in \bbR^{2d}\), its Jacobian $\nabla \Phi(z)$ satisfies
\begin{equation*}
    |\det(\nabla\Phi(z))| = 1~.
\end{equation*}
\end{definition}

The volume-preservation property is named as such since for any measurable set \(A \in \bbR^{2d}\) and \(\Phi\) that satisfies $|\det(\nabla\Phi(z))| = 1$, we have \(\mathrm{vol}(\Phi(A)) = \mathrm{vol}(A)\).
This follows by applying the change-of-variables formula, which gives:
\[
\mathrm{vol}(\Phi(A))
=\int_A \left|\det\bigl(\nabla \Phi(z_0)\bigr)\right|\,\rmd z_0
=\int_A \rmd z_0
=\mathrm{vol}(A).
\]

\begin{lemma}\label{lem:HF-symplectic}
For any \(t > 0\), the Hamiltonian flow map \((x, y) \mapsto \mathsf{HF}((x, y); t)\) associated with a twice continuously Hamiltonian \(H\) is both symplectic and volume-preserving.
\end{lemma}

\begin{proof}
For any \(z_{0} \in \bbR^{2d}\), define \(J(t) := \left.\nabla_{z} \HF(z; t) \right|_{z = z_{0}}\).
Since \(\dot{z}_t = \Omega \nabla H(z_t)\) along Hamiltonian flow \ref{eq:hf-apdx}, we have
\begin{equation*}
    \dot{J}(t) = \Omega \nabla^{2}H(z_t) J(t)~; \quad J(0) = \rmI_{2d}~. 
\end{equation*}

Define \(M(t) := J(t)^{\top}\Omega J(t)\).
By the product rule,
\begin{align*}
    \dot{M}(t) &= \dot{J}(t)^{\top}\Omega J(t) + J(t)^{\top}\Omega \dot{J}(t) \\
    &= J(t)^{\top} \nabla^{2}H(z_t)\Omega ^\top \Omega J(t) + J(t)^{\top}\Omega^{2}\nabla^{2} H(z_t) J(t)\\
    &=-J(t)^{\top} \nabla^{2}H(z_t)J(t) + J(t)^{\top}\nabla^{2} H(z_t) J(t)=0.
\end{align*}
In the third equality, we used the fact that \(\Omega^{2} = -\rmI_{2d}\) and \(\Omega^{\top}\Omega = \rmI_{2d}\).
In simpler terms, for \(M(t)\) is a constant function.
Consequently,
\begin{equation*}
    M(t) = M(0) \Rightarrow J(t)^{\top}\Omega J(t) = \Omega~.
\end{equation*}

Taking determinants on both sides and using the product rule results in
\begin{equation*}
    \det(J(t))^{2} = \det(\Omega)^{2} = 1~.
\end{equation*}
However since \(\det(J(0)) = 1\) and since \(t \mapsto J(t)\) is continuous, \(\det J(t) = 1\).
\end{proof}

Consequently, for any measurable set $A\subseteq \bbR^{2d}$ and for all $t > 0$:
\[
\mathrm{vol}\bigl(\mathsf{HF}(A;t)\bigr)=\mathrm{vol}(A).
\]

\paragraph{Time reversibility} 

\begin{lemma}\label{lem:hf-time-reversible}
Assume that the Hamiltonian \(H\) satisfies \(H(x, y) = H(x, -y)\) for all \(x, y\), and let $\mathsf{HF}(\cdot;t)$ denote the corresponding Hamiltonian flow map as defined by \ref{eq:hf-apdx}.
If \((x_{t}, y_{t}) = \mathsf{HF}((x_{0}, y_{0}); t)\) for any time \(t \geq 0\) and initial conditions \((x_{0}, y_{0})\), then
$\mathsf{HF}((x_t,-y_t);t)=(x_0,-y_0)$.
\end{lemma}
\begin{proof}
Let $(X_s,Y_s)_{s\ge 0}$ be the unique solution of \ref{eq:hf-apdx} with
$(X_0,Y_0)=(x_0,y_0)$, and suppose that $(X_t,Y_t)=(x_t,y_t)$ for some $t\ge 0$.
Define a new pair of curves on $[0,t]$ by
\[
\widetilde X_s := X_{t-s},
\qquad
\widetilde Y_s := -\,Y_{t-s}.
\]
Note that the derivative of an even function is odd.
The assumption about \(H\) states that for any \(x\), \(y \mapsto H(x, y)\) is an even function, and hence \(\nabla_{y}H(x, y) = -\nabla_{y}H(x, -y)\).
Therefore, by the chain rule,
\[
\dot{\widetilde X}_s
= -\dot X_{t-s}
= -\nabla_y H(X_{t-s},Y_{t-s})
= \nabla_y H(X_{t-s},-Y_{t-s})
= \nabla_y H(\widetilde X_s,\widetilde Y_s)~.
\]
Similarly,
\[
\dot{\widetilde Y}_s
= -(-\dot Y_{t-s})
= \dot Y_{t-s}
= -\nabla_x H(X_{t-s},Y_{t-s})
= -\nabla_x H(X_{t-s},-Y_{t-s})
= -\nabla_x H(\widetilde X_s,\widetilde Y_s)~.
\]
So $(\widetilde X_s,\widetilde Y_s)$ also satisfies the Hamiltonian dynamics
$\dot X_s=Y_s$, $\dot Y_s=-\nabla f(X_s)$ on $[0,t]$.
Moreover, $(\widetilde X_0,\widetilde Y_0)=(X_t,-Y_t)=(x_t,-y_t)$ and
$(\widetilde X_t,\widetilde Y_t)=(X_0,-Y_0)=(x_0,-y_0)$.
By uniqueness of solutions, the solution starting from $(x_t,-y_t)$ must equal
$(\widetilde X_s,\widetilde Y_s)$ on $[0,t]$, and in particular at time $t$ we obtain
$\mathsf{HF}((x_t,-y_t);t)=(x_0,-y_0)$.
\end{proof}
\subsection{Discussion of integrators}\label{app:sec:integrator-discuss}
As discussed in \Cref{sec:intro}, the Hamiltonian dynamics \ref{eq:hf-apdx} is exactly solvable only in special cases, for example, when $H(x, y)=x^\top Ax+y^\top By$. In general, we need to \textit{discretize} the dynamics. Here we discuss several elementary discretization schemes; see \citet[Chap. 2]{hairer2006geometric}, \citet[Chaps. 2, 4]{Leimkuhler_Reich_2005} for a more detailed exposition.

\subsubsection{Explicit and implicit integrators}

\begin{definition}\label{defn:exp_imp_apdx}
The \emph{explicit integrator} / forward Euler integrator of the Hamiltonian flow \emph{(\ref{eq:hf-apdx})} from initial values \((x_{0}, y_{0})\) with step size \(\eta > 0\) generates the sequence \(\{(x_{n}, y_{n})\}_{n \geq 0}\) where \((x_{n}, y_{n})\) satisfies the recursion
\begin{equation}
\label{eq:explicit-integ-gen}
\begin{aligned}
   x_{n+1}-x_n &= \eta \cdot \nabla_{y}H(x_n, y_n)~, \\
    y_{n+1}-y_n &= -\eta \cdot \nabla_{x}H(x_n, y_n)~.
    \end{aligned}
\end{equation}
More succinctly, using \(z_{n} := (x_{n}, y_{n})\),
\begin{equation*}
    z_{n + 1} - z_{n} = \eta \cdot \Omega \nabla H(z_{n})~.
\end{equation*}
\end{definition}

\begin{definition}
The \emph{implicit integrator} / backward Euler integrator of the Hamiltonian flow \emph{(\ref{eq:hf-apdx})} from initial values \((x_{0}, y_{0})\) with step size \(\eta > 0\) generates the sequence \(\{(x_{n}, y_{n})\}_{n \geq 0}\) where \((x_{n}, y_{n})\) satisfies the recursion
\begin{equation}
\label{eq:implicit-integ-gen}
\begin{aligned}
   x_{n+1}-x_n &= \eta \cdot \nabla_{y}H(x_{n+1}, y_{n+1})~, \\
    y_{n+1}-y_n &= -\eta \cdot \nabla_{x}H(x_{n+1}, y_{n+1})~.
    \end{aligned}
\end{equation}
More succinctly, using \(z_{n} := (x_{n}, y_{n})\),
\begin{equation*}
    z_{n + 1} - z_{n} = \eta \cdot \Omega \nabla H(z_{n + 1})~.
\end{equation*}
\end{definition}

\begin{remark}\label{rem:impl-prox-expl}
The explicit integrator is implementable whenever $\nabla H$ is accessible. The implicit integrator, however, requires solving a nonlinear system. Consider the special case $H(x,y)=f(x)+\frac12\|y\|^2$, for which
$\nabla_y H(x,y)=y$ and $\nabla_x H(x,y)=\nabla f(x)$.
Then the implicit updates read:
\begin{equation*}
\begin{aligned}
x_{n+1}-x_n &= \eta y_{n+1},\\
y_{n+1}-y_n &= -\eta \nabla f(x_{n+1})
\end{aligned}
\end{equation*}
which is precisely \emph{\ref{eq:implicit-integ}}.
Eliminating $y_{n+1}$ in the update for \(x_{n + 1}\) in \emph{\ref{eq:implicit-integ}} yields the implicit equation
\[
x_{n+1} = x_n + \eta y_n - \eta^2 \nabla f(x_{n+1}),
\]
or equivalently,
\[
x_{n+1} = \arg\min_{x\in\bbR^d}
\left\{ f(x) + \frac{1}{2\eta^2}\|x-(x_n+\eta y_n)\|^2 \right\}
= \mathrm{Prox}_{\eta^2 f}(x_n+\eta y_n)
\]
where \(\mathrm{Prox}_{\eta^{2}f}\) is the proximal operator for \(x \mapsto \eta^{2} f(x)\) \citep{parikh2014proximal}.
The update for $y_{n+1}$ is the same as in \emph{\ref{eq:implicit-integ}}.
Thus, when a proximal oracle for $f$ is available, the implicit integrator
admits a simple implementation
\[
x_{n+1}=\mathrm{Prox}_{\eta^2 f}(x_n+\eta y_n),
\qquad
y_{n+1}=y_n-\eta \nabla f(x_{n+1}).
\]
The extragradient integrator is formed by replacing the call to the proximal oracle for \(f\) by a gradient descent step. See \Cref{defn:extg} for more detail.
\end{remark}

\paragraph{Monotonicity properties for convex Hamiltonian.}
A key property about the explicit (resp.\ implicit) integrator is that it produces iterates along which the Hamiltonian values are monotonically increasing (resp.\ decreasing) when the Hamiltonian \(H\) is a convex function in the phase space.
We formalize this in following lemmas.
We note the following results only use the property that the Hamiltonian flow is a skew-gradient flow (i.e., the velocity vector field is orthogonal to the gradient); see also the discussion in~\citet[Appendix~B]{wibisono2022alternating}.

\begin{lemma}
\label{lem:fe-increase-H}
Assume that the Hamiltonian \(H\) is convex.
Suppose \((x_{n + 1}, y_{n + 1})\) is produced by one step of the explicit integrator \emph{(\cref{eq:explicit-integ-gen})} from \((x_{n}, y_{n})\) with step size \(\eta > 0\).
Then,
\begin{equation*}
    H(x_{n + 1}, y_{n + 1}) \geq H(x_{n}, y_{n})~.
\end{equation*}
\end{lemma}
\begin{proof}
We use the notation \(z_{n}\) and \(z_{n + 1}\) to denote the updates in phase space for convenience.
By convexity of \(H\) we have
\begin{align*}
    H(z_{n + 1}) &\geq H(z_{n}) + \langle \nabla H(z_{n}), z_{n + 1} - z_{n}\rangle \\
    &= H(z_{n}) + \eta \langle \nabla H(z_{n}), \Omega \nabla H(z_{n})\rangle \\
    &= H(z_{n})
\end{align*}
where the final step uses the fact that \(\langle v, \Omega v\rangle = 0\) since $\Omega$ is skew-symmetric.
\end{proof}

\begin{lemma}
\label{lem:hf-imp-nonincrease-ham}
Assume that the Hamiltonian \(H\) is convex.
Suppose \((x_{n + 1}, y_{n + 1})\) is produced by one step of the implicit integrator \emph{(\cref{eq:implicit-integ-gen})} from \((x_{n}, y_{n})\) with step size \(\eta > 0\).
Then,
\begin{equation*}
    H(x_{n + 1}, y_{n + 1}) \leq H(x_{n}, y_{n})~.
\end{equation*}
\end{lemma}
\begin{proof}
We use the notation \(z_{n}\) and \(z_{n + 1}\) to denote the updates in phase space for convenience.
By convexity of \(H\) we have
\begin{align*}
    H(z_{n + 1}) &\leq H(z_{n}) + \langle \nabla H(z_{n + 1}), z_{n + 1} - z_{n}\rangle \\
    &= H(z_{n}) + \eta \langle \nabla H(z_{n + 1}), \Omega \nabla H(z_{n + 1})\rangle \\
    &= H(z_{n})
\end{align*}
where the final step uses the fact that \(\langle v, \Omega v\rangle = 0\) since $\Omega$ is skew-symmetric.
\end{proof}

\subsubsection{Leapfrog integrator}

We describe another type of integrator called the leapfrog integrator.
This differs from the design of the explicit and implicit integrator due to interleaved updates for \(y\) and \(x\) as described below.

\begin{definition}\label{def:leapfrog}
The \emph{leapfrog integrator} of the Hamiltonian flow \emph{(\ref{eq:hf-apdx})} from initial values \((x_{0}, y_{0})\) with step size \(\eta > 0\) generates the sequence \(\{(x_{n}, y_{n})\}_{n \geq 0}\) where \((x_{n}, y_{n})\) satisfies the recursion
\begin{align*}
y_{n+\frac12} &= y_n - \frac{\eta}{2}\,\nabla_x H(x_n,y_n) \\
x_{n+1} &= x_n + \eta\,\nabla_y H(x_n, y_{n+\frac12}) \\
y_{n+1} &= y_{n+\frac12} - \frac{\eta}{2}\,\nabla_x H(x_{n+1},y_{n+\frac12}).
\end{align*}
\end{definition}
\begin{remark}
    When $H$ is separable, i.e. of the form $H(x, y)=f(x)+g(y)$,  the updates for \(y\) in the leapfrog integrator can be combined, and the resulting update rule is known as the symplectic Euler integrator for the separable Hamiltonian, which performs the update for the $x$ and $y$ variable in an alternating way:
    \begin{equation*}
    x_{n+1} = x_n + \eta\nabla_y H(x_n, y_{n})~, \quad
    y_{n+1} = y_{n} - \eta\nabla_x H(x_{n+1},y_{n})~.
    \end{equation*}
\end{remark}

While the explicit and implicit Euler schemes result in monotonic sequences of Hamiltonian values (\Cref{lem:fe-increase-H,lem:hf-imp-nonincrease-ham}), the Hamiltonian $H(x_n, y_n)$ along the iterates of the leapfrog integrator does \textit{not} change monotonically.
Instead, it enjoys a higher-order \emph{one-step conservation error of Hamiltonian}: under mild smoothness assumptions,
\[
\bigl|H(x_{n+1},y_{n+1})-H(x_n,y_n)\bigr|=\mathcal{O}(\eta^3)\qquad(\eta\to 0),
\]
whereas the explicit/implicit Euler schemes have a one-step defect of order $\mathcal{O}(\eta^2)$.
We refer readers to Section~3.3 of \cite{Bou_Rabee_2018} for a detailed discussion.

Another important property of leapfrog integrator is the preservation of symplecticity of Hamiltonian flow (see \Cref{lem:HF-symplectic}); the symplecticity of the leapfrog integrator can be found in~\citet[Theorem~4.1]{Bou_Rabee_2018}. In the continuous-time analysis of the present work, however, symplecticity does not play a direct role, and therefore we did not use leapfrog integrator in our discretization analysis. That said, a number of previous works on Hamiltonian optimization, including \cite{betancourt2018symplecticoptimization, Tran_2023_symplectic, Duruisseaux_2023_symplectic, lin2024simplifyingmomentumbasedpositivedefinitesubmanifold, zhu2026frictionless}, have shown that symplectic integrators can substantially improve the stability of the discrete trajectory, which in turn may lead to better convergence behavior. In the context of Hamiltonian Monte Carlo, the use of symplectic integrators is even more fundamental, as it is closely tied to the preservation of the target probability distribution. We refer to \cite{Bou_Rabee_2018} for a thorough discussion.
\section{A Review of Hamiltonian Monte Carlo for Sampling}\label{app:HMC}

We review the Hamiltonian Monte Carlo (\textsf{HMC}) framework for sampling~\citep{duane1987hybrid, neal2011mcmc}. Given a function $f \colon \R^d \to \R$, the goal in sampling is to draw approximate samples from the target distribution $\nu(x) \propto e^{-f(x)}$ on $\R^d$. We denote the set of probability distributions on $\R^d$ by $\mathcal{P}(\R^d)$. Assuming that the Hamiltonian dynamics~\ref{eq:HamFlow} can be solved exactly, the exact \textsf{HMC} algorithm to sample from $\nu$ is defined in \nameref{alg:eHMC} (\Cref{alg:eHMC}).

\begin{algorithm}[h]
\DontPrintSemicolon
\caption{Exact Hamiltonian Monte Carlo (\textsf{eHMC})}
\algotitle{\textsf{eHMC}}{alg:eHMC}
\SetKwInOut{Input}{Input}\SetKwInOut{Output}{Output}
\SetAlgoLined
\Input{Initial distribution $\mu \in \mathcal{P}(\R^d)$, Integration time $T > 0$, Number of iterations $K \in \mathbb{N}$}
Sample $X_0 \sim \mu$

\For{\(k = 1\) \KwTo \(K\)}{
    Sample $\xi \sim \mathcal{N}(0,I)$

    Define $X_k \coloneqq \Pi_1 \circ \HF((X_{k-1}, \xi); T)$
    }
\Return $X_K$
\end{algorithm}

The difference between \nameref{alg:eHMC} for sampling and~\ref{eq:HF-opt-update} for optimization is that for sampling, the velocity is periodically refreshed to be an independent draw from a standard Gaussian $\xi \sim \mathcal{N}(0, \mathrm{I}_d)$, instead of deterministically set to $\bz \in \R^d$ for optimization.

Since \ref{eq:HamFlow} can be implemented exactly only in special cases, a discretization scheme of the Hamiltonian dynamics has to be utilized, leading to the unadjusted HMC algorithm~\citep{bou2023mixing, bou2026tail}; several discretization schemes of the Hamiltonian dynamics are discussed in~\Cref{app:HF_Background} with a popular method being the St\"ormer--Verlet or leapfrog integrator.

Define $\pi = \nu \otimes \mathcal{N}(0,I) \in \mathcal{P}(\R^{2d})$. Therefore, $\pi$ is a distribution defined on the phase space and note that 
\[
\pi(x,y) \propto e^{-H(x,y)}
\]
where the Hamiltonian energy function $H: \R^{2d} \to \R$ is defined in Section~\ref{sec:ham-dynamics-def}.
To see that $\nu$ is in fact the stationary distribution for \nameref{alg:eHMC}, we will show that the Hamiltonian flow~\eqref{eq:HamFlow} preserves $\pi$, i.e., for any $t \geq 0$
\[
\HF(\,\cdot\,;\, t)_\# \pi = \pi\,.
\]
Let $\tilde \pi = \HF(\,\cdot\,;\, t)_\# \pi$. Then, by the change of variable formula, we have that for any $(x,y) \in \R^{2d}$,
\[
\pi(x,y) = \tilde \pi (\HF((x,y); t))\, |\det (\nabla \HF ((x,y); t))| = \tilde \pi (\HF((x,y); t))
\]
where the last equality is because the Hamiltonian flow preserves volume; see~\Cref{app:HF_Background}. Thus,
\[
\tilde \pi (\HF((x,y); t)) = \pi(x,y) = Z^{-1} e^{-H(x,y)} = Z^{-1} e^{-H(\HF((x,y); t))}
\]
where $Z$ is a normalizing constant and the final equality is due to conservation of the Hamiltonian along the Hamiltonian flow (see~\Cref{app:HF_Background}). Hence $\tilde \pi = \pi$, and therefore, the position marginal $\nu$ is stationary for \nameref{alg:eHMC}, when the velocity marginal is $\mathcal{N}(0, \mathrm{I}_d)$ and is independent with $\nu$.

Further properties of \nameref{alg:eHMC} can be found in~\citet[Section~A]{bou2026tail} and a discussion on how \nameref{alg:eHMC} and~\ref{eq:HF-opt-update} fit into a general Lift-Conserve-Project (LCP) scheme can be found in~\citet[Section~B]{fu2025hamiltoniandescentalgorithmsoptimization}. For a discussion on prior works studying \nameref{alg:eHMC} and unadjusted HMC, see e.g.,~\citet[Section~1.3]{bou2026tail}.
\section{Lower Bound for \ref{eq:HF-opt-update}}
\label{app:LowerBound}

\begin{lemma}\label{lem:LowerBound}
    Fix $c > 0$. For every $\kappa \geq \max \{1, c^2\}$, there exists an $\alpha$-strongly convex and $L = \kappa \alpha$-smooth quadratic function
    \[
    f(x) = \frac{\alpha}{2}x_1^2 + \frac{L}{2}x_2^2\,,\qquad x = (x_1, x_2)^\top\,,
    \]
    such that the algorithmic update~\eqref{eq:HF-opt-update} with integration time $T = \frac{c}{\sqrt{L}}$, initialized at $x_0 = (x_{0,1},0)$, requires $k = \Omega(\kappa \log (1/\varepsilon))$ many iterations to ensure that $f(x_k)-f(x^\star) \leq \varepsilon$\,. The total integration time $T \times k$ is therefore $\Omega(\frac{c\sqrt{L}}{\alpha}\log(1/\varepsilon))$.
\end{lemma}

\begin{proof}
As $f$ is separable, the dynamics~\eqref{eq:HamFlow} decouple along both coordinates and are given by
\[
\ddot{X}_t = -\alpha\, X_t\,, \qquad \ddot{Y}_t = -L\,Y_t\,.
\]
As the initial velocity in Step~(ii)~of~\ref{eq:HF-opt-update} is $\bz \in \R^2$, we have that $X'_0 = Y'_0 = 0$.
Note that in the proof of this lemma, $X,Y \in \R$ and denote the coordinates of the position space $\R^2$. 

For any $k \in \mathbb{N}$, denote the $k$th iterate of the algorithm~\eqref{eq:HF-opt-update} by $x_k = (x_{k,1}, x_{k,2}) \in \R^2$.
Given integration time $T > 0$, the update is therefore
\[
x_{k+1, 1} = x_{k,1} \cos(\sqrt{\alpha}\,T), \qquad x_{k+1, 2} = x_{k,2} \cos (\sqrt{L}\, T)\,,
\]
which can be rewritten as
\[
x_{k+1, 1} = x_{k,1} \cos\left(\frac{c}{\sqrt{\kappa}}\right), \qquad x_{k+1, 2} = x_{k,2} \cos (c)\,.
\]
Considering the initialization $x_0 = (a,0)$, we get that $x_{k,2} \equiv 0$ for all $k \geq 0$ and that
\[
x_{k,1} = x_{0,1} \cdot \cos^k\left(\frac{c}{\sqrt{\kappa}}\right)~.
\]
Therefore,
\[
f(x_k) - f(x^\star) = \frac{\alpha}{2} x_{k,1}^2 = \frac{\alpha}{2} x_{0,1}^2 \cos^{2k} \left(\frac{c}{\sqrt{\kappa}}\right) = \cos^{2k} \left(\frac{c}{\sqrt{\kappa}}\right) (f(x_0) - f(x^\star))\,.
\]
As $\kappa \geq c^2$, we have that $0 \leq \frac{c}{\sqrt{\kappa}} \leq 1$. Therefore, using 
\[
\cos(u) \geq 1-\frac{u^2}{2}
\]
for $0\leq u \leq1$ implies
\[
f(x_k)- f(x^\star) \geq \left(1-\frac{c^2}{2\kappa} \right)^{2k} (f(x_0) - f(x^\star))\,,
\]
which implies the claimed iteration complexity.
\end{proof}
\section{Extended discussion of \texorpdfstring{\Cref{sec:hf_convexity}}{Section 3}}

\subsection{Examples}

\subsubsection{A Quadratic Setting}
\label{app:sec:example-hf-quad}

As an example, we consider a quadratic objective function and illustrate how the final endpoint of Hamiltonian flow trajectories compares against averaged trajectories.
We take the objective function to be \(f(x) = \frac{1}{2}x^{\top}\Lambda x\) where \(\Lambda\) is the \(2 \times 2\) diagonal matrix $\Lambda = \begin{pmatrix}
    2 & 0 \\
    0 & 10
\end{pmatrix}$.
In this setting, the solution to the Hamiltonian dynamics is available in closed form.
We set the initial position to be \(X_{0} = [1, 1]^{\top}\), integration time to be \(T = 4\), and plot the three trajectories:
\begin{itemize}[itemsep=0pt, leftmargin=*]
    \item the position \((X_{t})_{t \in [0, T]}\)\,, 
    \item the aggregated position \((X^{\mathrm{avg}}(X_{0}; t))_{t \in [0, T]}\) as defined in \cref{eq:WeightedAvg}\,,
    \item the aggregated position
    \((X^{\mathrm{s-avg}}(X_{0}; t))_{t \in [0, T]}\) according to \(X^{\mathrm{s-avg}}(X_{0}; t) = \frac{1}{t}\int_{0}^{t} X_{s}\,\rmd s\)\,.
\end{itemize}

\begin{figure}[H]
\centering
\includegraphics[width=1.0\linewidth]{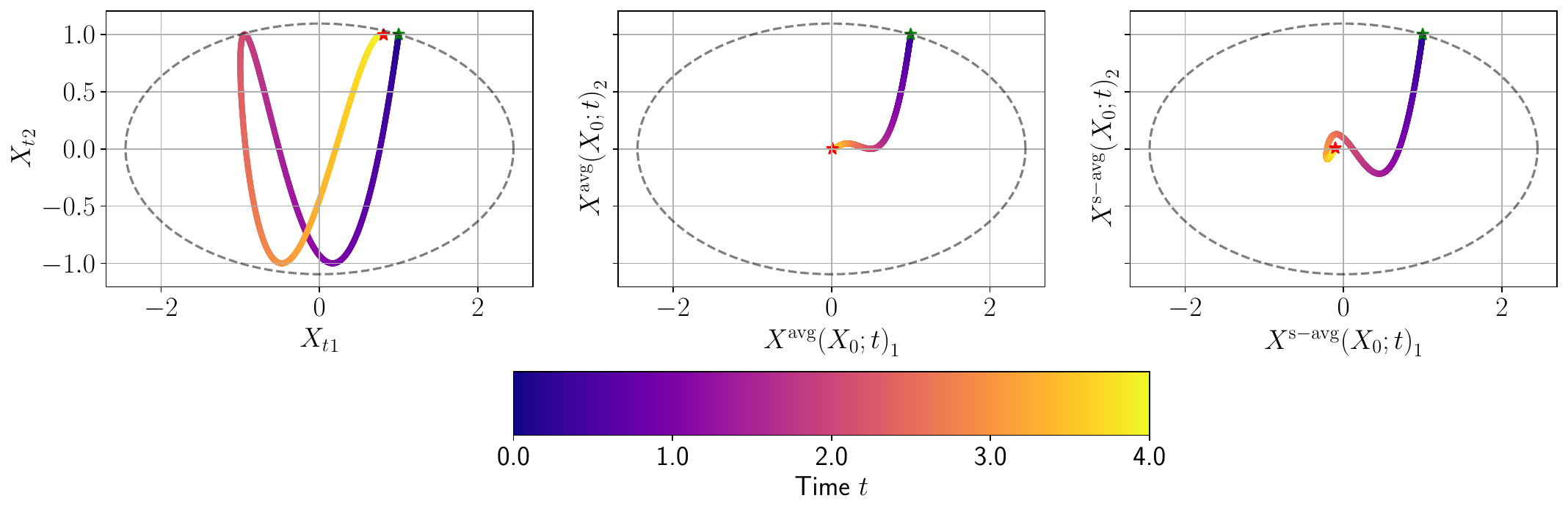}
\caption{The dotted lines represent the level set of \(f\), and green and red stars highlight the position at time \(t = 0\) and \(t = T\) respectively.}
\label{fig:agg-variation}
\end{figure}

The minimizer of \(f\) is at \([0, 0]^{\top}\) and \cref{fig:agg-variation} shows that the position \(X_{t}\) exhibits oscillatory behavior whereas the averages \(X^{\mathrm{s-avg}}(X_{0}; t)\) and \(X^{\mathrm{avg}}(X_{0}; t)\) approach the minimizer.
In \cref{fig:func-variation} we plot the function values corresponding to the positions plotted in \cref{fig:agg-variation}, and also plot the following function values:
\begin{itemize}[itemsep=0pt, leftmargin=*]
    \item \(f^{\mathrm{avg}}(X_{0}; t) := \frac{2}{t^{2}}\int_{0}^{t} (t - s) f(X_{s}) \rmd s\)\,,
    \item \(f^{\mathrm{s-avg}}(X_{0}; t) := \frac{1}{t}\int_{0}^{t} f(X_{s}) \rmd s\)\,.
\end{itemize}

\begin{minipage}{0.55\linewidth}
\begin{figure}[H]
  \centering
    \includegraphics[width=\textwidth]{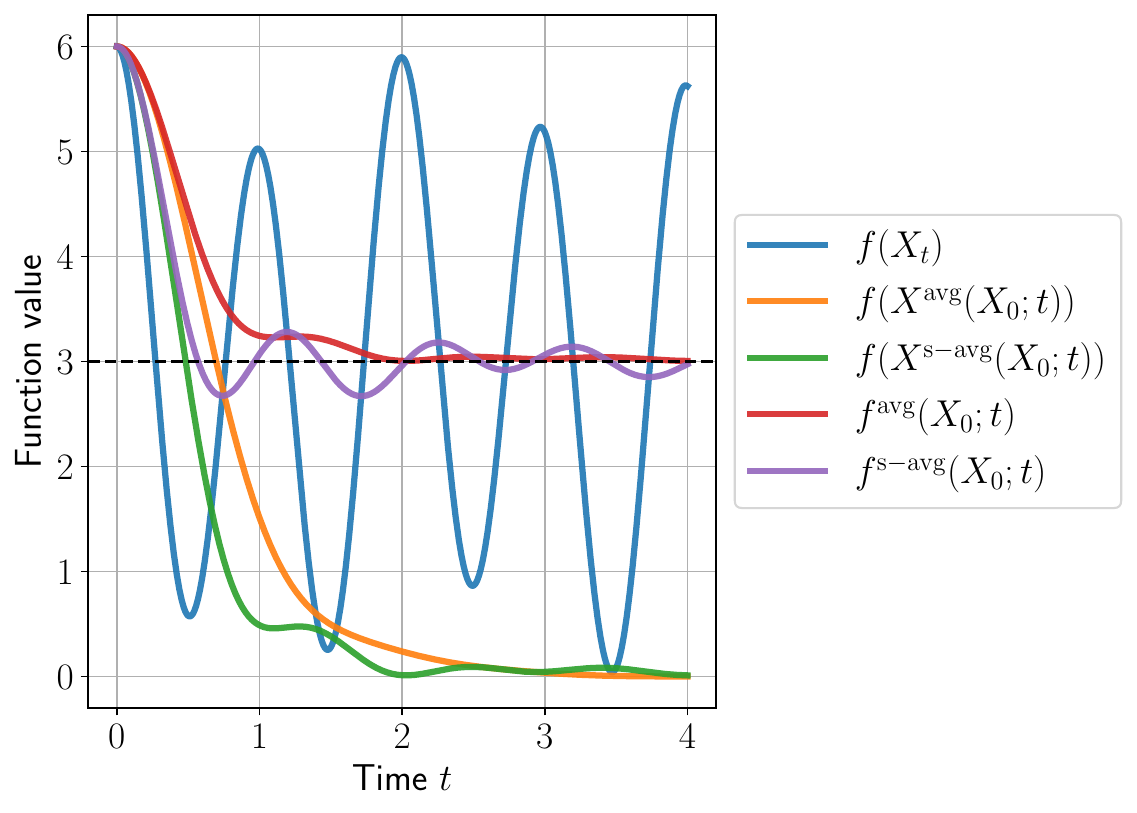}
  \caption{Evolution of function values with time \(t\).}
  \label{fig:func-variation}
\end{figure}
\end{minipage}
\hfill
\begin{minipage}{0.43\linewidth}
We plot \(f^{\mathrm{avg}}\) and \(f^{\mathrm{s-avg}}\) because the former is the quantity that we bound; see \Cref{rmk:stronger}, and the latter is the function value analogue of \(X^{\mathrm{s-avg}}\).
We note two key observations from this plot: (a) the function values along the aggregated position decay to \(0\), and (b) \(f^{\mathrm{avg}}(X_{0}; t)\), \(f^{\mathrm{s-avg}}(X_{0}; t)\) oscillates about \(\frac{1}{2}f(X_{0})\) with varying amplitudes.
The first observation speaks to the nature of oscillatory behaviour of \(X_{t}\) shown in \cref{fig:agg-variation}, and the second observation suggests our analysis might be able to be tightened, but it would not qualitatively affect the convergence rates.
\end{minipage}

\subsubsection{A Strictly Convex Setting}

In this example, we consider a 1 dimensional function \(f(x) = e^{x} - x - 1\).
This is a strictly convex function on account of \(\frac{\rmd^{2}}{\rmd x^{2}} f(x) = e^{x} > 0\).
This instructive example is meant to demonstrate that (a) the time averages of the position does not necessarily converge to the minimizer which is at \(0\) with \(f^{\star} = 0\), and (b) the optimality gap of the function averages \(f^{\mathrm{avg}}(X_{0}; t)\) and \(f^{\mathrm{s-avg}}(X_{0}; t)\) do not necessarily oscillate about \(\frac{1}{2}(f(X_{0}) - f^{\star})\).

Here, we do not have access to a closed form solution for \ref{eq:HamFlow}, and instead use a numerical integrator -- specifically the DOP853 integrator of Dormand and Prince as implemented in the \texttt{dop853} integrator in \texttt{SciPy} \citep{virtanen2020scipy}.

\begin{minipage}{0.4\linewidth}
We initialize \(X_{0} = 1\), and set \(T= 20\).
We plot same three trajectories as before, but the aggregated position are approximated due to the unavailability of closed form solutions.
Specifically, for \(\eta = 10^{-3}\) and \(k \in \{0, \ldots, \nicefrac{T}{\eta}\}\)
\begin{align*}
    X^{\mathrm{avg}}(X_{0}; k\eta) &\approx \frac{2}{k(k + 1)}\sum_{i=0}^{k}\sum_{j=0}^{i}X_{j\eta}~, \\
    X^{\mathrm{s-avg}}(X_{0}; k\eta) &\approx \frac{1}{k + 1}\sum_{i=0}^{k}X_{i\eta}~.
\end{align*}
\end{minipage}
\hfill
\begin{minipage}{0.55\linewidth}
\begin{figure}[H]
\centering
\includegraphics[width=\linewidth]{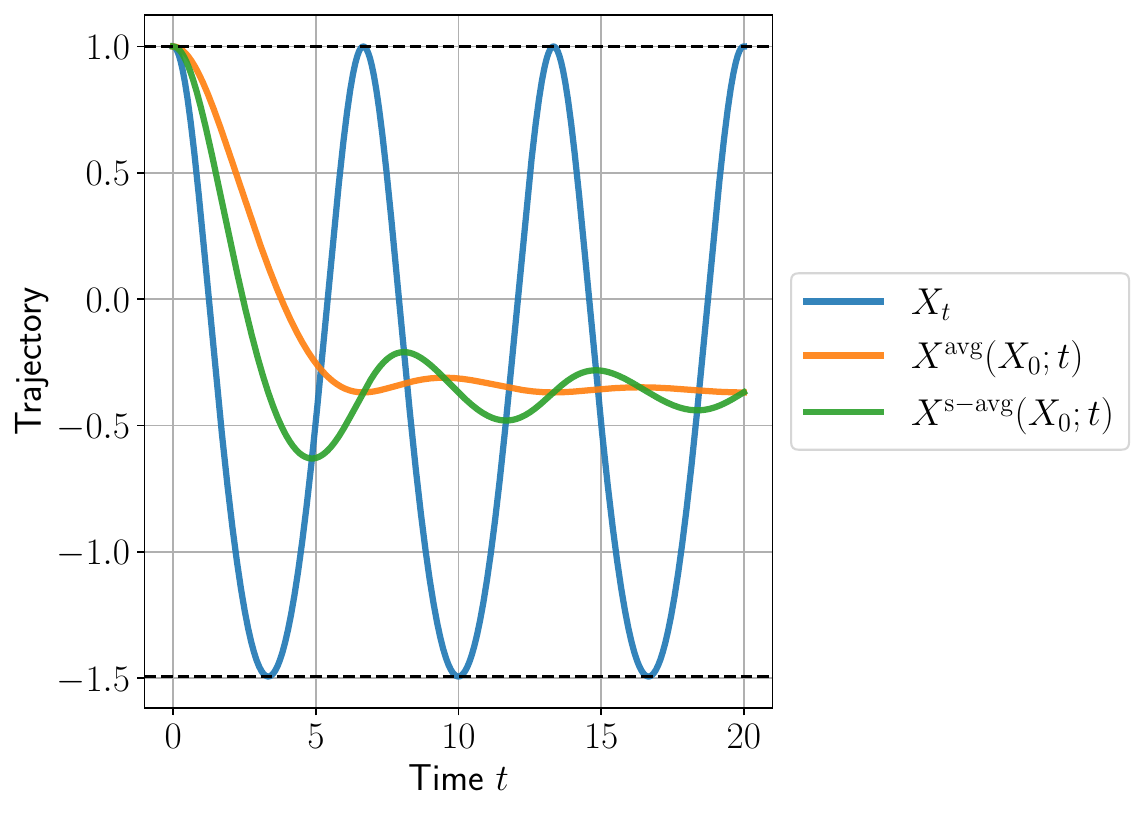}
\caption{The dotted lines are the boundaries of the sublevel set of \(f\) corresponding to \(f(x) \leq f(X_{0})\).}
\label{fig:agg-variation-2}
\end{figure}
\end{minipage}

In \Cref{fig:agg-variation-2}, we observe a similar oscillatory behavior for \(X_{t}\), while the averaged positions oscillate less so.
Notably, these do not oscillate close to \(0\), which we recall is the minimizer of this function.
We also plot the function values associated with these trajectories, and \(f^{\mathrm{avg}}\) and \(f^{\mathrm{s-avg}}\) as described in the previous subsection -- see \Cref{fig:func-variation-2}.
Again, due to the unavailability of closed form solution, we approximate the integrals in \(f^{\mathrm{avg}}\) and \(f^{\mathrm{s-avg}}\) as a finite-sum in the same manner as done for the approximations of the aggregated positions.

\begin{minipage}{0.55\linewidth}
\begin{figure}[H]
  \centering
    \includegraphics[width=\linewidth]{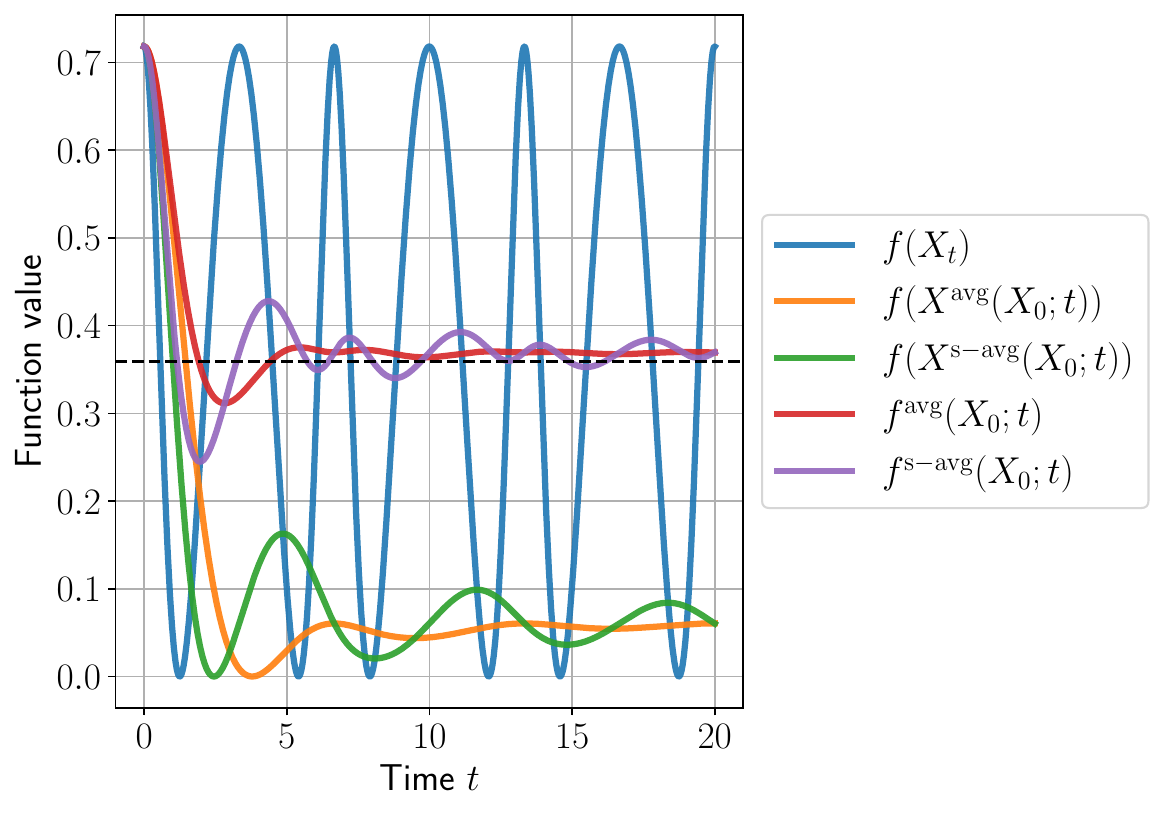}
  \caption{Evolution of function values with time \(t\).}
  \label{fig:func-variation-2}
\end{figure}
\end{minipage}\hfill
\begin{minipage}{0.4\linewidth}
Here, the dotted line represents \(\frac{1}{2}f(X_{0})\).
We see that the time averages of the function values \(f^{\mathrm{avg}}(X_{0}; t)\) and \(f^{\mathrm{s-avg}}(X_{0}; t)\) oscillates about a value that is slightly higher than this line with varying amplitudes.
This does not invalidate the result of \Cref{cor:hf-convex-main} which suggests the level \(\frac{2}{3}f(X_{0}) \approx 0.48\).
Consistent with the variation shows in \Cref{fig:agg-variation-2}, we see that the function values of the aggregated positions do not converge to \(0\), which is in contrast to what was observed in the quadratic setting (\Cref{fig:func-variation}).
\end{minipage}

\subsection{Proofs of main lemmas in \texorpdfstring{\Cref{sec:hf_convexity}}{Section 3}}

\subsubsection{Proof of \texorpdfstring{\Cref{cor:hf-convex-main}}{Corollary 1}}
\label{prf:cor:hf-convex-main}

To prove~\Cref{cor:hf-convex-main}, we first state and prove the following more general lemma.

\begin{lemma}
\label{lem:hf-convex-main-gen}
Let \(f\) satisfy \emph{\convexassump{}}.
Consider a non-negative weighting function \(\bm{w}(t)\) over \([0, T]\) such that \(\int_{0}^{T} \bm{w}(t) \rmd t = 1\), and define the \(\bm{w}\)-weighted average of \((\HF((X_{0}, \bm{0}); \tau))_{\tau \in [0, T]}\) as
\begin{equation*}
    X^{\bm{w}}(X_{0}; T) := \int_{0}^{T} \bm{w}(t) X_{t}\rmd t~.
\end{equation*}
Then,
\begin{align*}
    f(X^{\bm{w}}(X_{0}; T)) - f(z) &\leq \int_{0}^{T} \bm{w}(t)(f(X_{t}) - f(z)) \rmd t \\
    &\leq \frac{2}{3}(f(X_{0}) - f(z)) - \frac{1}{3}\bm{w}(T)\langle X_{T} - z, Y_{T}\rangle + \frac{1}{3}\int_{0}^{T} \dot{\bm{w}}(t)\langle X_{t} - z, Y_{t}\rangle\rmd t~.
\end{align*}
\end{lemma}

\begin{proof}
Recall that $h(t) = \langle X_t - z, Y_t\rangle$ and $h(0) = 0$. Multiplying both sides of the result of \Cref{lem:hf-convex-main} by \(\bm{w}(t)\) which is non-negative, we obtain
\begin{equation*}
    \bm{w}(t) (f(X_{t}) - f(z)) \leq \frac{2}{3} \bm{w}(t) (f(X_{0}) - f(z)) - \frac{1}{3}\bm{w}(t) \dot{h}(t)~.
\end{equation*}
Integrating both sides from \(0\) to \(T\), we get the following:

\textbf{LHS:}~ By convexity of $f$,
\begin{equation*}
    \int_{0}^{T} \bm{w}(t) (f(X_{t}) - f(z)) \, \rmd t \geq f\left(\int_{0}^{T}\bm{w}(t)X_{t} \, \rmd t\right) - f(z) = f(X^{\bm{w}}(X_{0}; T)) - f(z)~.
\end{equation*}
The inequality is due to Jensen's inequality as \(\int_{0}^{T}\bm{w}(t)\rmd t= 1\).

\textbf{RHS:}~~
Using integration-by-parts,
\begin{equation*}
    -\frac{1}{3}\int_{0}^{T}\bm{w}(t) \dot{h}(t) \rmd t = -\frac{1}{3}\left\{w(T)h(T) - w(0)h(0)\right\} + \frac{1}{3}\int_{0}^{T}\dot{\bm{w}}(t)h(t) \rmd t~.
\end{equation*}
This completes the proof.
\end{proof}

We now provide the proof of~\Cref{cor:hf-convex-main}.

\medskip

\begin{proofof}{\Cref{cor:hf-convex-main}}~~
As defined in \cref{eq:WeightedAvg}, \(X^{\mathrm{avg}}(X_{0}; T)\) is the \(\bm{w}\)-weighted average with weight \(\bm{w}(t) = \frac{2}{T^{2}}(T- t)\), so $\dot{\bm{w}}(t) = -\frac{2}{T^2}$ and $\bm{w}(T) = 0$.
Instantiating \Cref{lem:hf-convex-main-gen} with this weight gives
\begin{equation*}
    f(X^{\mathrm{avg}}(X_{0}; T)) - f(z) \leq \frac{2}{3}(f(X_{0}) - f(z)) - \frac{2}{3T^{2}}\int_{0}^{T} \langle X_{t} - z, Y_{t}\rangle \rmd t~.
\end{equation*}
Since \(\frac{\rmd}{\rmd t} \|X_{t} - z\|^{2} = 2\langle X_{t} - z, Y_{t}\rangle\), we can simplify the above to
\begin{equation*}
    f(X^{\mathrm{avg}}(X_{0}; T)) - f(z) \leq \frac{2}{3}(f(X_{0}) - f(z)) - \frac{1}{3T^{2}}\|X_{T} - z\|^{2} + \frac{1}{3T^{2}}\|X_{0} - z\|^{2}~
\end{equation*}
which proves the corollary.
\end{proofof}

\subsubsection{Versions of \texorpdfstring{\Cref{lem:hf-convex-main}}{Lemma 2} for the average and minimizer over the trajectory}
\label{app:hf-convex-main-other-weights}
\begin{corollary}
\label{cor:hf-convex-main-uniform-avg}
Consider the setting of \Cref{lem:hf-convex-main}.
Let \(X^{\mathrm{s-avg}}(X_{0}; T)\) be the simple average of \(X_{t}\) along the Hamiltonian flow trajectory \((\HF((X_{0}, \bm{0}); \tau))_{\tau \in [0, T]}\). 
Then,
\begin{equation*}
    f(X^{\mathrm{s-avg}}(X_{0}; T)) - f(x^{\star}) \leq \frac{2}{3}(f(X_{0}) - f(x^{\star})) + \frac{1}{3T} \|X_{T} - x^{\star}\| \cdot \sqrt{2(f(X_{0}) - f(X_{T}))}~.
\end{equation*}
Moreover, when \(f\) satisfies \emph{\quadgrowassump{}},
\begin{equation*}
    f(X^{\mathrm{s-avg}}(X_{0}; T)) - f(\xstar) \leq \left(\frac{2}{3} + \frac{2}{3T\sqrt{\alpha}}\right)(f(X_{0}) - f(\xstar))~.
\end{equation*}
\end{corollary}
\begin{proof}
    We instantiate \Cref{lem:hf-convex-main-gen} with \(\bm{w}(t) = \frac{1}{T}\).
    This gives
    \begin{equation*}
        f(X^{\mathrm{s-avg}}(X_{0}; T)) - f(z) \leq \frac{2}{3}(f(X_{0}) - f(z)) - \frac{1}{3T}\langle X_{T} - z, Y_{T}\rangle~.
    \end{equation*}
    The inner product can be bounded from above using Cauchy-Schwarz inequality as
    \begin{equation*}
        -\langle X_{T} - z, Y_{T}\rangle \leq \|X_{T} - Z\| \cdot \|Y_{T}\| \leq \|X_{T} - z\| \cdot \sqrt{2(f(X_{0}) - f(X_{T}))}~.
    \end{equation*}
    where the last inequality uses the energy conservation along the Hamiltonian flow trajectory to bound \(\|Y_{T}\|\).

    When \(f\) satisfies \quadgrowassump{}, using the energy conservation along the Hamiltonian flow trajectory, we have
    \begin{equation*}
        \|X_{T} - \xstar\| \leq \sqrt{\frac{2(f(X_{T}) - f(\xstar))}{\alpha}} \leq \sqrt{\frac{2(f(X_{0}) - f(\xstar))}{\alpha}}~.
    \end{equation*}
    Substituting this in the first statement of the corollary along with \(f(\xstar) \leq f(X_{T})\) by definition yields
    \begin{equation*}
        f(X^{\mathrm{s-avg}}(X_{0}; T)) - f(\xstar) \leq \frac{2}{3}(f(X_{0}) - f(\xstar)) + \frac{2}{3T\sqrt{\alpha}}(f(X_{0}) - f(\xstar))~.
    \end{equation*}
\end{proof}

The corollary for the minimizer along the trajectory is simpler to obtain owing to the fact that the minimum is at most any weighted average.

\begin{corollary}\label{cor:hf-convex-main-argmin}
Consider the setting of \Cref{lem:hf-convex-main}.
Let \(X^{\mathrm{min}}(X_{0}; T)\) denote the minimizer of \(f\) along the Hamiltonian flow trajectory \((\HF((X_{0}, \bm{0}); \tau))_{\tau \in [0, T]}\), i.e., \(f(X^{\mathrm{min}}(X_{0}; T)) \leq f(X_{t})\) for all \(t \in [0, T]\).
Then,
\begin{equation*}
    f(X^{\mathrm{min}}(X_{0}; T)) - f(x^{\star}) + \frac{1}{3T^{2}}\|X_{T} - x^{\star}\|^{2} \leq \frac{2}{3}(f(X_{0}) - f(x^{\star})) + \frac{1}{3T^{2}}\|X_{0} - x^{\star}\|^{2}~.
\end{equation*}
Moreover, when \(f\) is satisfies \emph{\quadgrowassump{}},
\begin{equation*}
    f(X^{\min}(X_{0}; T)) - f(\xstar) \leq \left(\frac{2}{3} + \frac{2}{3T^{2}\alpha}\right)(f(X_{0}) - f(\xstar))~.
\end{equation*}
\end{corollary}
\begin{proof}
    We use the following corollary of \Cref{lem:hf-convex-main}, which we obtain from integrating both sides of \cref{eq:hf-convex-weighted-avg} twice:
    $$
    \frac{2}{T^{2}}\int_{0}^{T} (T- t) (f(X_{t}) - f(z)) \rmd t + \frac{1}{3T^{2}} \|X_{T} - z\|^{2} \leq \frac{2}{3} (f(X_{0}) - f(z)) + \frac{1}{3T^{2}} \|X_{0} - z\|^{2}~.
    $$
    Since $f(X^{\min}(X_{0}; T)) \leq f(X_{t})$, letting $z=\xstar$ gives:
    $$
    \frac{2}{T^{2}}\int_{0}^{T} (T- t) (f(X^{\min}(X_{0}; T)) - f(\xstar)) \rmd t + \frac{1}{3T^{2}} \|X_{T} - \xstar\|^{2} \leq \frac{2}{3} (f(X_{0}) - f(z)) + \frac{1}{3T^{2}} \|X_{0} - \xstar\|^{2}~.
    $$
    Since $\int_0^T(T-t)\rmd t = \frac{T^2}{2}$, this further simplifies into:
    $$
    f(X^{\mathrm{min}}(X_{0}; T)) - f(x^{\star}) + \frac{1}{3T^{2}}\|X_{T} - x^{\star}\|^{2} \leq \frac{2}{3}(f(X_{0}) - f(x^{\star})) + \frac{1}{3T^{2}}\|X_{0} - x^{\star}\|^{2}~.
    $$
    When \(f\) is satisfies \emph{\quadgrowassump{}}, $\|X_0-\xstar\|^2\leq \frac{2}{\alpha}(f(X_0)-f(\xstar))$. Applying this bound and dropping the $\frac{1}{3T^2}\|X_T-\xstar\|^2$-term on the left gives:
    $$
    f(X^{\min}(X_{0}; T)) - f(\xstar) \leq \left(\frac{2}{3} + \frac{2}{3T^{2}\alpha}\right)(f(X_{0}) - f(\xstar))~.
    $$
\end{proof}

\subsubsection{Proof of \texorpdfstring{\Cref{thm:hf-cts-strong-cvx} and \Cref{cor:hf-cts-strong-cvx}}{Theorem 1 and Corollary 2}}
\label{app:prf:hf-cts-strong-cvx}

\begin{proofof}{\Cref{thm:hf-cts-strong-cvx}}
Notice that at each iteration \(k\), \(X_{k} = X^{\mathrm{avg}}(X_{k - 1}; T_{k})\).
Therefore, from \Cref{lem:hf-strong-convex-main} and the lower bound for \(T_{k}\), we have
\begin{equation*}
    f(X_{k}) - f(x^{\star}) \leq \frac{2}{3}\left(1 + \frac{1}{C^{2}}\right)(f(X_{k - 1}) - f(x^{\star}))~.
\end{equation*}
Iterating this inequality, we obtain the statement of the lemma.
\end{proofof}

\begin{proofof}{\Cref{cor:hf-cts-strong-cvx}}
    From \Cref{thm:hf-cts-strong-cvx}, we have for \(T_{k} = T = \frac{5}{2\sqrt{\alpha}}\) that
    \begin{equation*}
        f(X_{K}) - f(x^{\star}) \leq \left(\frac{58}{75}\right)^{K} \cdot (f(X_{0}) - f(\xstar))~.
    \end{equation*}
    Since \(K \geq 4\log \frac{f(X_{0}) - f(x^{\star})}{\varepsilon} \geq \left(\log \frac{75}{58}\right)^{-1} \log \frac{f(X_{0}) - f(x^{\star})}{\varepsilon}\),
    \begin{equation*}
        f(X_{K}) - f(\xstar) \leq \varepsilon~.
    \end{equation*}
    The total integration time is \(\sum_{k=1}^{K} T_{k} =  \left\lceil4 \log \frac{f(X_{0}) - f(x^{\star})}{\varepsilon}\right\rceil    \frac{5}{2\sqrt{\alpha}}\).
\end{proofof}

\subsubsection{Proof of \texorpdfstring{\Cref{lem:hf-strong-convex-main}}{Lemma 3}}
\label{app:prf:hf-strong-convex-main}
\begin{proof}
From \Cref{cor:hf-convex-main}, we have
\begin{equation*}
    f(X^{\mathrm{avg}}(X_{0}; T)) - f(x^{\star}) \leq \frac{2}{3}(f(X_{0})- f(x^{\star})) + \frac{1}{3T^{2}}\|X_{0} - x^{\star}\|^{2}~.
\end{equation*}
Since \(f\) satisfies \quadgrowassump{}, we have that \(\|X_{0} - x^{\star}\|^{2} \leq \frac{2}{\alpha}(f(X_{0}) - f(x^{\star}))\)~.
Substituting this in the inequality above results in the statement of the lemma.
\end{proof}

\subsubsection{Proof of \texorpdfstring{\Cref{lem:hf-convex-mix-main}}{Lemma 4}}
\label{prf:lem:hf-convex-mix-main}

\begin{proof}
\Cref{cor:hf-convex-main} and the energy conservation of the Hamiltonian dynamics guarantee \cref{eq:cor:hf-convex-main,eq:hf-conserve} respectively.
It remains to prove \cref{eq:cor:hf-convex-main-2}.
\Cref{lem:hf-convex-main} states that for \(h(t) := \langle X_{t} - z, Y_{t}\rangle\),
\begin{equation*}
    \dot{h}(t) \leq -3(f(X_{t}) - f(z)) + 2(f(X_{0}) - f(z))~.
\end{equation*}
When \(z \leftarrow \xstar\), we have $f(X_t) - f(x^*) \ge 0$, and thus from the above, \(\dot{h}(t) \leq 2(f(X_{0}) - f(\xstar))\).
Integrating both sides gives
\begin{equation*}
    h(t) - h(0) \leq 2t(f(X_{0}) - f(\xstar))~.
\end{equation*}
Since \(h(0) = 0\) and \(h(t) = \frac{1}{2} \frac{\rmd}{\rmd t}\|X_{t} - \xstar\|^{2}\), integrating again yields
\begin{equation*}
    \frac{1}{2}\|X_{t} - \xstar\|^{2} - \frac{1}{2}\|X_{0} - \xstar\|^{2} \leq t^{2} \cdot (f(X_{0}) - f(\xstar))~.
\end{equation*}
Multiplying both sides by \(\frac{2}{T^{2}}(T- t)\) and integrating from \(t = 0\) to \(t = T\) gives
\begin{equation*}
    \frac{1}{2}\int_{0}^{T} \frac{2(T- t)}{T^{2}} \|X_{t} - \xstar\|^{2} \rmd t - \frac{1}{2}\|X_{0} - \xstar\|^{2} \leq \frac{T^{2}}{6} (f(X_{0}) - f(x^{\star}))~.
\end{equation*}
By Jensen's inequality,
\begin{equation*}
    \frac{1}{2}\|X^{\mathrm{avg}}(X_{0}; T) - \xstar\|^{2} \leq \frac{1}{2}\int_{0}^{T} \frac{2(T- t)}{T^{2}} \|X_{t} - \xstar\|^{2} \rmd t
\end{equation*}
and hence
\begin{equation*}
    \frac{1}{2}\|X^{\mathrm{avg}}(X_{0}; T) - x^{\star}\|^{2} - \frac{1}{2}\|X_{0} - x^{\star}\|^{2} \leq \frac{T^{2}}{6}(f(X_{0}) - f(x^{\star}))
\end{equation*}
which proves \cref{eq:cor:hf-convex-main-2}.
The weighted sum \(\lambda \cdot \text{\cref{eq:cor:hf-convex-main}} + \text{\cref{eq:cor:hf-convex-main-2}} + \lambda^{2} \cdot \text{\cref{eq:hf-conserve}}\) (stated below in order for convenience)
\begin{align*}
    f(X^{\mathrm{avg}}(X_{0}; T)) - f(\xstar) + \frac{1}{3T^{2}}\|X_{T} - \xstar\|^{2} &\leq \frac{2}{3}(f(X_{0}) - f(\xstar)) + \frac{1}{3T^{2}}\|X_{0} - \xstar\|^{2} \\
    \frac{1}{3T^{2}}\|X^{\mathrm{avg}}(X_{0}; T) - \xstar\|^{2} &\leq \frac{1}{9}(f(X_{0}) - f(\xstar)) + \frac{1}{3T^{2}}\|X_{0} - \xstar\|^{2} \\
    f(X_{T}) - f(\xstar) &\leq f(X_{0}) - f(\xstar)~
\end{align*}
yields \(\textbf{LHS} \leq \textbf{RHS}\) where:

\textbf{LHS:~~}
\begin{align*}
    \lambda \cdot (f(X^{\mathrm{avg}}(X_{0}; T)) - f(\xstar)) + \lambda^{2} \cdot (f(X_{T}) - f(\xstar)) &\geq \lambda(1 + \lambda) \cdot (f(X^{\mathrm{mix}}(X_{0}; T)) - f(\xstar))~,\\
    \frac{\lambda}{3T^2} \cdot \|X_{T} - \xstar\|^{2} + \frac{1}{3T^2} \|X^{\mathrm{avg}}(X_{0}; T) - \xstar\|^{2} &\geq \frac{(1 + \lambda)}{3T^2} \cdot \|X^{\mathrm{mix}}(X_{0}; T) - \xstar\|^{2}
\end{align*}
and the inequalities are due to Jensen's inequality.

\textbf{RHS:~~} The RHS simplifies to
\begin{equation*}
    \left(\lambda^{2} + \frac{2\lambda}{3} + \frac{1}{9}\right) (f(X_{0}) - f(\xstar)) + \frac{1}{3T^{2}}(\lambda + 1) \|X_{0} - \xstar\|^{2}~.
\end{equation*}
Dividing both sides by \(\lambda(1 + \lambda)\), we get the claimed result:
\begin{equation*}
    f(X^{\mathrm{mix}}(X_{0}; T)) - f(\xstar) + \frac{1}{3T^{2}\lambda}\|X^{\mathrm{mix}}(X_{0}; T) - \xstar\|^{2} \leq C_{\lambda} \cdot (f(X_{0}) - f(\xstar)) + \frac{1}{3T^{2}\lambda}\|X_{0} - \xstar\|^{2}~.
\end{equation*}
\end{proof}

\subsubsection{Proof of \texorpdfstring{\Cref{thm:hf-cts-cvx} and \Cref{cor:hf-cts-cvx}}{Theorem 2 and Corollary 3}}
\label{app:prf:hf-cts-cvx}

\begin{proofof}{\Cref{thm:hf-cts-cvx}}
By definition of \(X_{k}\), for any \(k \in [K]\) it holds that \(X_{k} = X^{\mathrm{mix}}(X_{k - 1}; T_{k})\).
From \Cref{lem:hf-convex-mix-main}, we know that
\begin{align*}
    f(X_{k}) - f(x^{\star}) + \frac{1}{3T_{k}^{2}}\|X_{k} - x^{\star}\|^{2} 
    & = f(X^{\mathrm{mix}}(X_{k - 1}; T_{k})) - f(x^{\star}) + \frac{1}{3T_{k}^{2}}\|X_{k} - x^{\star}\|^{2} \\
    &\leq \frac{8}{9} \cdot (f(X_{k - 1}) - f(x^{\star})) + \frac{1}{3T_{k}^{2}}\|X_{k - 1} - x^{\star}\|^{2} \\
    &\leq \frac{8}{9} \left\{f(X_{k - 1}) - f(x^{\star}) + \frac{1}{3T_{k - 1}^{2}}\|X_{k - 1} - x^{\star}\|^{2}\right\}
\end{align*}
where we use the condition that \(T_{k}^{2} \geq \frac{9}{8}T_{k- 1}^{2}\).
Iterating this inequality, we get
\begin{equation*}
    f(X_{K}) - f(x^{\star}) \leq \left(\frac{8}{9}\right)^{K} \left\{f(X_{0}) - f(x^{\star}) + \frac{1}{3}\|X_{0} - \xstar\|^{2}\right\}~.
\end{equation*}
\end{proofof}

\begin{proofof}{\Cref{cor:hf-cts-cvx}}
For the specified choice of \(K\), we obtain from \Cref{thm:hf-cts-cvx} that \(f(X_{K}) - f(\xstar) \leq \varepsilon\).
The total integration time is bounded by:
\begin{align*}
    \sum_{k=1}^{K} T_{k} = \frac{3}{3 - \sqrt{8}} \cdot \left(\frac{9}{8}^{\nicefrac{K}{2}} - 1\right) &\leq \frac{3}{3 - \sqrt{8}} \cdot \left(\sqrt{\frac{f(X_{0})- f(x^{\star}) + \frac{1}{3}\|X_{0} - x^{\star}\|^{2}}{\varepsilon}} - 1\right)\\
    &\leq3(3+\sqrt{8})\left(\sqrt{\frac{f(X_{0})- f(x^{\star}) + \frac{1}{3}\|X_{0} - x^{\star}\|^{2}}{\varepsilon}}\right)\\
    &\leq 18\left(\sqrt{\frac{f(X_{0})- f(x^{\star}) + \frac{1}{3}\|X_{0} - x^{\star}\|^{2}}{\varepsilon}}\right).
\end{align*}
\end{proofof}

\section{Extended discussion of \texorpdfstring{\Cref{sec:discrete-algos}}{Section 4}}

\subsection{Proofs of the statements in \texorpdfstring{\Cref{sec:discrete-algos}}{Section 4}}

\subsubsection{Proof of \texorpdfstring{\Cref{lem:hf-extg-nonincrease-ham}}{Lemma 5}}
\label{app:prf:hf-extg-nonincrease-ham}

\begin{proof}
Consider the difference in the Hamiltonian at successive iterations:
\begin{align*}
    f(x_{n + 1}) + &\frac{1}{2}\|y_{n + 1}\|^{2} - f(x_{n}) - \frac{1}{2}\|y_{n}\|^{2} \\
    &\leq \langle \nabla f(x_{n + 1}), x_{n + 1} - x_{n}\rangle + \frac{1}{2}\|y_{n} - \eta\nabla f(x_{n + 1})\|^{2} - \frac{1}{2}\|y_{n}\|^{2} \\
    &= \langle \nabla f(x_{n + 1}), \eta y_{n} - \eta^{2}\nabla f(x_{n + \frac{1}{2}})\rangle + \frac{\eta^{2}}{2}\|\nabla f(x_{n + 1})\|^{2} - \eta \langle y_{n}, \nabla f(x_{n + 1})\rangle \\
    &= -\eta^{2}\langle \nabla f(x_{n + 1}), \nabla f(x_{n + \frac{1}{2}})\rangle + \frac{\eta^{2}}{2}\|\nabla f(x_{n + 1})\|^{2} \\
    &= \frac{\eta^{2}}{2}\left(\|\nabla f(x_{n + 1}) - \nabla f(x_{n + \frac{1}{2}})\|^{2} - \|\nabla f(x_{n + \frac{1}{2}})\|^{2}\right)~.
\end{align*}
Since \(x_{n + 1} = x_{n + \frac{1}{2}} - \eta^{2}\nabla f(x_{n + \frac{1}{2}})\), we have by smoothness of \(f\) that
\begin{equation*}
    \|\nabla f(x_{n + 1}) - \nabla f(x_{n + \frac{1}{2}})\|^{2} \leq L^{2}\|x_{n + 1} - x_{n + \frac{1}{2}}\|^{2} = L^{2}\eta^{4}\|\nabla f(x_{n + \frac{1}{2}})\|^{2}~.
\end{equation*}
When \(\eta \leq \frac{1}{\sqrt{L}}\),
\begin{equation*}
    f(x_{n + 1}) + \frac{1}{2}\|y_{n + 1}\|^{2} - f(x_{n}) - \frac{1}{2}\|y_{n}\|^{2} \leq \frac{\eta^{2}(L^{2}\eta^{4} - 1)}{2}\|\nabla f(x_{n + \frac{1}{2}})\|^{2} \leq 0~.
\end{equation*}
\end{proof}

\subsubsection{Proof of \texorpdfstring{\Cref{lem:hf-extg-convex-main}}{Lemma 6}}
\label{sec:prf:hf-extg-convex-main}

To prove \Cref{lem:hf-extg-convex-main}, we first discuss some parallels between the continuous-time and discrete-time analysis that we highlight to provide intuition.

    \begin{table}[H]
    \centering
    \small
    \renewcommand{\arraystretch}{1.2}
    \caption{Parallels between continuous-time and discrete-time analysis.}
    \label{tab:roadmap_3col}
    \begin{tabular}{p{0.3\linewidth} p{0.3\linewidth} p{0.33\linewidth}}
    \hline
    \textbf{} & \textbf{Continuous time} & \textbf{Discrete time} \\
    \hline
    
    \textbf{Evolution rule} &
    Hamiltonian flow (\ref{eq:HamFlow})&
    Extragradient Integrator (\ref{eqn:extg-update}) \\[2pt]
    
    \textbf{Hamiltonian property} &
    Conservation (\Cref{lem:hamflow-conserve}) &
    Dissipation (\Cref{lem:hf-extg-nonincrease-ham}) \\[2pt]
    
    \textbf{Boundary term} &
    $h_t=\langle X_t-z, Y_t\rangle$ &
    $h_n=\langle x_n-z, y_n\rangle$ \\[2pt]

    \textbf{Strategy} &
    Bound on $\dot{h}_{t}$ in terms of optimality gap (\Cref{lem:hf-convex-main}) &
    Bound on $h_{n+1}-h_n$ in terms of optimality gap (\Cref{lem:hf-extg-convex-hnterm}) \\[2pt]

    \multirow{2}{*}{\textbf{Average}} &
    \(X^{\mathrm{avg}}(X_{0}; T)=\) & \(x^{\mathrm{avg}}(x_{0}; N)=\) \\
    & \(\frac{2}{T^{2}}\int_{0}^{T} (T- t)X_{t} \rmd t\) & \(\frac{2}{N(N+1)}\sum_{n=1}^{N}(N - n + 1)x_{n}\).
    \end{tabular}
    \end{table}

\begin{lemma}
\label{lem:hf-extg-convex-hnterm}
Assume \(f\) satisfies \convexassump{} and \smoothassump{}.
Assume \(f\) is convex and \(L\)-smooth.
Let \(\{(x_{n}, y_{n})\}_{n \geq 1}\) be generated by the extragradient integrator \emph{(\ref{eqn:extg-update})} from any initial \(x_{0} \in \bbR^{d}\) and \(y_{0} = \bm{0}\), and step size \(\eta \leq \frac{1}{\sqrt{L}}\).
Then, for any \(N \geq 1\) and reference point \(z \in \bbR^{d}\), the following holds:
\begin{equation*}
    \langle x_{N} - z, y_{N}\rangle \leq 2\eta N(f(x_{0}) - f(z)) - 3\eta \sum_{n=1}^{N}(f(x_{n}) - f(z)).
\end{equation*}
\end{lemma}

We prove this lemma in \Cref{prf:lem:hf-extg-convex-hnterm}.
\Cref{lem:hf-extg-convex-main} builds on this lemma, proven below.
\begin{proofof}{\Cref{lem:hf-extg-convex-main}}
We begin by bounding the difference \(\Delta_{n} := \|x_{n + 1} - z\|^{2} - \|x_{n} - z\|^{2}\), and recall that \(h_{n} = \langle x_{n} - z, y_{n}\rangle\).
Using the relation between \(x_{n + 1}, y_{n + 1}\) from \(x_{n}, y_{n}\) as given in \ref{eqn:extg-update}, we have
\begin{align*}
    \Delta_{n} &= \|x_{n + 1} - x_{n}\|^{2} + 2\langle x_{n + 1} - x_{n}, x_{n} - z\rangle \\
    &= \|\eta y_{n} - \eta^{2} \nabla f(x_{n + \frac{1}{2}})\|^{2} + 2\langle \eta y_{n} - \eta^{2} \nabla f(x_{n + \frac{1}{2}}), x_{n} - z\rangle \\
    &= 2\eta h_{n} - 2\eta^{2}\langle \nabla f(x_{n + \frac{1}{2}}), x_{n} - z\rangle - 2\eta^{3}\langle y_{n}, \nabla f(x_{n + \frac{1}{2}})\rangle + \eta^{2}\|y_{n}\|^{2} + \eta^{4}\|\nabla f(x_{n + \frac{1}{2}})\|^{2} \\
    &= 2\eta h_{n} - 2\eta^{2}\langle \nabla f(x_{n + \frac{1}{2}}), x_{n}+\eta y_n - z\rangle+ \eta^{2}\|y_{n}\|^{2} + \eta^{4}\|\nabla f(x_{n + \frac{1}{2}})\|^{2} \\
    &= 2\eta h_{n} - 2\eta^{2}\langle \nabla f(x_{n + \frac{1}{2}}), x_{n + \frac{1}{2}} - z\rangle + \eta^{2} \|y_{n}\|^{2} + \eta^{4}\|\nabla f(x_{n + \frac{1}{2}})\|^{2}~\\
    &\leq 2\eta h_{n} + 2\eta^{2}(f(z) - f(x_{n +\frac{1}{2}})) +  \eta^{2} \|y_{n}\|^{2} + \eta^{4}\|\nabla f(x_{n + \frac{1}{2}})\|^{2}~,
\end{align*}
where the final inequality above uses the convexity of \(f\). Since $x_{n+1} = x_{n+\frac{1}{2}}-\eta^2 \nabla f(x_{n+\frac{1}{2}})$, by smoothness of \(f\),
\begin{equation*}
    f(x_{n + 1}) \leq f(x_{n + \frac{1}{2}}) - \eta^{2}\left(1 - \frac{L\eta^{2}}{2}\right)\|\nabla f(x_{n + \frac{1}{2}})\|^{2}~.
\end{equation*}
When \(\eta \leq \frac{1}{\sqrt{L}}\), this implies
\begin{equation*}
\frac{\eta^2}{2}\|\nabla f(x_{n + \frac{1}{2}})\|^{2}\leq\eta^{2}\left(1 - \frac{L\eta^{2}}{2}\right)\|\nabla f(x_{n + \frac{1}{2}})\|^{2} \leq f(x_{n + \frac{1}{2}})-f(x_{n + 1})~.
\end{equation*}
Substituting this bound and using \Cref{lem:hf-extg-nonincrease-ham} to write \(\|y_{n}\|^{2} \leq 2(f(x_{0}) - f(x_{n}))\), we get
\begin{align*}
    \Delta_{n} &\leq 2\eta h_{n} + 2\eta^{2}(f(z) - f(x_{n + \frac{1}{2}})) + 2\eta^{2} \cdot (f(x_{0}) - f(x_{n})) + 2\eta^{2}(f(x_{n + \frac{1}{2}}) - f(x_{n + 1})) \\
    &= 2\eta h_{n} +2\eta^{2}(f(x_{0}) - f(x_n))+ 2\eta^{2}(f(z) - f(x_{n+1})).\numberthis\label{eq:deltan-bound}
\end{align*}
Recall \Cref{lem:hf-extg-convex-hnterm} guarantees that
\begin{align*}
    h_n &\leq 2\eta n(f(x_0)-f(z))-3\eta \sum_{n'=1}^n(f(x_{n'})-f(z))\\
    &= 2(n+1)\eta (f(x_0)-f(z))-3\eta\sum_{n'=1}^{n+1}(f(x_{n'})-f(z))+\eta (3f(x_{n+1})-2f(x_0)-f(z)).
\end{align*}
Substituting this in the bound for \(\Delta_{n}\), we get
\begin{align*}
    \Delta_{n} &\leq 2\eta h_{n} +2\eta^{2}(f(x_{0}) - f(x_n))+ 2\eta^{2}(f(z) - f(x_{n+1}))\\
    &=4\eta^2 (n+1)(f(x_0)-f(z))-6\eta^2  \sum_{n'=1}^{n+1}(f(x_{n'})-f(z)) +2\eta^2  (3f(x_{n+1})-2f(x_0)-f(z)) \\
    &\qquad +2\eta^{2}(f(x_{0}) - f(x_n))+ 2\eta^{2}(f(z) - f(x_{n+1}))\\
    &=4\eta^2 (n+1)(f(x_0)-f(z))-6\eta^2  \sum_{n'=1}^{n+1}(f(x_{n'})-f(z))+2\eta^{2}(f(x_{n+1})-f(x_n))\\
    &\qquad +2\eta^{2}(f(x_{n+1})-f(x_0))\\
    &\leq 4\eta^2 (n+1)(f(x_0)-f(z))-6\eta^2  \sum_{n'=1}^{n+1}(f(x_{n'})-f(z))+2\eta^{2}(f(x_{n+1})-f(x_n)),
\end{align*}
where the last inequality is due to \Cref{lem:hf-extg-nonincrease-ham}.
We now sum this inequality from $n=0$ to $n=N-1$:
\begin{align*}
    \|x_{N} &- z\|^{2} - \|x_{0} - z\|^{2} \\
    &= \sum_{n=0}^{N - 1} \Delta_{n} \\
    &\leq \sum_{n=0}^{N-1}4\eta^2 (n+1)(f(x_0)-f(z))-6\eta^2\sum_{n=0}^{N-1}  \sum_{n'=1}^{n+1}(f(x_{n'})-f(z))+2\eta^{2}\sum_{n=0}^{N-1}(f(x_{n+1})-f(x_n))\\
    &=\sum_{n=1}^{N} 4\eta^{2}n(f(x_{0}) - f(z)) - 6\eta^{2}\sum_{n=1}^{N}\sum_{n'=1}^{n}(f(x_{n'}) - f(z)) + 2\eta^{2}(f(x_{N}) - f(x_{0})) \\
    &=2N(N+1)\eta^2(f(x_0)-f(z))-6\eta^2\sum_{n=1}^{N}(N-n+1)(f(x_n)-f(z))+2\eta^2(f(x_N)-f(x_0))\\
    &\leq 2N(N+1)\eta^2(f(x_0)-f(z))-6\eta^2\sum_{n=1}^{N}(N-n+1)(f(x_n)-f(z))~
\end{align*}
where the last inequality again uses \Cref{lem:hf-extg-nonincrease-ham}.
Dividing both sides by \(3\eta^{2}N(N + 1)\) and rearranging terms, we get
\begin{multline*}
    \frac{2}{N(N + 1)}\sum_{n=1}^{N}(N - n + 1)(f(x_{n}) - f(z)) + \frac{1}{3\eta^{2}N(N + 1)}\|x_{N} - z\|^{2} \\
    \leq \frac{2}{3}(f(x_{0}) - f(z)) + \frac{1}{3\eta^{2}N(N + 1)}\|x_{0} - z\|^{2}~.
\end{multline*}
Since \(\frac{2}{N(N + 1)}\sum_{n=1}^{N} (N - n + 1) = 1\), we have from Jensen's inequality and the definition of \(x^{\mathrm{avg}}(x_{0}; N)\) that
\begin{multline*}
    f(x^{\mathrm{avg}}(x_{0}; N )) - f(z) + \frac{1}{3\eta^{2}N(N+1)}\|x_{N} - z\|^{2} 
    \leq \frac{2}{3}(f(x_{0}) - f(z)) + \frac{1}{3\eta^{2}N(N + 1)}\|x_{0} - z\|^{2}~.
\end{multline*}
\end{proofof}

\subsubsection{Proofs of \texorpdfstring{ \Cref{thm:hf-extg-strong-cvx,cor:hf-extg-strong-cvx}}{Theorem 3 and Corollary 4}}
\label{app:prf:hf-extg-strong-cvx}

\begin{proofof}{\Cref{thm:hf-extg-strong-cvx}}
    Recall that \(x_{k}\) produced by \nameref{alg:disc_phfopt} satisfies \(x_{k} = x^{\mathrm{avg}}(x_{k - 1}; N_{k})\).
From \Cref{lem:hf-extg-convex-main},
\begin{align*}
    f(x_{k}) - f(x^{\star}) &= f(x^{\mathrm{avg}}(x_{k - 1}; N_{k})) - f(x^{\star}) \\
    &\leq \frac{2}{3}(f(x_{k - 1}) - f(x^{\star})) + \frac{1}{3\eta^{2}N_{k}(N_{k} + 1)}\|x_{k - 1} - x^{\star}\|^{2} \\
    &\overset{(a)}\leq \left(\frac{2}{3} + \frac{2}{3\eta^{2}\alpha N_{k}(N_{k} + 1)}\right) (f(x_{k - 1}) - f(x^{\star})) \\
    &\overset{(b)}\leq \left(\frac{2}{3} + \frac{2}{3c^{2}}\right) \cdot (f(x_{k - 1}) - f(\xstar))~.
\end{align*}
In step \((a)\), we use the fact that \(f\) satisfies \quadgrowassump{}, and in step \((b)\) we use the condition that \(N_{k} \geq \lceil \frac{c}{\eta\sqrt{\alpha}}\rceil \geq \frac{c}{\eta\sqrt{\alpha}}\).
Iterating this inequality,
\begin{equation*}
    f(x_{K}) - f(\xstar) \leq \left(\frac{2}{3} + \frac{2}{3c^{2}}\right)^{K} \cdot (f(x_{0}) - f(\xstar))~.
\end{equation*}

\end{proofof}
\begin{proofof}{\Cref{cor:hf-extg-strong-cvx}}
Since \(N_{k} = \lceil \frac{5}{2}\sqrt{\kappa}\rceil \geq \frac{5}{2}\sqrt{\kappa}\), 
from \Cref{thm:hf-extg-strong-cvx} we have
\begin{equation*}
    f(x_{K}) - f(x^{\star}) \leq \left(\frac{58}{75}\right)^{K} \cdot (f(x_{0}) - f(x^{\star}))~.
\end{equation*}
Since \(K = \left\lceil 4\log \frac{f(x_{0}) - f(x^{\star})}{\varepsilon}\right\rceil \geq \left(\log \frac{75}{58}\right)^{-1}\frac{f(x_{0}) - f(x^{\star})}{\varepsilon}\), 
\begin{equation*}
    f(x_{K}) - f(x^{\star}) \leq \varepsilon~.
\end{equation*}
Since at each iteration of the extragradient integrator we use 2 gradient calls, we have the total gradient complexity as
\begin{equation*}
    2 \cdot \left\lceil \frac{5}{2}\sqrt{\frac{L}{\alpha}}\right\rceil \cdot \left\lceil 4 \log \frac{f(x_{0}) - f(x^{\star})}{\varepsilon} \right\rceil~.
\end{equation*}    
\end{proofof}

\subsubsection{Proof of \texorpdfstring{\Cref{lem:hf-extg-convex-mix-main}}{Lemma 7}}
\label{prf:lem:hf-extg-convex-mix-main}

\begin{proof}
From \cref{eq:deltan-bound}, we have that \(\Delta_{n} := \|x_{n + 1} - \xstar\|^{2} - \|x_{n} - \xstar\|^{2}\) satisfies
\begin{equation*}
    \Delta_{n'} \leq 2\eta h_{n'} + 2\eta^{2} (f(x_{0}) - f(x_{n'}))~.
\end{equation*}
Setting $z=\xstar$ in \Cref{lem:hf-extg-convex-hnterm}, we could bound $h_{n'}$ using:
$$
\langle x_{n'} - \xstar, y_{n'}\rangle \leq 2\eta n'(f(x_{0}) - f(\xstar)) - 3\eta \sum_{n=1}^{n'}(f(x_{n}) - f(\xstar)) \leq 2\eta n'(f(x_{0}) - f(\xstar)).
$$
Plugging this bound into the upper bound for $\Delta_{n'}$ gives:
\begin{equation*}
    \Delta_{n'} \leq 4\eta^{2} n'(f(x_{0}) -f(\xstar)) + 2\eta^{2} (f(x_{0}) - f(x_{n'})) \leq 4\eta^{2}n'(f(x_{0}) - f(x^{\star})) + 2\eta^{2}(f(x_{0})- f(x^{\star}))~.
\end{equation*}
Summing both sides from \(n' =0\) to \(n - 1\),
\begin{equation*}
   \|x_{n} - \xstar\|^{2} - \|x_{0} - \xstar\|^{2} \leq 2\eta^{2} n^{2}(f(x_{0}) - f(\xstar))~.
\end{equation*}
Multiplying both sides by \(N - n + 1\) and summing from \(n =1\) to \(n = N\), we get
\begin{equation*}
    \sum_{n=1}^{N} (N - n + 1) \|x_{n} - \xstar\|^{2} \leq \frac{N(N + 1)}{2} \|x_{0} - \xstar\|^{2} + \frac{\eta^{2}N(N+1)^{2}(N + 2)}{6}(f(x_{0}) - f(x^{\star}))~.
\end{equation*}
Dividing both sides by \(N(N + 1)\), we get
\begin{equation*}
    \frac{2}{N(N + 1)}\sum_{n=1}^{N} (N - n + 1) \frac{\|x_{n} - \xstar\|^{2}}{2} \leq \frac{\|x_{0} - \xstar\|^{2}}{2} + \frac{\eta^{2}(N + 1)(N + 2)}{6}(f(x_{0}) - f(x^{\star}))~.
\end{equation*}
Note that \(\sum_{n=1}^{N}(N - n + 1) = \frac{N(N + 1)}{2}\), and therefore by Jensen's inequality and the definition of \(x^{\mathrm{avg}}(x_{0}; N)\) we get
\begin{equation*}
    \frac{1}{2}\|x^{\mathrm{avg}}(x_{0}; N) - \xstar\|^{2} \leq \frac{1}{2}\|x_{0} - \xstar\|^{2} + \frac{\eta^{2}(N + 1)(N + 2)}{6}(f(x_{0}) - f(\xstar))~.
\end{equation*}
Since \((N + 1)(N + 2) \leq \frac{3}{2}N(N + 1)\) for \(N \geq 4\), we have
\begin{align}
    \frac{1}{2}\|x^{\mathrm{avg}}(x_{0}; N) - \xstar\|^{2} &\leq \frac{1}{2}\|x_{0} - \xstar\|^{2} + \frac{\eta^{2}N(N + 1)}{4}(f(x_{0}) - f(\xstar))\nonumber\\
    \frac{1}{3\eta^{2}N(N + 1)}\|x^{\mathrm{avg}}(x_{0}; N) - x^{\star}\|^{2} &\leq \frac{1}{3\eta^{2}N(N + 1)}\|x_{0} - x^{\star}\|^{2} + \frac{1}{6}(f(x_{0}) - f(x^{\star}))~.\label{eq:cor:hf-extg-convex-main-2}
\end{align}

Like in the proof of \Cref{lem:hf-convex-mix-main}, we consider a weighted sum of the following inequalities:
\begin{align*}
f(x^{\mathrm{avg}}(x_{0}; N )) - f(x^\star) + \frac{\|x_{N} - x^\star\|^{2}}{3\eta^{2}N(N+1)} 
    &\leq \frac{2}{3}(f(x_{0}) - f(x^\star)) + \frac{\|x_{0} - x^\star\|^{2}}{3\eta^{2}N(N + 1)}~ & \text{(\cref{eqn:oracle_step_c_extg})} \\
    \frac{1}{3\eta^{2}N(N + 1)}\|x^{\mathrm{avg}}(x_{0}; N) - x^{\star}\|^{2} &\leq \frac{1}{3\eta^{2}N(N + 1)}\|x_{0} - x^{\star}\|^{2} + \frac{1}{6}(f(x_{0}) - f(x^{\star})) & \text{(\cref{eq:cor:hf-extg-convex-main-2})} \\
    f(x_{N}) - f(x^{\star}) &\leq f(x_{0}) - f(x^{\star}) & \text{(\cref{eq:hf-extg-nonincrease-ham})}
\end{align*}
with weights \(\lambda, 1, \lambda^{2}\) respectively.
This yields \(\textbf{LHS} \leq \textbf{RHS}\) where:

\textbf{LHS:}~~
\begin{align*}
    &\lambda \cdot (f(x^{\mathrm{avg}}(x_{0}; N) - f(x^{\star})) + \lambda^{2} \cdot (f(x_{N}) - f(\xstar))\\
    &+\frac{\lambda}{3\eta^{2}N(N+1)} \|x_{N} - x^{\star}\|^{2} +\frac{1}{3\eta^{2}N(N+1)}\|x^{\mathrm{avg}}(x_{0}; N) - x^{\star}\|^{2}\\
    &\geq \lambda(1 + \lambda) \cdot (f(x^{\mathrm{mix}}(x_{0}; N)) - f(x^{\star}))+ \frac{(1 + \lambda)}{3\eta^{2}N(N+1)}\cdot \|x^{\mathrm{mix}}(x_{0}; N) - \xstar\|^{2}~.
\end{align*}

\textbf{RHS:}~~
\begin{equation*}
    \left(\lambda^{2} + \frac{2\lambda}{3} + \frac{1}{6}\right)(f(x_{0}) - f(\xstar)) + \frac{1 + \lambda}{3\eta^{2}N(N + 1)}\|x_{0} - \xstar\|^{2}~.
\end{equation*}
After dividing both sides by \(\lambda(1 + \lambda)\), this finally returns
\begin{multline*}
    f(x^{\mathrm{mix}}(x_{0}; N)) - f(\xstar) + \frac{1}{3\lambda\eta^{2}N(N + 1)}\|x^{\mathrm{mix}}(x_{0}; N) - \xstar\|^{2} \\
    \leq c_{\lambda}(f(x_{0}) - f(\xstar)) + \frac{1}{3\lambda\eta^{2}N(N + 1)}\|x_{0} - \xstar\|^{2}
\end{multline*}
which is precisely the statement of the lemma.
\end{proof}

\subsubsection{Proofs of \texorpdfstring{\Cref{thm:hf-extg-cvx} and \Cref{cor:hf-extg-cvx}}{Theorem 4 and Corollary 5}}
\label{app:prf:hf-extg-cvx}

\begin{proofof}{\Cref{thm:hf-extg-cvx}}
    Note that for any \(k \geq 1\), \(x_{k} = x^{\mathrm{mix}}(x_{k-1}; N_{k})\) with \(\lambda = \lambda^{\star} = \frac{\sqrt{3} + 1}{2}\).
    Consequently, from \Cref{lem:hf-extg-convex-mix-main}, we have
    \begin{align*}
        f(x_{k}) &- f(x^{\star}) + \frac{1}{3\lambda^{\star}\eta^{2}N_{k}(N_{k} + 1)}\|x_{k} - x^{\star}\|^{2} \\
        &= f(x^{\mathrm{mix}}(x_{k-1}; N_{k})) - f(x^{\star}) + \frac{1}{3\lambda^{\star}\eta^{2}N_{k}(N_{k} + 1)}\|x^{\mathrm{mix}}(x_{k-1}; N_{k}) - x^{\star}\|^{2} \\
        &\leq c_{\lambda^{\star}}(f(x_{k- 1}) - f(x^{\star})) + \frac{1}{3\lambda^{\star}\eta^{2}N_{k}(N_{k} + 1)}\|x_{k-1} - x^{\star}\|^{2} \\
        &= c_{\lambda^{\star}} \left\{f(x_{k - 1}) - f(x^{\star}) + \frac{1}{3\lambda^{\star}\eta^{2}c_{\lambda^{\star}}N_{k}(N_{k} + 1)}\|x_{k - 1} - \xstar\|^{2}\right\} \\
        &\leq c_{\lambda^{\star}} \left\{f(x_{k - 1}) - f(x^{\star}) + \frac{1}{3\lambda^{\star}\eta^{2}N_{k - 1}(N_{k - 1} + 1)}\|x_{k - 1} - \xstar\|^{2} \right\}~.
    \end{align*}
    The final inequality uses the fact that \(N_{k}(N_{k} + 1) \geq \frac{1}{c_{\lambda^{\star}}}N_{k - 1}(N_{k - 1} + 1)\).
    Iterating this from \(k = 1\) to \(k = K\), we get
    \begin{equation*}
        f(x_{K}) - f(x^{\star}) \leq \left(\frac{\sqrt{3} + 1}{3}\right)^{K} \left\{f(x_{0}) - f(x^{\star}) + \frac{1}{3\lambda^{\star}\eta^{2}N_{0}(N_{0} + 1)}\|x_{0} - \xstar\|^{2} \right\}~.
    \end{equation*}
    Setting \(N_{0} = 4\) and \(\lambda^{\star} = \frac{\sqrt{3} + 1}{2}\) completes the proof.
\end{proofof}
\begin{proofof}{\Cref{cor:hf-extg-cvx}}
We first show that the desired inequality between $N_k$ and $N_{k+1}$ holds under the specified $\{N_k\}$ sequence. Define $q=\nicefrac{1}{\sqrt{c_\lambda^{\star}}}=\sqrt{\frac{3}{1+\sqrt{3}}} \leq 2$. By properties of $\lceil \cdot \rceil$, we calculate that
\begin{align*}
    N_{k+1}^2+N_{k+1} &= \lceil qN_{k}+\tfrac{1}{2}\rceil^2+\lceil qN_{k}+\tfrac{1}{2}\rceil\\
    &\geq (qN_{k}+\tfrac{3}{2})(qN_{k}+\tfrac{1}{2})\\
    &=q^2N_{k}^2+2qN_{k}+\tfrac{3}{4}\\
    &\geq q^2(N_{k}^2+N_{k}),
\end{align*}
where the last inequality followed because $q\leq 2$. 
Additionally, note that
\begin{align*}
    K &\geq \left\lceil \left(\log \frac{3}{1 + \sqrt{3}}\right)^{-1} \cdot \log \frac{\frac{L}{2}\|x_{0} - \xstar\|^{2} + \frac{(\sqrt{3} - 1)L}{60}\|x_{0} - \xstar\|^{2}}{\varepsilon}\right\rceil \\
    &\geq \left(\log\frac{3}{1 + \sqrt{3}}\right)^{-1} \cdot \log \frac{f(x_{0}) - f(x^{\star}) + \frac{1}{3\lambda^{\star}\eta^{2}N_{0}(N_{0} + 1)}\|x_{0} - \xstar\|^{2}}{\varepsilon}\,.
\end{align*}
Therefore, from \Cref{thm:hf-extg-cvx},
\begin{equation*}
    f(x_{K}) - f(\xstar) \leq \varepsilon\,.
\end{equation*}
To obtain the number of gradient evaluations, recall that by definition, \(N_{k + 1} \leq qN_{k} + \frac{3}{2}\).
Consequently,
$$
N_k \leq q^kN_0+\frac{3}{2}\left(\frac{q^{k}-1}{q-1} \right)\leq q^k\lp N_0+\frac{3}{2(q-1)}\rp.
$$
Then, summing both sides from \(k = 1\) to \(k = K\),
\begin{equation*}
    \sum_{k=1}^{K} N_{k} \leq  \left(\frac{N_{0}q}{q - 1} + \frac{3q}{2(q - 1)^{2}}\right) (q^{K} - 1) \leq  \left(\frac{N_{0}q}{q - 1} + \frac{3q}{2(q - 1)^{2}}\right) q^{K}~.
\end{equation*}
The quantity \(q^{K}\) can be simplified as
\begin{align*}
    q^K &=\exp(K\log q) \\
    &\leq \exp\lp\frac{1}{2}\log\lp\frac{3}{1+\sqrt{3}}\rp \cdot \left( 2\log\left(\frac{3}{1+\sqrt{3}}\right)^{-1} \cdot \log \lp\sqrt{\frac{L}{\epsilon}}\|x_0-\xstar\|\rp+1\right)\rp \\
    &=\sqrt{\frac{L}{\epsilon}}\|x_0-\xstar\| \cdot q\,.
\end{align*}
Substituting the values for \(N_{0}, q\), the constant \(\frac{N_{0}q^{2}}{q - 1} + \frac{3q^{2}}{2(q - 1)^{2}} \leq 810\), which completes the proof.
\end{proofof}

\subsubsection{Proof of \texorpdfstring{\Cref{lem:hf-extg-convex-hnterm}}{Lemma 15}}
\label{prf:lem:hf-extg-convex-hnterm}

\begin{proof}
Denote \(h_{n} = \langle x_{n} - z, y_{n}\rangle\).
We first compute the difference
\begin{align*}
    h_{n + 1} - h_{n} &= \langle x_{n + 1} - z, y_{n + 1}\rangle - \langle x_{n} - z, y_{n}\rangle \\
    &= \langle x_{n + 1} - z, y_{n + 1} - y_{n}\rangle + \langle x_{n + 1} - z, y_{n}\rangle - \langle x_{n} - z, y_{n}\rangle \\
    &= -\eta \langle x_{n + 1} - z, \nabla f(x_{n + 1})\rangle + \langle x_{n + 1} - x_{n}, y_{n}\rangle~.
\end{align*}
For the first term on the right hand side, we use the convexity of \(f\):
\begin{equation*}
    -\eta \langle x_{n + 1} - z, \nabla f(x_{n + 1})\rangle \leq \eta (f(z) - f(x_{n + 1}))~.
\end{equation*}
For the second term on the right hand side, we additionally use the smoothness of \(f\) and the form of the extragradient integrator to get
\begin{align*}
    \langle x_{n + 1} - x_{n}, y_{n}\rangle &= \langle \eta y_{n} - \eta^{2}\nabla f(x_{n + \frac{1}{2}}), y_{n}\rangle \\
    &= \eta \|y_{n}\|^{2} - \eta \langle \nabla f(x_{n + \frac{1}{2}}), x_{n + \frac{1}{2}} - x_{n}\rangle \\
    &\overset{(a)}\leq 2\eta(f(x_{0}) - f(x_{n})) + \eta (f(x_{n}) - f(x_{n} + \eta y_{n})) \\
    &\leq 2\eta f(x_{0}) - \eta f(x_{n}) - \eta f(x_{n + \frac{1}{2}}) \\
    &\overset{(b)}\leq 2\eta f(x_{0}) - \eta f(x_{n}) - \eta f(x_{n + 1})~.
\end{align*}
Steps \((a)\) and \((b)\) are due to the choice of \(\eta \leq \frac{1}{\sqrt{L}}\) from \Cref{lem:hf-extg-nonincrease-ham} and the fact that \(f(x_{n + 1}) \leq f(x_{n + \frac{1}{2}})\) respectively.
Combining the bounds, we get
\begin{equation*}
    h_{n + 1} - h_{n} \leq 2\eta (f(x_{0}) - f(z)) - \eta (f(x_{n}) - f(z)) - 2\eta (f(x_{n + 1}) - f(z))~.
\end{equation*}
Summing both sides from \(n =0\) to \(n = N - 1\) and noting that \(h_{0} = 0\),
\begin{align*}
    h_{N} &\leq 2\eta N(f(x_{0}) - f(z)) - 3\eta \sum_{n=1}^{N - 1} (f(x_{n}) - f(z)) - 2\eta (f(x_{N}) - f(z)) - \eta (f(x_{0}) - f(z)) \\
    &\leq 2\eta N(f(x_{0}) - f(z)) - 3\eta \sum_{n=1}^{N}(f(x_{n}) - f(z))
\end{align*}
where in the final inequality we use the fact that \(f(x_{N}) \leq f(x_{0})\).
\end{proof}
\subsection{Discretized Hamiltonian Flow for optimization with Averaging using \texorpdfstring{\ref{eq:implicit-integ}}{HFimp}}
\label{app:sec:dhfa-im}

Here, we discuss how the implicit integrator is also capable of achieving accelerated rates.
As discussed previously, the implicit integrator assumes access to the proximal oracle for \(f\).
The resulting algorithm~\nameref{alg:disc_pphfopt} is provided in~\Cref{alg:disc_pphfopt}.
By \Cref{lem:hf-imp-nonincrease-ham}, when \(f\) is convex, the iterates generated by the implicit integrator produce a sequence of non-increasing Hamiltonians, mirroring \Cref{lem:hf-extg-nonincrease-ham};
notably, this holds without requiring the smoothness of \(f\).
As a consequence of this monotonicity property, the iterates generated by the implicit integrator also satisfies the same guarantee as in \cref{eqn:oracle_step_c_extg}; see the lemma below, proven in \Cref{subsub_apdx:one_step_HFimp}.
\begin{lemma}
\label{lem:hf-imp-convex-main}
Assume that \(f\) satisfies \emph{\convexassump{}}, and let \(\{(x_{n}, y_{n})\}_{n \geq 1}\) be the iterates of the implicit integrator \emph{(\ref{eq:implicit-integ})}
with any step size \(\eta > 0\), starting from \((x_{0}, \bm{0})\).
Then, for any \(N \geq 1\) and reference point \(z \in \bbR^{d}\)
\begin{equation}
\label{eqn:oracle_step_c_imp}
    f(x^{\mathrm{avg}}(x_{0}; N)) - f(z) + \frac{1}{3\eta^{2}N(N + 1)}\|x_{N} - z\|^{2} \leq \frac{2}{3}(f(x_{0}) - f(z)) + \frac{1}{3\eta^{2}N(N + 1)}\|x_{0} - z\|^{2}~.
\end{equation}
\end{lemma}

Now, we state the analogues of \Cref{thm:hf-extg-strong-cvx} and \Cref{cor:hf-extg-strong-cvx} for \nameref{alg:disc_pphfopt}.

\begin{algorithm}[t]
\caption{Discretized Hamiltonian Flow for optimization with Aggregation\\using implicit integrator (\PDISCALG{})}
\algotitle{\PDISCALG{}}{alg:disc_pphfopt}
\SetKwInOut{Input}{Input}\SetKwInOut{Output}{Output}
\Input{Initialization: $x_0\in \R^d$, Number of iterations: $K \in \mathbb{N}$, parameter: $\lambda  \in \R_{\geq 0}$, step size $\eta > 0$, discretization steps sequence: $\{N_1, \dots, N_K\} \in \mathbb{N}^K$, oracle for implicit integrator \ref{eq:implicit-integ}}
\For{\(k = 1\) \KwTo \(K\)}{
    Obtain sequence \(\{x^{(k)}_{n}, y^{(k)}_{n}\}_{n = 1}^{N_{k}}\) using \ref{eq:implicit-integ} from \((x_{0}^{(k)}, y_{0}^{(k)}) = (x_{k- 1}, \bm{0})\).

    Compute \(x^{\mathrm{avg}}(x_{k - 1}; N_{k}) = \frac{2}{N_{k}(N_{k} + 1)}\sum_{n=1}^{N_{k}} (N_{k} - n + 1) x^{(k)}_{n}\).

    Set \(x_{k} = \frac{1}{\lambda + 1} x^{\mathrm{avg}}(x_{k - 1}; N_{k}) + \frac{\lambda}{\lambda + 1} x^{(k)}_{N_{k}}\).
    }
\Return $x_K$
\end{algorithm}

\begin{theorem}\label{thm:hf-imp-strong-cvx}
Assume that \(f\)  satisfies \emph{\convexassump{}} and \emph{\quadgrowassump{}}.
Suppose \(x_{K}\) is the output of running \emph{\nameref{alg:disc_pphfopt}} for \(K\) iterations with parameter \(\lambda = 0\), any step size \(\eta>0\), and discretization steps sequence satisfying \(N_{k} = N \geq \lceil \frac{c}{\eta\sqrt{\alpha}}\rceil\) for all \(k \in [K]\) and initialization \(x_{0} \in \bbR^{d}\).
Then,
\begin{equation*}
    f(x_{K}) - f(x^{\star}) \leq \left(\frac{2}{3} + \frac{2}{3c^{2}}\right)^{K} \cdot (f(x_{0}) - f(x^{\star}))~.
\end{equation*}
\end{theorem}

\begin{proof}
    Recall that \(x_{k}\) produced by \nameref{alg:disc_pphfopt} satisfies \(x_{k} = x^{\mathrm{avg}}(x_{k - 1}; N_{k})\).
From \Cref{lem:hf-imp-convex-main},
\begin{align*}
    f(x_{k}) - f(x^{\star}) &= f(x^{\mathrm{avg}}(x_{k - 1}; N_{k})) - f(x^{\star}) \\
    &\leq \frac{2}{3}(f(x_{k - 1}) - f(x^{\star})) + \frac{1}{3\eta^{2}N_{k}(N_{k} + 1)}\|x_{k - 1} - x^{\star}\|^{2} \\
    &\overset{(a)}\leq \left(\frac{2}{3} + \frac{2}{3\eta^{2}\alpha N_{k}(N_{k} + 1)}\right) (f(x_{k - 1}) - f(x^{\star})) \\
    &\overset{(b)}\leq \left(\frac{2}{3} + \frac{2}{3c^{2}}\right) \cdot (f(x_{k - 1}) - f(\xstar))~.
\end{align*}
In step \((a)\), we use the fact that \(f\) satisfies \quadgrowassump{}, and in step \((b)\) we use the condition that \(N_{k} \geq \lceil \frac{c}{\eta\sqrt{\alpha}}\rceil \geq \frac{c}{\eta\sqrt{\alpha}}\).
Iterating this inequality,
\begin{equation*}
    f(x_{K}) - f(\xstar) \leq \left(\frac{2}{3} + \frac{2}{3c^{2}}\right)^{K} \cdot (f(x_{0}) - f(\xstar))~.
\end{equation*}
\end{proof}

\begin{corollary}
\label{cor:hf-imp-strong-cvx}
Assume that \(f\)  satisfies \emph{\convexassump{}} and \emph{\quadgrowassump{}}.
To obtain \(x_{K}\) that is a \(\varepsilon\)-accurate minimizer, it suffices to run \nameref{alg:disc_pphfopt} with parameter \(\lambda = 0\) and any step size \(\eta > 0\) for \(K\) iterations and discretization step sequence \(\{N_{k}\}_{k=1}^{K}\) with
\begin{equation*}
    K = \left\lceil 4 \log \frac{f(x_{0}) - f(x^{\star})}{\varepsilon} \right\rceil~; \quad N_{k} = \left\lceil \frac{5}{2\eta \sqrt{\alpha}} \right\rceil~\forall~k \in [K]~.
\end{equation*}
The total number of calls to the implicit discretization oracle is 
$$N_{\mathrm{tot}} = \left\lceil \frac{5}{2\eta \sqrt{\alpha}} \right\rceil \cdot \left\lceil 4 \log \frac{f(x_{0}) - f(x^{\star})}{\varepsilon} \right\rceil~.$$
\end{corollary}

\begin{proof}
Since \(N_{k} = \lceil \frac{5}{2\eta \sqrt{\alpha}} \rceil \geq \frac{5}{2\eta \sqrt{\alpha}}\), 
from \Cref{thm:hf-imp-strong-cvx} we have
\begin{equation*}
    f(x_{K}) - f(x^{\star}) \leq \left(\frac{58}{75}\right)^{K} \cdot (f(x_{0}) - f(x^{\star}))~.
\end{equation*}
Since \(K = \left\lceil 4\log \frac{f(x_{0}) - f(x^{\star})}{\varepsilon}\right\rceil \geq \left(\log \frac{75}{58}\right)^{-1}\frac{f(x_{0}) - f(x^{\star})}{\varepsilon}\), 
\begin{equation*}
    f(x_{K}) - f(x^{\star}) \leq \varepsilon~.
\end{equation*}
Therefore, the total number of calls to the implicit discretization oracle is
\begin{equation*}
    \left\lceil \frac{5}{2\eta \sqrt{\alpha}} \right\rceil \cdot \left\lceil 4 \log \frac{f(x_{0}) - f(x^{\star})}{\varepsilon} \right\rceil~.
\end{equation*}    
\end{proof}

To obtain the result for minimizing convex functions, we first state the analogous version of \Cref{lem:hf-extg-convex-mix-main} for \ref{eq:implicit-integ}. A proof is provided in \Cref{subsub_apdx:multi_step_HFimp}.
\begin{lemma}
\label{lem:hf-imp-convex-mix-main}
Assume that \(f\) satisfies \emph{\convexassump{}}.
Let \(\{(x_{n}, y_{n})\}_{n\geq 1}\) be generated by the implicit integrator with step size \(\eta > 0\).
For any \(N \geq 4\) and initial position \(x_{0}\), define
\begin{equation*}
    x^{\mathrm{mix}}(x_{0}; N) = \frac{\lambda}{\lambda + 1} x_{N} + \frac{1}{1 + \lambda} x^{\mathrm{avg}}(x_{0}; N)~.
\end{equation*}
Then, defining \(c_{\lambda} = \frac{(6\lambda^{2} + 4\lambda + 1)}{6\lambda(\lambda + 1)}\), we have
\[
f(x^{\mathrm{mix}}(x_{0}; N)) - f(x^{\star}) + \frac{\|x^{\mathrm{mix}}(x_{0}; N) - x^{\star}\|^{2}}{3\lambda \eta^{2}N(N + 1)} \leq c_{\lambda}(f(x_{0}) - f(x^{\star})) + \frac{\|x_{0} - x^{\star}\|^{2}}{3\lambda\eta^{2}N(N + 1)}~.
\]
\end{lemma}

Since the conclusion of the \Cref{lem:hf-imp-convex-mix-main} is the same as the conclusion of \Cref{lem:hf-extg-convex-mix-main}, and \nameref{alg:disc_pphfopt} is algorithmically the same as \nameref{alg:disc_phfopt} with just \ref{eq:implicit-integ} in lieu of \ref{eqn:extg-update}, we have the following theorem and corollary for minimizing convex \(f\), analogous to \Cref{thm:hf-extg-cvx,} and \Cref{cor:hf-extg-cvx}. We state them without proof, as they can be derived from \Cref{lem:hf-imp-convex-mix-main} arguments similar to that in \Cref{app:prf:hf-extg-cvx}.

\begin{theorem}
\label{thm:hf-imp-cvx}
Assume that \(f\) satisfies \emph{\convexassump{}}.
Suppose \(x_{K}\) is the output of running \emph{\nameref{alg:disc_pphfopt}} for \(K\) iterations with parameter \(\lambda = \frac{\sqrt{3}+1}{2}\), any step size \(\eta>0\), and discretization steps sequence \(\{N_{k}\}_{k=1}^{K}\) satisfying \(N_{k}(N_{k} + 1) \geq \frac{3}{\sqrt{3}+1}N_{k - 1}(N_{k - 1} + 1)\) for all \(k \in [K]\) with \(N_{0} = 4\).
Then,
\begin{equation*}
    f(x_{K}) - f(x^{\star}) \leq \lp \frac{\sqrt{3}+1}{3}\rp^{K}\left(f(x_{0}) - f(x^{\star}) + \frac{\sqrt{3}-1}{60\eta^{2}}\|x_{0} - x^{\star}\|^{2}\right)~.
\end{equation*}
\end{theorem}

\begin{corollary}
\label{cor:hf-imp-cvx}
Assume that \(f\) satisfies \emph{\convexassump{}}.
To obtain \(x_{K}\) that is a \(\varepsilon\)-accurate minimizer, it suffices to run \nameref{alg:disc_pphfopt} with parameter \(\lambda = \frac{\sqrt{3}+1}{2}\) and any step size $\eta > 0$ for \(K\) iterations and discretization step sequence \(\{N_{k}\}_{k=1}^{K}\) such that
\begin{align*}
    K &= \left\lceil \left(\log\left(\frac{3}{1+\sqrt{3}}\right)\right)^{-1} \cdot \log \left(\frac{f(x_{0}) - f(x^{\star}) + \frac{\sqrt{3}-1}{60\eta^{2}}\|x_{0} - x^{\star}\|^{2}}{\varepsilon}\right) \right\rceil~, \\
    N_k &=\left\lceil \sqrt{\frac{3}{1+\sqrt{3}}}N_{k-1}+\tfrac{1}{2} \right\rceil; \quad N_0=4.
\end{align*}
The total number of numerical integration steps satisfies $$N_{\mathrm{tot}}\leq 810\sqrt{\frac{f(x_{0}) - f(x^{\star}) + \frac{\sqrt{3}-1}{60\eta^{2}}\|x_{0} - x^{\star}\|^{2}}{\varepsilon}}.$$
\end{corollary}

\subsubsection{Proof of \texorpdfstring{\Cref{lem:hf-imp-convex-main}}{Lemma 16}}\label{subsub_apdx:one_step_HFimp}

Similar to the proof of \Cref{lem:hf-extg-convex-main}, we first begin with following helper lemma whose proof is provided in \Cref{subsubapdx:helper_lemma_proof}.

\begin{lemma}
\label{lem:hf-imp-convex-hnterm}
    Assume that \(f\) satisfies \emph{\convexassump{}}. Let $\{(x_n, y_n)\}_{n\geq1}$ be generated by the implicit integrator \emph{(\ref{eq:implicit-integ})} from any initial $x_0\in \R^d$ and $y_0 =\bm{0}$. Then, for any $N\geq 1$, the following holds:
    $$
    \langle x_N-z, y_N \rangle \leq 2\eta N (f(x_0)-f(z))-3\eta \sum_{n=1}^N(f(x_n)-f(z)).
    $$
\end{lemma}

\begin{proofof}{\Cref{lem:hf-imp-convex-main}}
We adopt the strategy to prove \Cref{lem:hf-extg-convex-main} and consider the difference \(\Delta_{n} := \|x_{n + 1} - z\|^{2} - \|x_{n} - z\|^{2}\).
By definition of \ref{eq:implicit-integ}, \(x_{n + 1}, y_{n + 1}\) from \(x_{n}, y_{n}\)
\begin{align*}
\Delta_{n} - 2\eta h_{n} &= \|x_{n + 1} - z\|^{2} - \|x_{n} - z\|^{2} - 2\eta \langle x_{n} - z, y_{n}\rangle \\
&= \|x_{n + 1} - x_{n}\|^{2} + 2\langle x_{n + 1} - x_{n}, x_{n} - z\rangle - 2\eta \langle x_{n} - z, y_{n}\rangle \\
&= \eta^{2}\|y_{n + 1}\|^{2} + 2\eta \langle y_{n + 1} - y_{n}, x_{n} - z\rangle \\
&= \eta^{2}\|y_{n + 1}\|^{2} - 2\eta^{2}\langle \nabla f(x_{n + 1}), x_{n} - z\rangle \\
&=\eta^2\|y_{n+1}\|^2-2\eta^2\langle \nabla f(x_{n+1}), x_{n+1}-z\rangle-2\eta^2\langle \nabla f(x_{n+1}), x_{n}-x_{n+1} \rangle \\
    &=\eta^2\|y_{n+1}\|^2-2\eta^2\langle \nabla f(x_{n+1}), x_{n+1}-z\rangle+2\eta^3\langle \nabla f(x_{n+1}), y_{n+1} \rangle \\
    &=\eta^2\|y_{n+1}+\eta \nabla f(x_{n+1})\|^2-2\eta^2\langle \nabla f(x_{n+1}), x_{n+1}-z\rangle-\eta^4 \|\nabla f(x_{n+1})\|^2 \\
    &=\eta^2 \|y_n\|^2-2\eta^2\langle \nabla f(x_{n+1}), x_{n+1}-z\rangle-\eta^4 \|\nabla f(x_{n+1})\|^2 \\
    &\leq2\eta^2(f(x_{0})-f(x_n))+2\eta^2 (f(z)-f(x_{n+1}))\numberthis\label{eq:deltan-bound-imp}.
\end{align*}
The final step uses the convexity of \(f\).
Note that \cref{eq:deltan-bound-imp} takes the same form as \cref{eq:deltan-bound} shown for the extragradient integrator.
Since the implicit integrator also satisfies a bound on \(h_{n}\) from \Cref{lem:hf-imp-convex-hnterm} that matches the bound for the \(h_{n}\) term for the extragradient integrator in \Cref{lem:hf-extg-convex-hnterm}, and the non increasing property of the Hamiltonian (\Cref{lem:hf-imp-nonincrease-ham}), the rest of the proof follows according to the proof of \Cref{lem:hf-extg-convex-main}.
\end{proofof}

\subsubsection{Proof of \texorpdfstring{\Cref{lem:hf-imp-convex-mix-main}}{Lemma 17}}\label{subsub_apdx:multi_step_HFimp}
\begin{proof}
Since \cref{eq:deltan-bound-imp} for implicit integrator is identical with \Cref{eq:deltan-bound} for extragradient integrator and \Cref{lem:hf-imp-convex-hnterm} for implicit integrator is identical with \Cref{lem:hf-extg-convex-hnterm} for extragradient integrator, the same derivation in the proof of \Cref{lem:hf-extg-convex-mix-main} shows the following for iterations of implicit integrator, whenever $N\geq 4$:
\begin{equation}
\label{eq:cor:hf-imp-convex-main-2}
    \frac{1}{2}\|x^{\mathrm{avg}}(x_{0}; N) - \xstar\|^{2} \leq \frac{1}{2}\|x_{0} - \xstar\|^{2} + \frac{\eta^{2}N(N + 1)}{4}(f(x_{0}) - f(\xstar))~.
\end{equation}
Therefore, like in the proof of \Cref{lem:hf-convex-mix-main} and \Cref{lem:hf-extg-convex-mix-main}, we consider a weighted sum of the following inequalities from the equations: \cref{eqn:oracle_step_c_imp,eq:cor:hf-imp-convex-main-2} and the non-increasing property of the implicit integrator (\Cref{lem:hf-imp-nonincrease-ham}).
\begin{align*}
f(x^{\mathrm{avg}}(x_{0}; N )) - f(z) + \frac{\|x_{N} - z\|^{2}}{3\eta^{2}N(N+1)} 
    &\leq \frac{2}{3}(f(x_{0}) - f(z)) + \frac{\|x_{0} - z\|^{2}}{3\eta^{2}N(N + 1)}~ \\
    \frac{1}{3\eta^{2}N(N + 1)}\|x^{\mathrm{avg}}(x_{0}; N) - x^{\star}\|^{2} &\leq \frac{1}{3\eta^{2}N(N + 1)}\|x_{0} - x^{\star}\|^{2} + \frac{1}{6}(f(x_{0}) - f(x^{\star})) \\
    f(x_{N}) - f(x^{\star}) &\leq f(x_{0}) - f(x^{\star})
\end{align*}
with weights \(\lambda, 1, \lambda^{2}\) respectively. Identical simplification as in the proof of \Cref{lem:hf-extg-convex-mix-main} then gives:
\begin{multline*}
    f(x^{\mathrm{mix}}(x_{0}; N)) - f(\xstar) + \frac{1}{3\lambda\eta^{2}N(N + 1)}\|x^{\mathrm{mix}}(x_{0}; N) - \xstar\|^{2} \\
    \leq c_{\lambda}(f(x_{0}) - f(\xstar)) + \frac{1}{3\lambda\eta^{2}N(N + 1)}\|x_{0} - \xstar\|^{2}
\end{multline*}
which is precisely the statement of the lemma.
\end{proof}

\subsubsection{Proof of \texorpdfstring{\Cref{lem:hf-imp-convex-hnterm}}{Lemma 18}}\label{subsubapdx:helper_lemma_proof}

\begin{proof}
Denote \(h_{n} = \langle x_{n} - z, y_{n}\rangle\).
We first compute the difference
\begin{align*}
    h_{n + 1} - h_{n} &= \langle x_{n + 1} - z, y_{n + 1}\rangle - \langle x_{n} - z, y_{n}\rangle \\
    &= \langle x_{n + 1} - z, y_{n + 1} - y_{n}\rangle + \langle x_{n + 1} - z, y_{n}\rangle - \langle x_{n} - z, y_{n}\rangle \\
    &= -\eta \langle x_{n + 1} - z, \nabla f(x_{n + 1})\rangle + \langle x_{n + 1} - x_{n}, y_{n}\rangle~.
\end{align*}
For the first term on the left hand side, we use convexity of \(f\):
\begin{equation*}
    -\eta \langle x_{n + 1} - z, \nabla f(x_{n + 1})\rangle \leq \eta (f(z) - f(x_{n + 1}))~.
\end{equation*}
For the second term on the left hand side, we use the update rule and convexity of $f$:
\begin{align*}
    \langle x_{n+1}-x_n, y_n\rangle &= \eta\langle y_{n+1}, y_n\rangle \\
    &=\eta\langle y_{n}-\eta \nabla f(x_{n+1}), y_n\rangle \\
    &=\eta||y_n||^2-\eta \langle \nabla f(x_{n+1}), \eta y_n\rangle\\
    &\overset{(a)}{=}\eta||y_n||^2-\eta \langle \nabla f(x_{n+1}), x_{n+1}-x_n+\eta^2 \nabla f(x_{n+1})\rangle\\
    &=\eta||y_n||^2-\eta^3||\nabla f(x_{n+1})||^2-\eta \langle x_{n+1}-x_n, \nabla f(x_{n+1})\rangle\\
    &\overset{(b)}{\leq}\eta||y_n||^2+\eta(f(x_n)-f(x_{n+1}))\\
    &\overset{(c)}{\leq}2\eta f(x_0)-2\eta f(x_n)+\eta(f(x_n)-f(x_{n+1}))\\
    &=2\eta f(x_{0}) - \eta f(x_{n}) - \eta f(x_{n + 1}).
\end{align*}
In step $(a)$, we use the property that $x_{n+1}=x_n+\eta y_{n+1}=x_n+\eta (y_n-\eta \nabla f(x_{n+1}))$. In step $(b)$ we used the convexity of $f$. In step $(c)$ we use the fact that along the iterates of implicit integrator, \(f(x_{n}) + \frac{1}{2}\|y_{n}\|^{2} \leq f(x_{0})\).

Combining the bounds, we get
\begin{equation*}
    h_{n + 1} - h_{n} \leq 2\eta (f(x_{0}) - f(z)) - \eta (f(x_{n}) - f(z)) - 2\eta (f(x_{n + 1}) - f(z))~.
\end{equation*}
Summing both sides from \(n =0\) to \(n = N - 1\) and noting that \(h_{0} = 0\),
\begin{align*}
    h_{N} &\leq 2\eta N(f(x_{0}) - f(z)) - 3\eta \sum_{n=1}^{N - 1} (f(x_{n}) - f(z)) - 2\eta (f(x_{N}) - f(z)) - \eta (f(x_{0}) - f(z)) \\
    &\leq 2\eta N(f(x_{0}) - f(z)) - 3\eta \sum_{n=1}^{N}(f(x_{n}) - f(z))
\end{align*}
where in the final inequality we use the fact that \(f(x_{N}) \leq f(x_{0})\).
\end{proof}
\section{Numerical Experiments}
\label{sec:num-exp}

In this section, we perform numerical experiments\footnote[2]{Code for this can be found at: \url{https://github.com/vishwakftw/dHFA}.} to corroborate our theoretical findings in the preceding sections.
We specifically assess the performance on two classical tasks: linear regression and logistic regression.
We use \((A, b)\) to be denote the dataset where \(A\) is a \(\bbR^{n \times d}\) matrix representing \(n\) covariates of dimension \(d\) each, and \(b \in \calY^{n}\) is the vector of associated responses.
For linear regression \(\calY = \bbR\), and for logistic regression \(\calY \in \{-1, +1\}\).
The objectives for linear and logistic regression are
\begin{align*}
    f_{\text{linear}}(x) &:= \frac{1}{2n}\|Ax - b\|^{2} + \frac{\beta}{2}\|x\|^{2}~, \\
    f_{\text{logistic}}(x) &:= \frac{1}{n}\sum_{i=1}^{n} \log(1 + \exp(-b_{i} A_{i}^{\top}x)) + \frac{\beta}{2}\|x\|^{2}~,
\end{align*}
respectively.
In both settings, \(\beta > 0\) is the regularization parameter.

\paragraph{Linear Regression}
We set \(n = 100\) and \(d = 200\) (an underdetermined system) where \(\lambda_{\min}(A^{\top}A) = 0\).
The matrix \(A\) is constructed by sampling each entry uniformly in \(\left[-\nicefrac{2}{\sqrt{d}}, \nicefrac{2}{\sqrt{d}}\right]\), and setting \(b\) as \(Ax^{\star} + \xi\) where \(x^{\star} = \bm{1}_{d}\) and \(\xi\) is a \(d\)-dimensional standard normal vector.
The Hessian is computable in closed form as \(\frac{1}{n} A^{\top}A + \beta \cdot \rmI_{d}\), and hence the strong convexity and smoothness parameters are \(\alpha = \beta + \frac{1}{n}\lambda_{\min}(A^{\top}A) = \beta\) and \(L = \beta + \frac{1}{n}\lambda_{\max}(A^{\top}A)\) respectively and \(\kappa = \frac{L}{\alpha}\).

\paragraph{Logistic Regression} We set \(n = 500\) and \(d = 100\).
The matrix \(A\) is constructed in the same way as done for linear regression above.
The response is \(b_{i} = \mathrm{sign}(A_{i}^{\top}x^{\star} + \xi)\) where \(x^{\star} = \bm{1}_{d}\) and \(\xi\) is a \(d\)-dimensional standard normal vector.
The smoothness and strongly convexity parameters cannot be computed in closed form for this problem, however the strong convexity parameter is at least \(\beta\) owing to the convexity of the data fidelity and the regularization terms.

\vspace*{-2mm}
\subsection{Comparison with gradient descent and accelerated gradient descent}

\begin{figure}[t]
\begin{minipage}{0.24\linewidth}
\centering
\includegraphics[width=\linewidth]{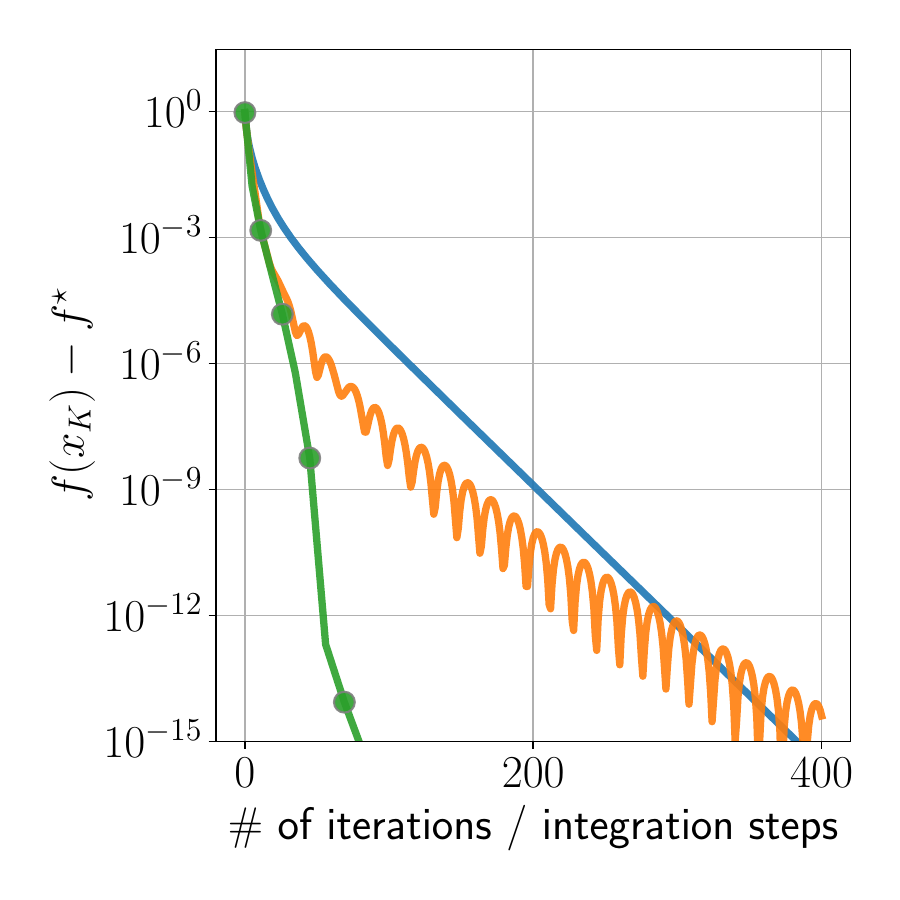} \\
{\footnotesize Linear Regression\\\(\beta = 0\).}
\end{minipage}
\hfill
\begin{minipage}{0.24\linewidth}
\centering
\includegraphics[width=\linewidth]{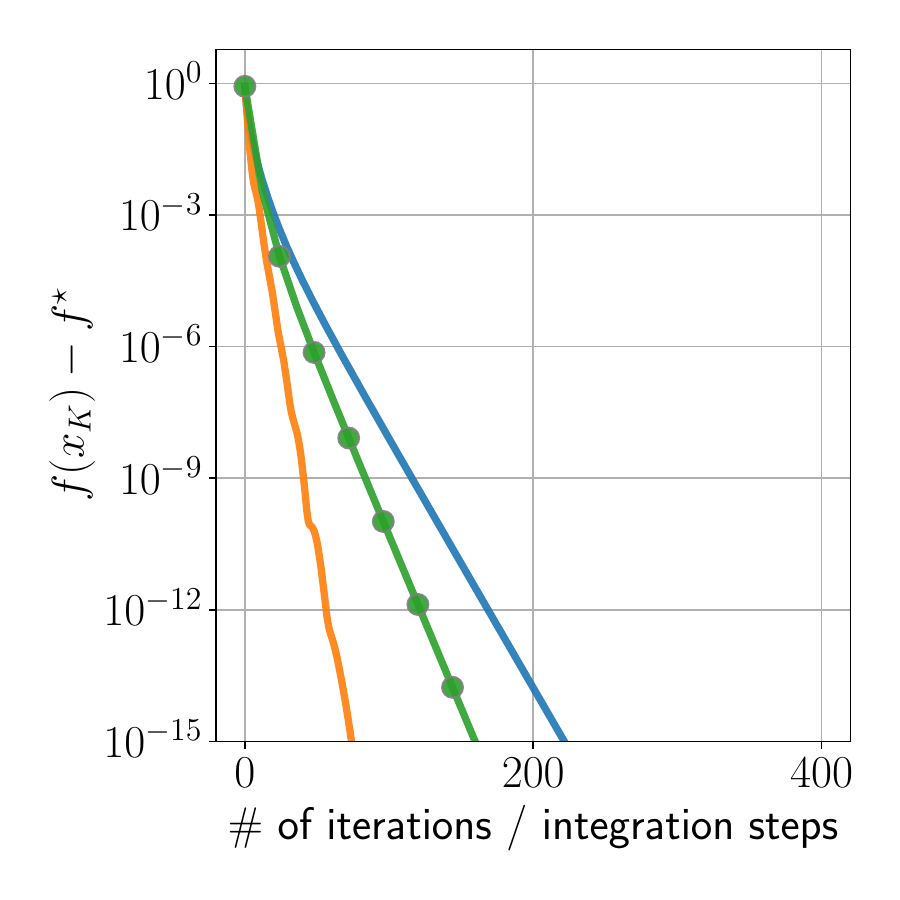} \\
{\footnotesize Linear Regression\\\(\beta = 10^{-3}\).}
\end{minipage}
\hfill
\begin{minipage}{0.24\linewidth}
\centering
\includegraphics[width=\linewidth]{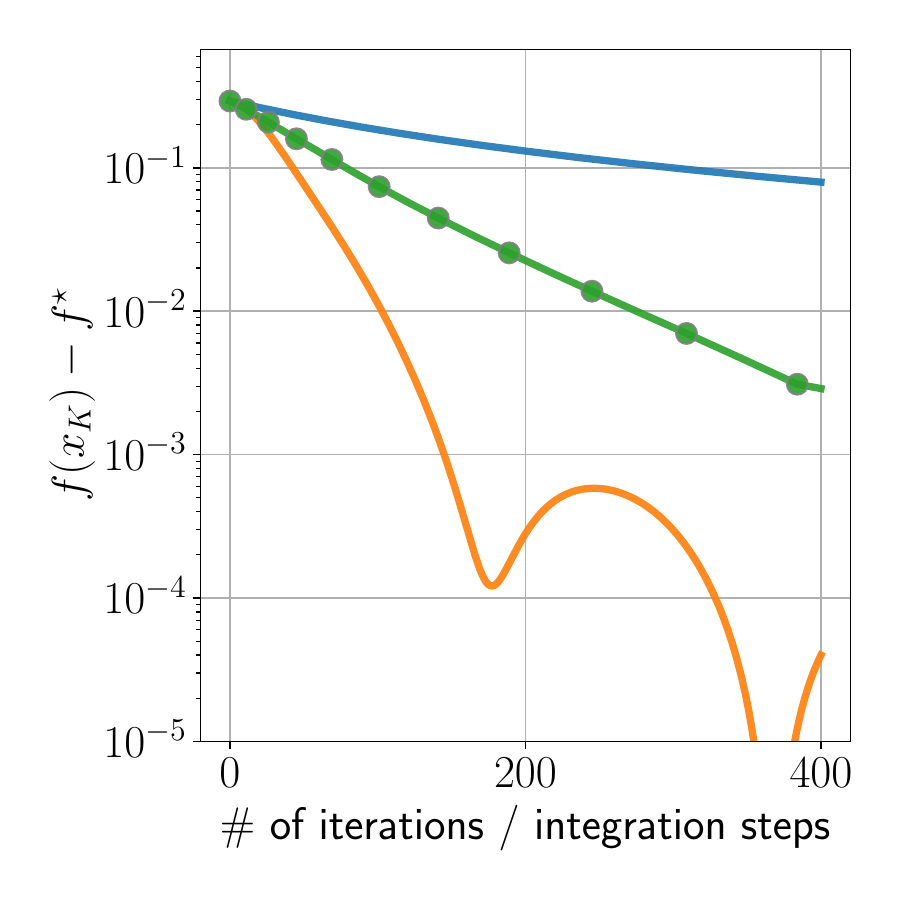} \\
{\footnotesize Logistic Regression\\\(\beta = 0\).}
\end{minipage}
\hfill
\begin{minipage}{0.24\linewidth}
\centering
\includegraphics[width=\linewidth]{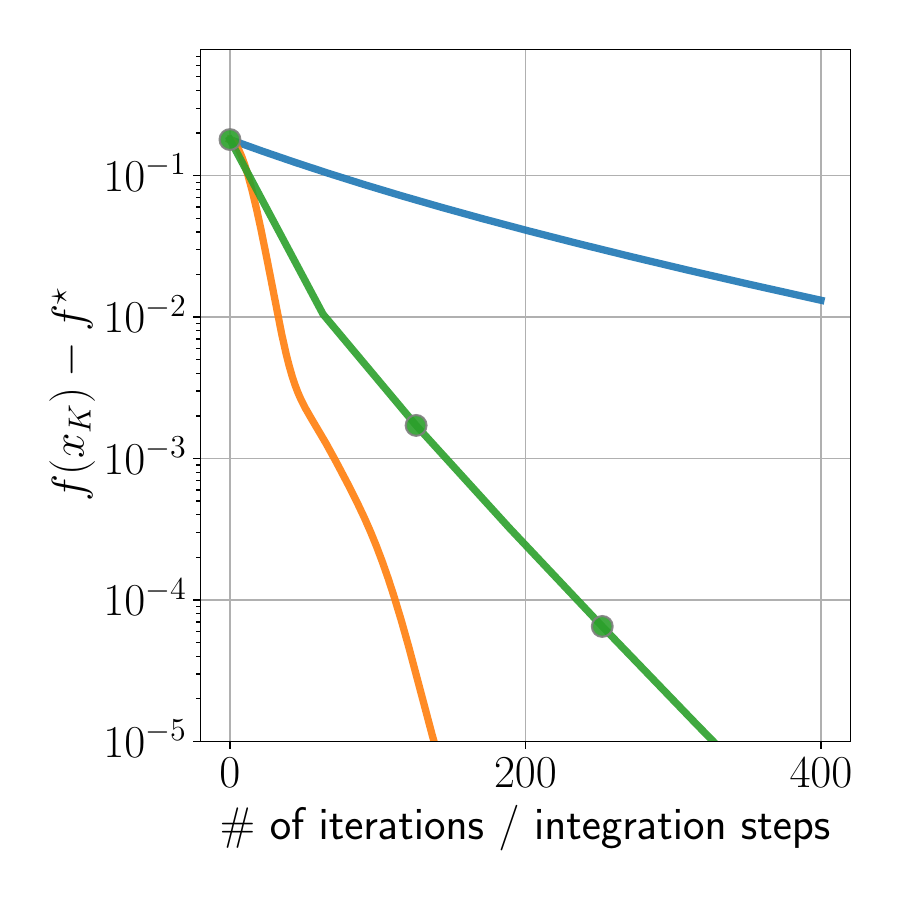} \\
{\footnotesize Logistic Regression\\\(\beta = 10^{-3}\).}
\end{minipage}
\caption{Comparison between \textsf{GD}  (\textcolor{Cerulean}{\rule[2pt]{15pt}{1.5pt}}), \textsf{AGD} (\textcolor{orange}{\rule[2pt]{15pt}{1.5pt}}), and \nameref{alg:disc_phfopt} (\textcolor{ForestGreen}{\rule[2pt]{10pt}{1.5pt}\(\bullet\)\rule[2pt]{10pt}{1.5pt}}).}
\label{fig:main-lin-log-reg}
\end{figure}

For both linear regression and logistic regression, we consider two settings of \(\beta\): \(\beta \in \{0, 10^{-3}\}\).
When \(\beta = 0\), both problems are convex.
For Nesterov's accelerated gradient method (\textsf{AGD}), the momentum parameter \(\gamma_{k}\) at iteration is set to \(\frac{k - 1}{k + 2}\) in this setting.
For \nameref{alg:disc_phfopt}, \(\lambda = \frac{1 + \sqrt{3}}{2}\) and choose \(q = \frac{1}{\sqrt{c_{\lambda}}}\) to determine the sequence \(\{N_{k}\}_{k \geq 0}\) according to \Cref{cor:hf-extg-cvx}.
On the other hand, when \(\beta = 10^{-3}\), both problems are strongly convex.
Here, for \textsf{AGD}, the momentum parameter \(\gamma_{k}\) is set to \(\frac{1 - \sqrt{\eta\alpha}}{1 + \sqrt{\eta\alpha}}\) and for \nameref{alg:disc_phfopt}, \(\lambda = 0\) and \(N_{k} = \lceil \frac{2}{\eta\sqrt{\alpha}}\rceil\) for all iterations \(k\) according to \Cref{cor:hf-extg-strong-cvx}.
The step size \(\eta\) to set to \(\frac{1}{L}\) for \textsf{GD} and \textsf{AGD}, and to \(\frac{1}{\sqrt{L}}\) for \nameref{alg:disc_phfopt} for the linear regression problem.
For the logistic regression problem, we performed grid search on a collection of 25 uniformly log-spaced values from \(10^{-5}\) to \(1\) to determine the optimal step size $\eta$, since since we do not have exact access to the smoothness parameter \(L\).

In \cref{fig:main-lin-log-reg}, we plot these curves for each of the four settings.
In linear regression with \(\beta = 0\)\,, we observe that \nameref{alg:disc_phfopt} is clearly better than the rest, while in the other settings, \textsf{AGD} converges faster.
This is likely due to a suboptimal setting of parameters, which we note can be further tuned.
Regardless, we see clear acceleration by \nameref{alg:disc_phfopt} relative to \textsf{GD}.
Interestingly, in the logistic regression experiments, the best-performing step size on our search grid was \(\eta=1\) for all algorithms, which was the largest value considered.

\vspace*{-2mm}
\subsection{Alternative averaging for iterates}

The weighted average of the iterates \(x^{\mathrm{avg}}\) used in \nameref{alg:disc_phfopt} (Line 3) is motivated by the continuous-time average \(X^{\mathrm{avg}}\).
Here, we investigate the effect of replacing this with the simple average of the iterates generated by extragradient integrator \(x^{\mathrm{s-avg}}(x_{0}; N) = \frac{1}{N}\sum_{n=1}^{N} x_{n}\).
With the same setting of step size \(\eta\) for the empirical setups described previously, we run \nameref{alg:disc_phfopt} with \(x^{\mathrm{s-avg}}\) in lieu of \(x^{\mathrm{avg}}\), and compare it with the original \nameref{alg:disc_phfopt}; the results are plotted in \cref{fig:avg-comps}.
From these, interestingly we see an advantage of using the simple average in all settings except for linear regression with \(\beta = 0\), which is a simple convex setting.
We also observe that when \(\beta = 10^{-3}\) in both the linear and logistic regression setups, the performance closely matches that of \textsf{AGD}.

\begin{figure}[t]
\begin{minipage}{0.24\linewidth}
\centering
\includegraphics[width=\linewidth]{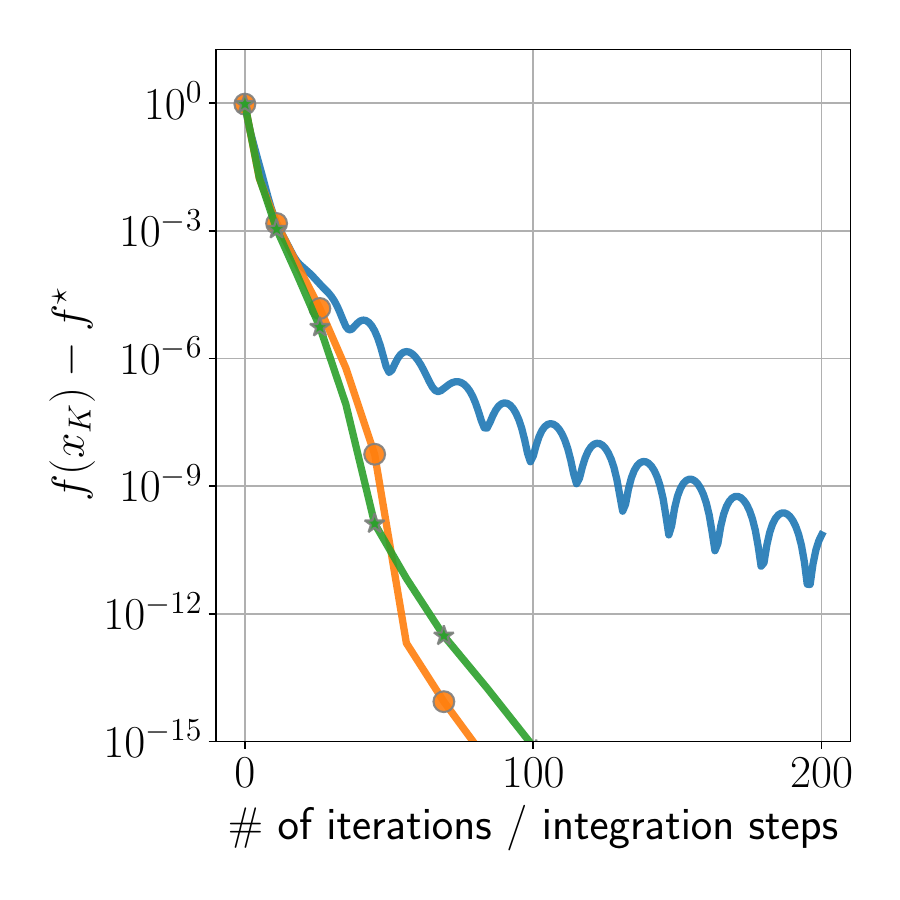} \\
{\footnotesize Linear Regression\\\(\beta = 0\).}
\end{minipage}
\hfill
\begin{minipage}{0.24\linewidth}
\centering
\includegraphics[width=\linewidth]{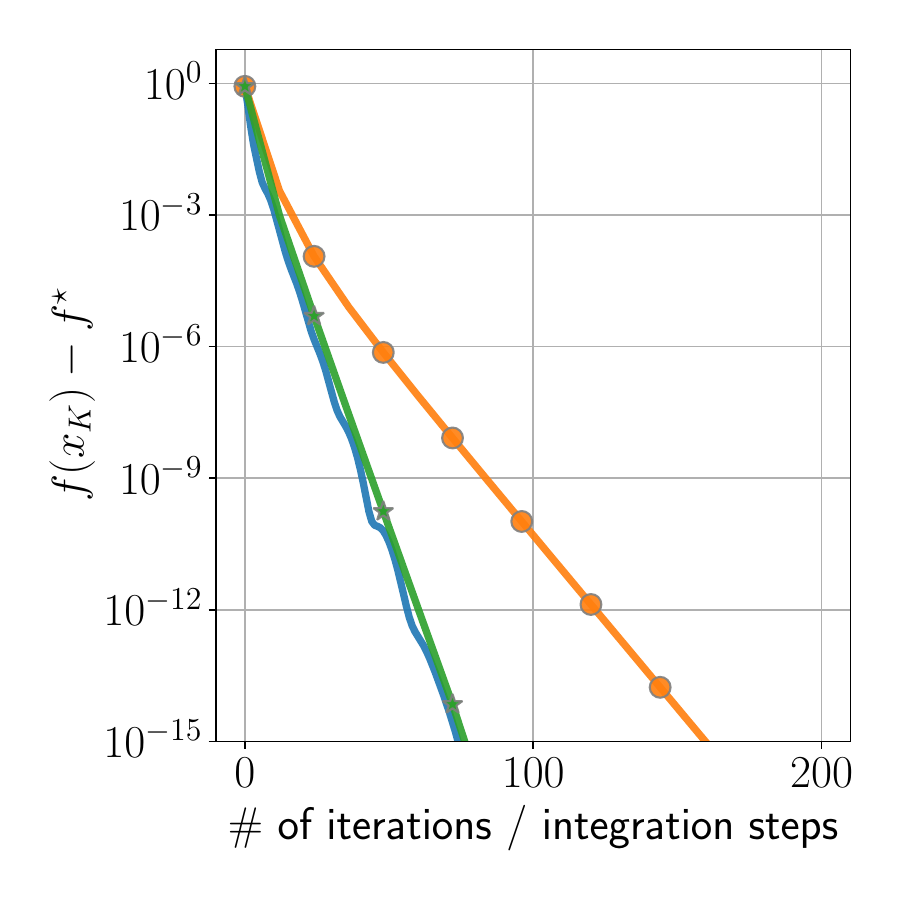} \\
{\footnotesize Linear Regression\\\(\beta = 10^{-3}\).}
\end{minipage}
\hfill
\begin{minipage}{0.24\linewidth}
\centering
\includegraphics[width=\linewidth]{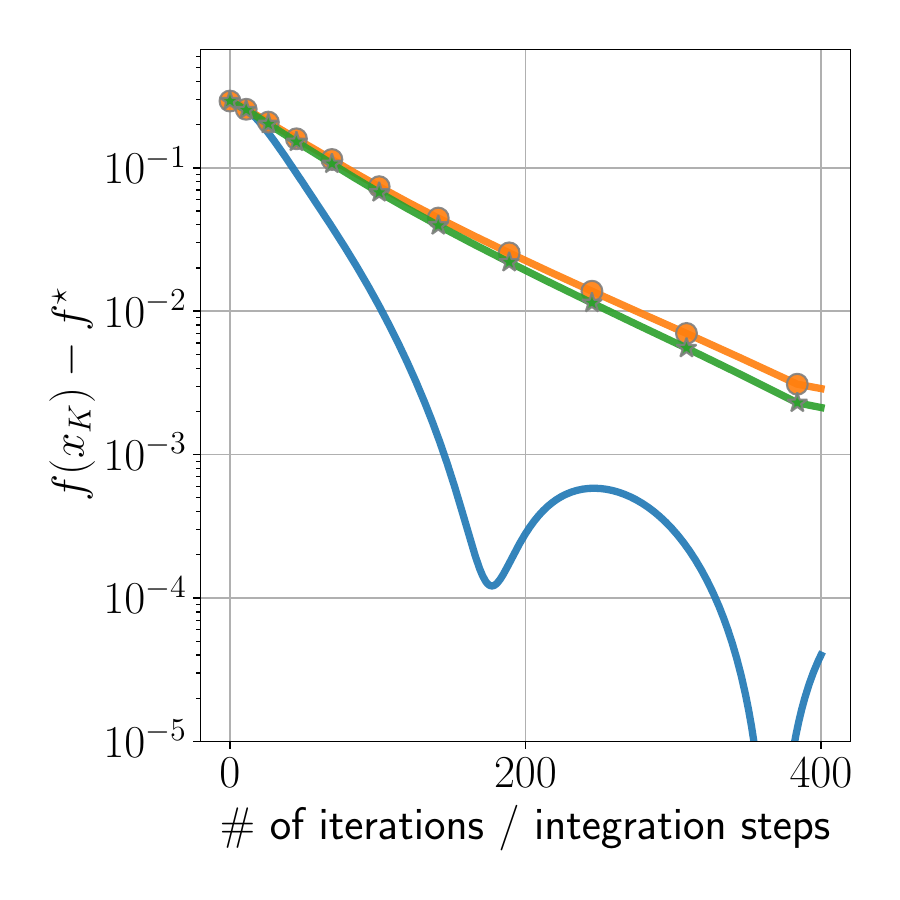} \\
{\footnotesize Logistic Regression\\\(\beta = 0\).}
\end{minipage}
\hfill
\begin{minipage}{0.24\linewidth}
\centering
\includegraphics[width=\linewidth]{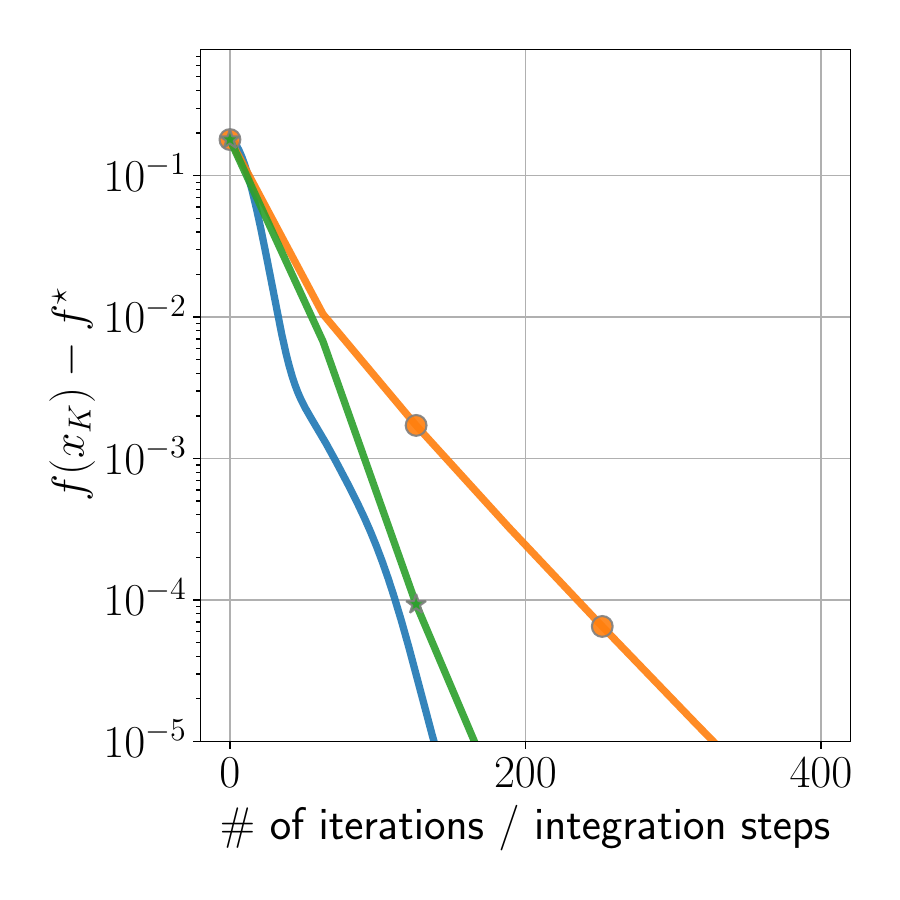} \\
{\footnotesize Logistic Regression\\\(\beta = 10^{-3}\).}
\end{minipage}
\caption{Comparison between averaging schemes.  The plots depict three curves per setting: \nameref{alg:disc_phfopt} with \(x^{\mathrm{avg}}\) (\textcolor{orange}{\rule[2pt]{10pt}{1.5pt}\(\bullet\)\rule[2pt]{10pt}{1.5pt}}), \nameref{alg:disc_phfopt} with \(x^{\mathrm{s-avg}}\) (\textcolor{ForestGreen}{\rule[2pt]{10pt}{1.5pt}\(\star\)\rule[2pt]{10pt}{1.5pt}}), and \textsf{AGD} (\textcolor{Cerulean}{\rule[2pt]{15pt}{1.5pt}}).}
\label{fig:avg-comps}
\end{figure}

\vspace*{-2mm}
\subsection{Varying the integrator in \nameref{alg:disc_phfopt}}

\nameref{alg:disc_phfopt} is a practically viable version of \nameref{alg:phfopt} which is made possible by using the extragradient integrator (\ref{eqn:extg-update}) for the Hamiltonian dynamics.
Coincidentally, this integrator also satisfies desirable properties, which lead to the theoretical guarantees of \nameref{alg:disc_phfopt} discussed in the previous section.
In this subsection, we investigate how different integrators perform empirically in comparison to the extragradient integrator in \nameref{alg:disc_phfopt}.
We specifically focus on the explicit and leapfrog integrators (see \Cref{app:sec:integrator-discuss} for definitions) as these only require access to the gradients of \(f\).
The empirical results are plotted in \cref{fig:int-comps}.

\begin{figure}[t]
\begin{minipage}{0.24\linewidth}
\centering
\includegraphics[width=\linewidth]{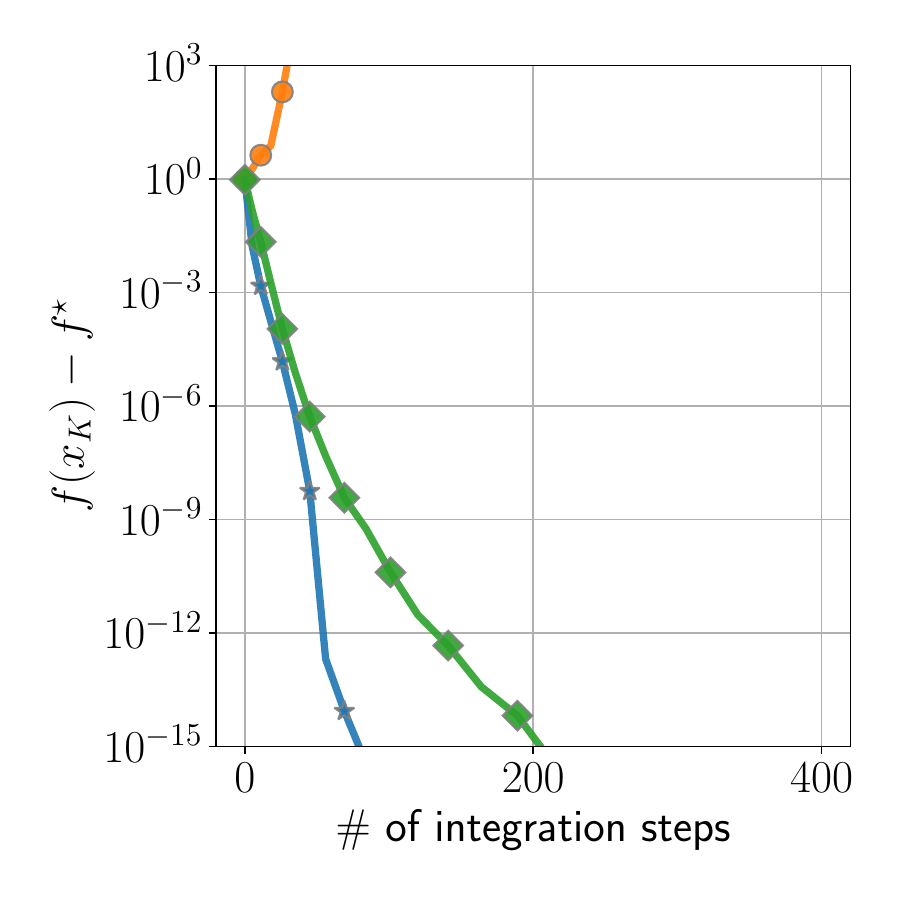} \\
{\footnotesize Linear Regression\\\(\beta = 0\).}
\end{minipage}
\hfill
\begin{minipage}{0.24\linewidth}
\centering
\includegraphics[width=\linewidth]{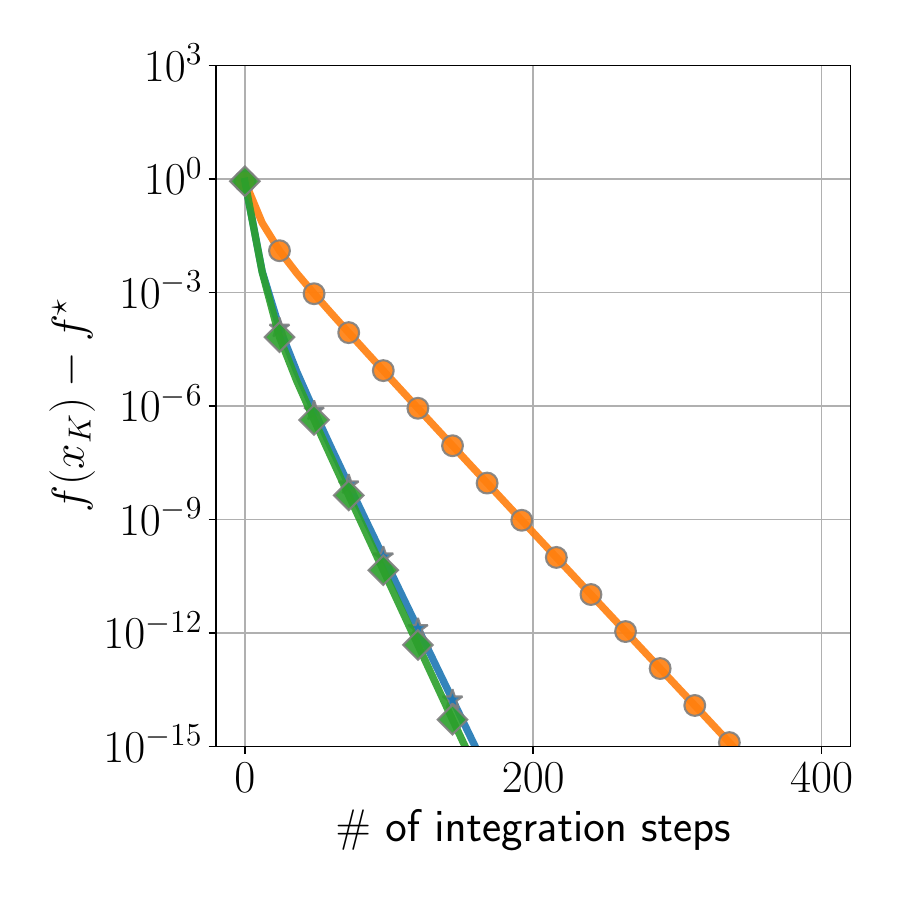} \\
{\footnotesize Linear Regression\\\(\beta = 10^{-3}\).}
\end{minipage}
\hfill
\begin{minipage}{0.24\linewidth}
\centering
\includegraphics[width=\linewidth]{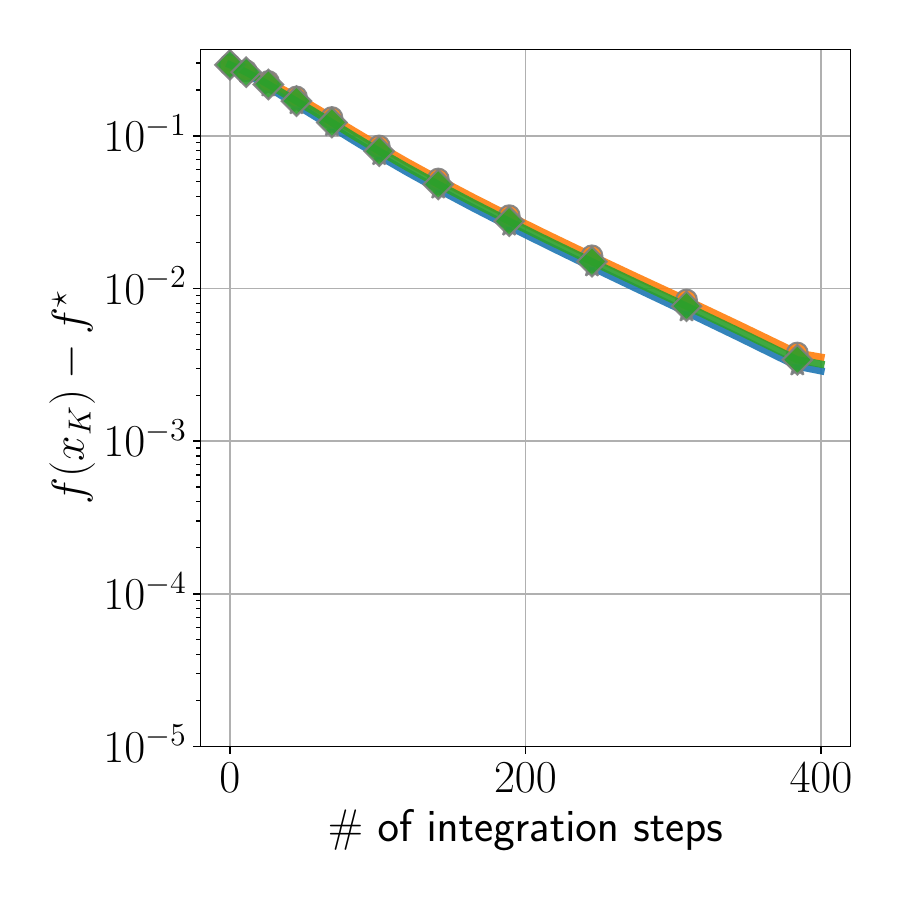} \\
{\footnotesize Logistic Regression\\\(\beta = 0\).}
\end{minipage}
\hfill
\begin{minipage}{0.24\linewidth}
\centering
\includegraphics[width=\linewidth]{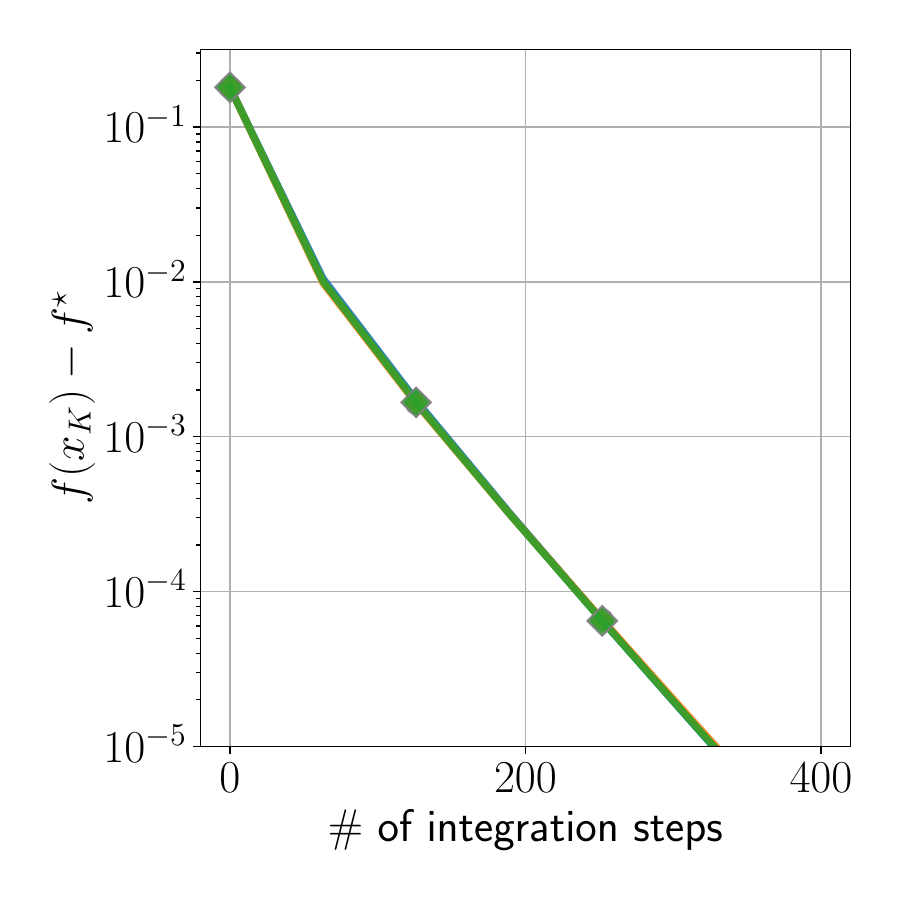} \\
{\footnotesize Logistic Regression\\\(\beta = 10^{-3}\).}
\end{minipage}
\caption{Comparison between integrators. 
The plots depict three curves per setting: 
\textsf{dHFA-ex} (\textcolor{orange}{\rule[2pt]{10pt}{1.5pt}\(\bullet\)\rule[2pt]{10pt}{1.5pt}}), \textsf{dHFA-lf} (\textcolor{ForestGreen}{\rule[2pt]{10pt}{1.5pt}\(\blacklozenge\)\rule[2pt]{10pt}{1.5pt}}), \nameref{alg:disc_phfopt} (\textcolor{Cerulean}{\rule[2pt]{10pt}{1.5pt}\(\star\)\rule[2pt]{10pt}{1.5pt}}).
\textsf{dHFA-ex} and \textsf{dHFA-lf} are \nameref{alg:disc_phfopt} but with the extragradient integrator replaced by the explicit and leapfrog integrators respectively.}
\label{fig:int-comps}
\end{figure}

We use the same empirical setups as described previously.
The step size for these methods are set to the same value; for linear regression, this is set as \(\eta = \frac{1}{\sqrt{L}}\), whereas for logistic regression, this is set as \(\eta = 1\).
From the plots on the right, we see that the extragradient integrator is the best performing integrator in all settings.
Interestingly, for the linear regression setup with \(\beta = 0\), the explicit integrator diverges rapidly, while for the logistic regression problem, the performance matches closely with the others.
\end{document}